\documentclass[aos,preprint]{imsart}

\RequirePackage[OT1]{fontenc}
\RequirePackage{amsthm,amsmath}
\RequirePackage{natbib}
\RequirePackage[colorlinks,citecolor=blue,urlcolor=blue]{hyperref}

\usepackage{amssymb}

\usepackage{multirow}
\usepackage{tabularx}

\usepackage{amssymb}
\usepackage{graphicx}
\usepackage{caption}
\usepackage[labelformat=simple, margin=20pt]{subcaption}

\usepackage{booktabs}

\newtheorem{example}{Example}
\newtheorem*{theorem*}{Theorem}
\newtheorem{theorem}{Theorem}
\newtheorem{condition}{Condition}
\newtheorem{proposition}{Proposition}
\newtheorem{definition}{Definition}
\newtheorem{corollary}{Corollary}
\newtheorem*{corollary*}{Corollary}
\newtheorem{lemma}{Lemma}

\def\apprsim{\overset{.}{\sim}}
\def\bm{\boldsymbol}
\def\coef{\bm{b}}
\def\obs{\textup{obs}}

\def\Cov{\textup{Cov}}
\def\orthocoef{\bm{c}}
\def\covcoef{\tilde{\bm{B}}}
\def\E{\mathbb{E}}

\newcommand{\myparallel}{{\mkern3mu\vphantom{\perp}\vrule depth 0pt\mkern2mu\vrule depth 0pt\mkern3mu}}
\def\myperp{\perp}

\def\Tau{{T}}
\def\Var{\textup{Var}}

\RequirePackage[normalem]{ulem}

\startlocaldefs
\numberwithin{equation}{section}
\theoremstyle{plain}

\endlocaldefs

\begin{document}

\begin{frontmatter}
\title{Rerandomization in $\boldsymbol{2^K}$ Factorial Experiments
}
\runtitle{Rerandomization in $\boldsymbol{2^K}$ Factorial Experiments}

\begin{aug}
\author{\fnms{Xinran} \snm{Li}\thanksref{t1}\ead[label=e1]{lixinran@wharton.upenn.edu}},
\author{\fnms{Peng} \snm{Ding}\thanksref{t2}\ead[label=e2]{pengdingpku@berkeley.edu}}
\and
\author{\fnms{Donald B.} \snm{Rubin}\thanksref{t3}
\ead[label=e3]{dbrubin@me.com}
}

\runauthor{X. Li, P. Ding and D. B. Rubin}

\affiliation{University of Pennsylvania,\thanksmark{t1} University of California, Berkeley\thanksmark{t2},
Tsinghua University and Temple University\thanksmark{t3}}

\address{Department of Statistics\\
University of Pennsylvania\\
Philadelphia, Pennsylvania 19104\\
USA\\
\printead{e1}
}

\address{Department of Statistics\\
University of California, Berkeley\\
Berkeley, CA 94720\\
USA\\
\printead{e2}
}

\address{Department of Statistical Science\\
	Temple University\\
	Philadelphia, Pennsylvania 19122\\
	USA\\
	\printead{e3}
}
\end{aug}

\begin{abstract}
With many pretreatment covariates and treatment factors, the classical factorial experiment often fails to balance covariates across
multiple factorial effects simultaneously. Therefore, it is intuitive to restrict the randomization of the treatment factors to satisfy certain covariate balance criteria, possibly conforming to the tiers of factorial effects and covariates based on their relative importances. This is rerandomization in factorial experiments. We study the asymptotic properties of this experimental design under the randomization inference framework without imposing any distributional or modeling assumptions of the covariates and outcomes. We derive the joint asymptotic sampling distribution of the usual estimators of the factorial effects, and show that it is symmetric, unimodal, and more ``concentrated'' at the true factorial effects under rerandomization than under the classical factorial experiment. We quantify this advantage of rerandomization using the notions of ``central convex unimodality'' and ``peakedness'' of the joint asymptotic sampling distribution. We also construct conservative large-sample confidence sets for the factorial effects.
\end{abstract}



\end{frontmatter}

\section{Introduction}

Factorial experiments, initially proposed by \citet{Fisher:1935} and \citet{yates1937design}, have been widely used in the agricultural science (see textbooks by \citealt{cochran1950experimental, kempthorne1952design, Hinkelmann2007, cox2000theory}) and engineering (see textbooks by \citealt{box2005statistics, wu2011experiments}). Recently, factorial experiments also become popular in social sciences \citep[e.g.,][]{angrist2009incentives, dasgupta2014causal, branson2016improving}. 
The completely randomized factorial experiment (CRFE) balances covariates under different treatment combinations on average. However, with increasing numbers of pretreatment covariates and treatment factors, there will be an increasing chance to have unbalanced covariates with respect to multiple factorial effects. Many researchers have recognized this issue in different experimental designs  \citep[e.g.,][]{fisher1926, student:1938, hansen2008covariate, Bruhn:2009}. To avoid this, we can force a treatment allocation to have covariate balance, which is sometimes called rerandomization \citep[e.g.,][]{cox:1982, cox:2009, morgan2012rerandomization}, restricted or constrained randomization \citep[e.g.,][]{yates1948comment, grundy1950restricted, youden1972randomization, bailey1983restricted}.

Extending \citet{morgan2012rerandomization}'s proposal for treatment-control experiments, \citet{branson2016improving} proposed to use rerandomization in factorial experiments to improve covariate balance, and studied finite sample properties of this design under the assumptions of equal sample sizes of all treatment combinations, Gaussianity of covariate and outcome means, and additive factorial effects. Without requiring any of these assumptions, 
we propose more general covariate balance criteria for rerandomization in $2^K$ factorial experiments, extend their theory
with an asymptotic analysis of the sampling distributions of the usual factorial effect estimators, and provide large-sample confidence sets for the average factorial effects.

Rerandomization in factorial experiments have two salient features that differ from rerandomization in treatment-control experiments. First, the factorial effects can have different levels of importance a priori. Many factorial experimental design principles hinge on the belief that main effects are often more important than two-way interactions, and two-way interactions are often more important than higher-order interactions \citep[e.g.,][]{finney1943fractional, bose1947mathematical, cochran1950experimental, cox2000theory, wu2015post}. Consequently, we need to impose different stringencies for balancing covariates with respect to factorial effects of different importance. Second, covariates may also vary in importance based on prior knowledge about their associations with the outcome. We establish a general theory that can accommodate rerandomization with tiers of both factorial effects and covariates.

Second, in treatment-control experiments, we are often interested in a single treatment effect. In factorial experiments, however, multiple factorial effects are simultaneously of interest, motivating the asymptotic theory about the joint sampling distribution of the usual factorial effect estimators. In particular, for the joint sampling distribution, we use ``central convex unimodality'' \citep{dharmadhikari1976,kanter1977} to describe its unimodal property, and ``peakedness'' \citep{sherman1955} to quantify the intuition that it is more ``concentrated'' at the true factorial effects under rerandomization than the CRFE. These two mathematical notions for multivariate distributions extend unimodality and narrower quantile ranges for univariate distributions \citep{asymrerand2016}, and they are also crucial for constructing large-sample 
confidence sets for factorial effects.

In sum, our asymptotic analysis further demonstrates the benefits of rerandomization in factorial experiments compared to the classical CRFE \citep{branson2016improving}. The proposed large-sample confidence sets for factorial effects will facilitate the practical use of rerandomization in factorial experiments and the associated repeated sampling inference.

The paper proceeds as follows. 
Section \ref{sec:notation} introduces the notation.
Section \ref{sec:CRFE} discusses sampling properties and linear projections under the CRFE.
Section \ref{sec:ReFM} studies rerandomization using the Mahalanobis distance
criterion.
Section \ref{sec:ReFTF} studies rerandomization with tiers of factorial effects. 
Section \ref{sec:edu_example} contains an application to an education dataset.
Section \ref{sec:extension} concludes with possible extensions.
The online Supplementary Material \citep{rerandfacsupp} contains all technical details.
We do not make any attempt to review the extensive literature of confounding and fractional replication in the context of factorial experiments. Instead, we focus on the repeated sampling properties of estimators under rerandomization in $2^K$ factorial experiments.

\section{Notation for a $\boldsymbol{2^K}$ factorial experiment}\label{sec:notation}

\subsection{Potential outcomes and causal estimands}
Consider a factorial experiment with $n$ units and $K$ treatment factors, where each factor has two levels, $-1$ and $+1$. In total there are $Q=2^K$ treatment combinations, and for each treatment combination $1\leq q\leq Q$, let $\bm{\iota}(q)=(\iota_1(q), \iota_2(q),\ldots, \iota_K(q))\in \{-1, +1\} ^K$ be the levels of the $K$ factors. 
We use potential outcomes to define causal effects in factorial experiments \citep{Neyman:1923, dasgupta2014causal, branson2016improving}. For unit $i$, 
let
$Y_i(q)$ be the potential outcome under treatment combination $q$, 
and 
$\bm{Y}_i=\left(
Y_i(1), Y_i(2), \ldots, Y_i(Q)
\right)$ 
be the $Q$ dimensional row vector of all potential outcomes. 
Let $\bar{Y}(q)=\sum_{i=1}^n Y_i(q)/n$ be the average potential outcome under treatment combination $q$, 
and $\bar{\bm{Y}} = (\bar{Y}(1), \bar{Y}(2), \ldots, \bar{Y}(Q))$ be the $Q$ dimensional row vector of all average potential outcomes. 
\citet{dasgupta2014causal} characterized each factorial effect by a $Q$ dimensional column vector with half of its elements being $-1$ and the other half being $+1$. For example, the average main effect of factor $k$ is
\begin{align*}
\tau_k & = \frac{2}{Q}\sum_{q=1}^{Q} 1\{\iota_{k}(q)=1\} \bar{Y}(q) - \frac{2}{Q}\sum_{q=1}^{Q} 1\{\iota_{k}(q)=-1\} \bar{Y}(q) 
\\ & = \frac{1}{2^{K-1}}\sum_{q=1}^{Q} \iota_{k}(q)\bar{Y}(q) 
= \frac{1}{2^{K-1}}\bar{\bm{Y}}\bm{g}_k, \quad (1\leq k\leq K)
\end{align*}
where $\bm{g}_k=(g_{k1}, \ldots, g_{kQ})' = (\iota_{k}(1), \iota_{k}(2)\ldots, \iota_{k}(Q))'$ is called the \textit{generating vector} for the main effect of factor $k$. 
For an interaction effect among several factors, the $\bm{g}$-vector is an element-wise multiplication of the $\bm{g}$-vectors for the main effects of the corresponding factors.
There are in total $F=2^K-1=Q-1$ factorial effects. 
Let $\bm{g}_f=(g_{f1}, \ldots, g_{fQ})' \in \{-1, +1\}^Q $ be the generating vector for the $f$th factorial effect ($1\leq f\leq F$). 
For unit $i$, 
$\tau_{if}=2^{-(K-1)}\bm{Y}_i\bm{g}_f
$ is the $f$th individual factorial effect,
and 
$\bm{\tau}_i = (\tau_{i1}, \ldots, \tau_{iF})'$ is the $F$ dimensional column vector of all individual factorial effects.  
Let 
$
\tau_f =  2^{-(K-1)}\bar{\bm{Y}}\bm{g}_f
$
be the $f$th average factorial effect, 
and 
$\bm{\tau} =   (\tau_{1}, \ldots, \tau_{F})'$ be the $F$ dimensional column vector of all average factorial effects. 
The definitions of the factorial effects imply 
$\bm{\tau}_i = 2^{-(K-1)}\sum_{q=1}^{Q}  \coef_q Y_i(q)$ and 
$\bm{\tau} = 2^{-(K-1)}\sum_{q=1}^{Q}  \coef_q \bar{Y}(q),$ 
with coefficient vectors   
\begin{align}\label{eq:coefficient_b}
\coef_1 =
\begin{pmatrix}
g_{11}\\
g_{21}\\
\vdots\\
g_{F1}
\end{pmatrix}, 
\quad 
\coef_2 = 
\begin{pmatrix}
g_{12}\\
g_{22}\\
\vdots\\
g_{F2}
\end{pmatrix},
\quad 
\ldots
\quad 
\coef_{Q} = 
\begin{pmatrix}
g_{1Q}\\
g_{2Q}\\
\vdots\\
g_{FQ}
\end{pmatrix}. 
\end{align}
Intuitively, the $k$th main effect compares the average potential outcomes when factor $k$ is at $+1$ and $-1$ levels, 
and 
the interaction effect among two factors compares the average potential outcomes when both factors are at the same level and different levels. We can view a higher order interaction as the difference between two conditional lower order interactions. For example, the interaction among factors 1--3 equals the difference between the interactions of factors 1 and 2 given factor 3 at $+1$ and $-1$ levels. See \citet{dasgupta2014causal} for more details. Below we use an example to illustrate the definitions.

\begin{example}
	We consider factorial experiments with $K=3$ factors, and use $(1,2,3)$ to denote these three factors. 
	Table \ref{tab:matrix_K3} shows the 
	definitions of the $\bm{g}_f$'s and the $\coef_q$'s. 
	Specifically, the first three columns $(\bm{g}_1, \bm{g}_2,\bm{g}_3)$ represent the levels of three factors in all treatment combinations, and they generate the main effects of factors $(1, 2, 3)$. The remaining columns $(\bm{g}_4, \ldots, \bm{g}_7)$ are the element-wise multiplications of subsets of $(\bm{g}_1, \bm{g}_2, \bm{g}_3)$ that generate the interaction effects. 
	The coefficient vector $\coef_q$ consists of all the elements in the $q$th row of Table \ref{tab:matrix_K3}. $\hfill\Box$
	\begin{table}
		\centering
		\caption{$\bm{g}_f$'s and $\coef_q$'s for $2^3$ factorial experiments}\label{tab:matrix_K3}
		\begin{tabular}{|c|c|c|c|c|c|c|c|}
			\hline
			\ $1$\ \   & \ \ $2$ \    & \ \ $3$\ \    & \ $12$ \    & \ $13$ \ & \ $23$ \ & $123$ & \\[0.5pt]
			\hline
			$-1$ & $-1$ & $-1$ & $+1$ & $+1$ & $+1$ & $-1$ & $\coef_1'$\\
			$-1$ & $-1$ & $+1$ & $+1$ & $-1$ & $-1$ & $+1$ & $\coef_2'$\\
			$-1$ & $+1$ & $-1$ & $-1$ & $+1$ & $-1$ & $+1$ & $\coef_3'$\\
			$-1$ & $+1$ & $+1$ & $-1$ & $-1$ & $+1$ & $-1$ & $\coef_4'$\\
			$+1$ & $-1$ & $-1$ & $-1$ & $-1$ & $+1$ & $+1$ & $\coef_5'$\\
			$+1$ & $-1$ & $+1$ & $-1$ & $+1$ & $-1$ & $-1$ & $\coef_6'$\\
			$+1$ & $+1$ & $-1$ & $+1$ & $-1$ & $-1$ & $-1$ & $\coef_7'$\\
			$+1$ & $+1$ & $+1$ & $+1$ & $+1$ & $+1$ & $+1$ & $\coef_8'$\\
			\hline
			$\bm{g}_1$ & $\bm{g}_2$ & $\bm{g}_3$ & $\bm{g}_4$ & $\bm{g}_5$ & $\bm{g}_6$ & $\bm{g}_7$ & \\
			\hline
		\end{tabular}
	\end{table}
\end{example}

\subsection{Treatment assignment, covariate imbalance and rerandomization}
For each unit $i$, $\bm{x}_i$ represents the $L$ dimensional column vector of pretreatment covariates.
For instance, in the education example in Section \ref{sec:edu_example}, college freshmen receive different academic services and incentives
after entering the university, and their pretreatment covariates include high school GPA, gender, age, and etc. 
Let 
$Z_i$ be the treatment assignment, where $Z_i=q$ if unit $i$ receives treatment combination $q$. 
Let $n_q$ be the number of units under treatment combination $q$, and $\bm{Z}=(Z_1, \ldots, Z_n)$ be the treatment assignment vector for all units.
In the CRFE, the probability that $\bm{Z}$ takes a particular value $\bm{z}=(z_1, \ldots, z_n)$ is $n_1!\cdots n_{Q}!/n!$, where $\sum_{i=1}^n 1\{z_i=q\}=n_q$ for all $q$. 
Let $\bar{\bm{x}} = n^{-1} \sum_{i=1}^n \bm{x}_i$ be the finite population covariate mean vector; 
for $1\leq q\leq Q$, let $\hat{\bar{{\bm{x}}}}(q) = n_q^{-1}\sum_{i:z_i=q}\bm{x}_i$ be the covariate mean vector for units that receive treatment combination $q$. 
For $1\leq f\leq F$, 
the $L$ dimensional difference-in-means vector of covariates with respect to the $f$th factorial effect is 
\begin{align}\label{eq:diff_in_mean_X_f}
\hat{\bm{\tau}}_{\bm{x},f} & = 
\frac{2}{Q} \sum_{q=1}^{Q} g_{fq} \hat{\bar{{\bm{x}}}}(q) = 
\frac{1}{2^{K-1}}
\sum_{q:g_{fq}=1}\hat{\bar{{\bm{x}}}}(q) - 
\frac{1}{2^{K-1}}\sum_{q:g_{fq}=-1}\hat{\bar{{\bm{x}}}}(q).
\end{align}
Let $\hat{\bm{\tau}}_{\bm{x}} = 
(\hat{\bm{\tau}}_{\bm{x},1}', \ldots, \hat{\bm{\tau}}_{\bm{x},F}' )'$ be the $LF$ dimensional column vector of the difference-in-means of covariates with respect to all factorial effects. 
Although $\hat{\bm{\tau}}_{\bm{x}}$ has mean zero under the CRFE, for a realized value of $\bm{Z}$, 
covariate distributions are often imbalanced among 
different treatment combinations.  For example, we consider a CRFE with $K=2$ factors, $L=4$ uncorrelated covariates, and equal treatment group sizes $n_q = n/Q$.  In this case, with asymptotic probability $1-(1-5\%)^{4(2^2-1)} \approx 46.0\%$, at least one of the difference-in-means in \eqref{eq:diff_in_mean_X_f} with respect to a covariate and a factorial effect standardized by its standard deviation is larger than 1.96, the 0.975-quantile of $\mathcal{N}(0,1)$.
This holds due to the asymptotic Gaussianity of $\hat{\bm{\tau}}_{\bm{x}}  $ with zero mean and diagonal covariance matrix, implied by Proposition \ref{prop:mean_var_cre} discussed shortly.

Rerandomization is a design to prevent undesirable treatment allocations. When covariate imbalance occurs for a realized randomization under a certain criterion, we discard this unlucky realization and rerandomize the treatment assignment until this criterion is satisfied. Generally, rerandomization proceeds as follows \citep{morgan2012rerandomization}: first, we collect covariate data and specify a covariate balance criterion; second, we continue randomizing the units into different treatment groups until the balance criterion is satisfied; third, we conduct the physical experiment using the accepted randomization. A major goal of this paper is to discuss the statistical analysis of the data from a rerandomized factorial experiment.

There are three additional issues on covariates. First, covariates are attributes of the units that are fixed before the experiment.
The experimenter may observe some covariates, but s/he can not change their values. The treatments do not affect the covariates. Second, the covariates can be general (discrete or continuous). We can use binary indicators to represent discrete covariates. 
Third, the covariates can include transformations of the basic covariates and their interactions. This enables us to balance the marginal and joint distributions of the basic covariates. See \citet{covadapt2011} for a related discussion in the treatment-control experiment.

\subsection{Notation}
To facilitate the discussion, for a positive semi-definite matrix $\bm{A}\in \mathbb{R}^{m\times m}$ with rank $p_0$, and a positive integer $p\geq p_0$, we use $\bm{A}^{1/2}_p \in \mathbb{R}^{m\times p}$ to denote a matrix such that 
$\bm{A}^{1/2}_p (\bm{A}^{1/2}_p)'= \bm{A}$. Specifically, if $\bm{A}=\bm{\Gamma}\bm{\bm{\Lambda}}^2\bm{\Gamma}'$ is the eigen-decomposition of $\bm{A}$ where $\bm{\Gamma}\in \mathbb{R}^{m\times p_0}$, $\bm{\Gamma}'\bm{\Gamma}=\bm{I}_{p_0}$ and $\bm{\Lambda}=\textup{diag}(\lambda_1, \ldots, \lambda_{p_0})$,
then we can choose 
$\bm{A}^{1/2}_p = (\bm{\Gamma}\bm{\Lambda}, \bm{0}_{m\times(p-p_0)})$. 
The choice of $\bm{A}^{1/2}_p$ is generally not unique. 
In the special case with $p=m$, we use $\bm{A}^{1/2}$ to denote the unique positive-semidefinite matrix satisfying the definition of $\bm{A}_m^{1/2}$. 
We use $\otimes$ for the Kronecker product of two matrices, and $\circ$ for element-wise multiplications of vectors. We say a matrix $\bm{M}_1$ is smaller than or equal to $\bm{M}_2$ and write as $\bm{M}_1\leq \bm{M}_2$, if $\bm{M}_2 - \bm{M}_1$ is positive semi-definite.
We say a random vector $\bm{\phi}$ (or its distribution) is symmetric, if $\bm{\phi} \sim -\bm{\phi}$ have the same distribution. 
We say a random vector is spherically symmetric, if its distribution is invariant under orthogonal transformations.
In the asymptotic analysis, we use $\apprsim$ for two sequences of random vectors 
converging weakly to the same distribution, after scaling by $\sqrt{n}$.

\section{$\boldsymbol{2^K}$  completely randomized factorial experiments}\label{sec:CRFE}
The sampling distributions of factorial effect estimators under rerandomization are the same as their conditional distributions given that the treatment assignment vector satisfies the balance criterion. Therefore, we first study the joint sampling distribution of the difference-in-means of the outcomes and covariates. It depends on the finite population variances and covariances:
$
S_{qq} = (n-1)^{-1}\sum_{i=1}^{n}\{ Y_i(q)-\bar{Y}(q) \}^2
$
and
$
S_{qk} = (n-1)^{-1}\sum_{i=1}^{n}\{ Y_i(q)-\bar{Y}(q) \} \{ Y_i(k)-\bar{Y}(k) \}
$
for potential outcomes, 
$
\bm{S}_{\bm{\tau\tau}} = (n-1)^{-1} \sum_{i=1}^{n} (\bm{\tau}_i-\bm{\tau}) (\bm{\tau}_i-\bm{\tau})' 
$
for factorial effects,
$
\bm{S}_{\bm{xx}} = (n-1)^{-1} \sum_{i=1}^{n} (\bm{x}_i - \bar{\bm{x}}) (\bm{x}_i - \bar{\bm{x}})' 
$
for covariates, and 
$
\bm{S}_{q, \bm{x}} = \bm{S}_{\bm{x},q}' = (n-1)^{-1} \sum_{i=1}^{n} \{Y_i(q) - \bar{Y}(q)\} (\bm{x}_i - \bar{\bm{x}})' 
$ 
for potential outcomes and covariates. 
The covariance $\bm{S}_{\bm{xx}}$ is known without any uncertainty. However, other variances or covariances (e.g., $S_{qk}, \bm{S}_{\tau\tau}$ and $\bm{S}_{q,\bm{x}}$) involve potential outcomes or individual factorial effects and are thus generally unknown.

\subsection{Asymptotic sampling distribution under the CRFE}\label{sec:asym_dist_crfe}

Let $Y_i^\obs  = \sum_{q=1}^{Q}1\{Z_i=q\}Y_i(q)$ be the observed outcome of unit $i$, 
and 
$\hat{\bar{Y}}(q) = n_q^{-1}\sum_{i:Z_i=q}Y_i^\obs$ be the average observed  outcome under treatment combination $q$. 
For $1\leq f\leq F$, 
the difference-in-means estimator for the $f$th average factorial effect is  
$$
\hat{\tau}_f = 
\frac{2}{Q}\sum_{q=1}^{Q} g_{fq} \hat{\bar{Y}}(q) = 
\frac{1}{2^{K-1}}
\sum_{q:g_{fq}=1} \hat{\bar{Y}}(q)
-
\frac{1}{2^{K-1}}\sum_{q:g_{fq}=-1} \hat{\bar{Y}}(q). 
$$
Let $\hat{\bm{\tau}} = (\hat{\tau}_1, \ldots, \hat{\tau}_{F})'$ be the $F$ dimensional column vector consisting of all factorial effect estimators.

In the finite population inference, the covariates and potential outcomes are all fixed, and the only random component is the treatment vector $\bm{Z}$. In the asymptotic analysis, we further embed the finite population into a sequence with increasing sizes, and introduce the following regularity conditions.

\begin{condition}\label{cond:fp}
	As $n\rightarrow\infty$, 
	the sequence of finite populations satisfies that for each $1\leq q\neq k\leq Q$, 
	\begin{itemize}
		\item[(i)] the proportion of units under treatment combination $q$, $n_q/n$, has a positive limit, 
		\item[(ii)] the finite population variance and covariances $S_{qq}, S_{qk}, \bm{S}_{\bm{xx}}$ and $\bm{S}_{q,\bm{x}}$ have limiting values, and $\bm{S}_{\bm{xx}}$ and its limit are non-degenerate,
		\item[(iii)] $\max_{1\leq i\leq n} |Y_i(q) - \bar{Y}(q)|^2/n \rightarrow 0$ and 
		$
		\max_{1\leq i\leq n}\|\bm{x}_i - \bar{\bm{x}}\|_2^2/n \rightarrow 0.  
		$
	\end{itemize}
\end{condition}

\begin{proposition}\label{prop:mean_var_cre}
	Under the CRFE, 
	$(\hat{\bm{\tau}}'-\bm{\tau}', \hat{\bm{\tau}}_{\bm{x}}')'$ has mean zero and sampling covariance matrix
	\begin{align*}
	\bm{V} & \equiv 
	2^{-2(K-1)}
	\sum_{q=1}^{Q} n_q^{-1}
	\begin{pmatrix}
	\coef_q \coef_q' S_{qq} & (\coef_q \coef_q') \otimes \bm{S}_{q,\bm{x}} \\
	(\coef_q \coef_q') \otimes \bm{S}_{\bm{x}, q} & 
	(\coef_q \coef_q') \otimes \bm{S}_{\bm{xx}}
	\end{pmatrix} 
	- 
	n^{-1}
	\begin{pmatrix}
	\bm{S}_{\bm{\tau\tau}} & \bm{0}\\
	\bm{0} & \bm{0}
	\end{pmatrix}
	\\
	&\equiv
	\begin{pmatrix}
	\bm{V}_{\bm{\tau} \bm{\tau}} & \bm{V}_{\bm{\tau} \bm{x}}\\
	\bm{V}_{\bm{x} \bm{\tau}} & \bm{V}_{\bm{x} \bm{x}}
	\end{pmatrix}.
	\end{align*}
	Under the CRFE and Condition \ref{cond:fp}, $(\hat{\bm{\tau}}'-\bm{\tau}', \hat{\bm{\tau}}_{\bm{x}}')' \apprsim \mathcal{N}(\bm{0}, \bm{V})$.
\end{proposition}

Proposition \ref{prop:mean_var_cre} follows from a finite population central limit theorem \citep[][Theorems 3 and 5]{fcltxlpd2016}, with the proof   in Appendix A2 of the Supplementary Material \citep{rerandfacsupp}.
Proposition \ref{prop:mean_var_cre} immediately gives the sampling properties of any single factorial effect estimator. 
Let $S_{\tau_f\tau_f}$ be the $f$th diagonal element of $\bm{S}_{\bm{\tau\tau}},$ 
and 
$
V_{\tau_f\tau_f} = 
2^{-2(K-1)}\sum_{q=1}^{Q} n_q^{-1}S_{qq} - n^{-1}S_{\tau_f\tau_f}
$
be the $f$th diagonal element of $\bm{V}_{\bm{\tau\tau}}$.   
Then  
$\hat{\tau}_f$ is unbiased for $\tau_f$ with sampling variance $V_{\tau_f\tau_f}$, and 
$
\hat{\tau}_f-\tau_f \apprsim \mathcal{N}(0, V_{\tau_f\tau_f}).   
$
Moreover, $\bm{S}_{\bm{\tau\tau}}$ cannot be unbiasedly estimated from the observed data, and it equals $\bm{0}$ under the {\it additivity} defined below. Under the additivity, the individual treatment effect does not depend on covariates, i.e., there is no treatment-covariate interaction.

\begin{definition}\label{def:additive}
	The factorial effects are additive if and only if the individual factorial effect $\bm{\tau}_i$ is a constant vector for all units, or, equivalently,  
	$\bm{S}_{\bm{\tau\tau}} = \bm{0}$. 
\end{definition}

Under the CRFE, the observed sample variance $s_{qq}=(n_q-1)^{-1}\sum_{i:Z_i=q}\{ Y_i^\obs - \hat{\bar{Y}}(q) \}^2$ is unbiased for $S_{qq}$, because the units receiving treatment combination $q$ are from a simple random sample of size $n_q$. 
Similar to \citet{Neyman:1923}, we can conservatively estimate $\bm{V}_{\bm{\tau\tau}}$ by $2^{-2(K-1)}
\sum_{q=1}^{Q} n_q^{-1}\coef_q \coef_q' s_{qq}$, 
and then construct Wald-type confidence sets for $\bm{\tau}$. 
Both the sampling covariance estimator and confidence sets are asymptotically conservative 
unless the additivity holds.  
It is then straightforward to construct confidence sets for any linear transformations of $\bm{\tau}$.

\subsection{Linear projections}
First, we decompose the potential outcomes.
Let 
$Y^\myparallel_i(q) = \bar{Y}(q) + \bm{S}_{q,\bm{x}}\bm{S}_{\bm{xx}}^{-1}(\bm{x}_i-\bar{\bm{x}})$ be the finite population linear projection of the $Y_i(q)$'s on the $\bm{x}_i$'s, and 
$Y^\myperp_i(q) = Y_i(q) - Y^\myparallel_i(q)$
be the corresponding residual. 
The finite population linear projection of $\bm{\tau}_i$ on $\bm{x}_i$ is then $\bm{\tau}_i^{\myparallel} = 
2^{-(K-1)}\sum_{q=1}^{Q}  \coef_q Y_i^\myparallel(q)$, and the corresponding residual is $\bm{\tau}_i^\myperp = 2^{-(K-1)}\sum_{q=1}^{Q}  \coef_q Y_i^\myperp(q)$. 
Let $S_{qq}^{\myparallel}, S_{qq}^{\myperp}, \bm{S}_{\bm{\tau\tau}}^{\myparallel}$ and $\bm{S}_{\bm{\tau\tau}}^{\myperp}$ be the finite population variances and covariances of $Y^\myparallel(q), Y^\myperp(q), \bm{\tau}^{\myparallel}$ and $\bm{\tau}^{\myperp}$, respectively. 
Define 
\begin{eqnarray*}
\bm{V}_{\bm{\tau}\bm{\tau}}^\myparallel  &=&
2^{-2(K-1)} \sum_{q=1}^{Q} n_q^{-1} \coef_q\coef_q' \cdot S_{qq}^{\myparallel} - n^{-1}\bm{S}_{\bm{\tau\tau}}^{\myparallel} 
,\\
\bm{V}_{\bm{\tau}\bm{\tau}}^\myperp &=& 
2^{-2(K-1)} \sum_{q=1}^{Q} n_q^{-1} \coef_q\coef_q' \cdot S_{qq}^{\myperp} - n^{-1}\bm{S}_{\bm{\tau\tau}}^{\myperp}
\end{eqnarray*}
as analogues of the sampling covariance $\bm{V}_{\bm{\tau}\bm{\tau}}$ in Proposition \ref{prop:mean_var_cre},
with the potential outcomes $Y_i(q)$'s replaced by the linear projections $Y_i^{\myparallel}(q)$'s and the residuals $Y_i^{\myperp}(q)$'s, respectively. We have $\bm{V}_{\bm{\tau}\bm{\tau}} = \bm{V}_{\bm{\tau}\bm{\tau}}^\myparallel + \bm{V}_{\bm{\tau}\bm{\tau}}^\myperp$. 

Second, we decompose the factorial effect estimator $\hat{\bm{\tau}}$.

\begin{theorem}\label{thm:joint_var_expl_unexpl}
	Under the CRFE, the linear projection of $\hat{\bm{\tau}}-\bm{\tau}$ on $\hat{\boldsymbol{\tau}}_{\boldsymbol{x}}$ is $\bm{V}_{\bm{\tau x}} \bm{V}_{\bm{xx}}^{-1}\hat{\boldsymbol{\tau}}_{\boldsymbol{x}}$, 
	the corresponding residual is $\hat{\bm{\tau}}-\bm{\tau}-\bm{V}_{\bm{\tau x}} \bm{V}_{\bm{xx}}^{-1}\hat{\boldsymbol{\tau}}_{\boldsymbol{x}}$, 
	and they have sampling covariances: 
	\begin{align*}
	& \Cov\left(
	\bm{V}_{\bm{\tau x}} \bm{V}_{\bm{xx}}^{-1}\hat{\boldsymbol{\tau}}_{\boldsymbol{x}}
	\right)
	= \bm{V}_{\bm{\tau}\bm{\tau}}^\myparallel, 
	\quad 
	\Cov\left(
	\hat{\bm{\tau}}-\bm{\tau}-\bm{V}_{\bm{\tau x}} \bm{V}_{\bm{xx}}^{-1}\hat{\boldsymbol{\tau}}_{\boldsymbol{x}}
	\right)
	= \bm{V}_{\bm{\tau}\bm{\tau}}^\myperp, 
	\\  
	& \Cov
	\left(
	\bm{V}_{\bm{\tau x}} \bm{V}_{\bm{xx}}^{-1}\hat{\boldsymbol{\tau}}_{\boldsymbol{x}}, \ 
	\hat{\bm{\tau}}-\bm{\tau}-\bm{V}_{\bm{\tau x}} \bm{V}_{\bm{xx}}^{-1}\hat{\boldsymbol{\tau}}_{\boldsymbol{x}}
	\right) = \bm{0}. 
	\end{align*}
\end{theorem}

Theorem \ref{thm:joint_var_expl_unexpl} follows from Proposition \ref{prop:mean_var_cre} and some matrix calculations, with the proof   in Appendix A2 of the Supplementary Material \citep{rerandfacsupp}.
Let $V_{\tau_f\tau_f}^\myparallel$ and 
$S_{\tau_f\tau_f}^{\myparallel}$ be the $f$th diagonal elements of $\bm{V}_{\bm{\tau\tau}}^\myparallel$ and  $\bm{S}_{\bm{\tau\tau}}^{\myparallel}$, respectively. 
The multiple correlation in the following corollary will play an important role in the asymptotic sampling distribution of $\hat{\tau}_f$ under rerandomization. We summarize its equivalent forms below.

\begin{corollary}\label{cor:R2_f}
	Under the CRFE, 
	the sampling squared multiple correlation between $\hat{\tau}_f$ and $\hat{\bm{\tau}}_{\bm{x}}$ 
	has the following equivalent forms: 
	\begin{align*} 
	R^2_f = 
	\text{Corr}^2(\hat{\tau}_f, \hat{\bm{\tau}}_{\bm{x}})
	= 
	\frac{V_{\tau_f\tau_f}^\myparallel}{V_{\tau_f\tau_f}}
	=
	\frac{
		2^{-2(K-1)}\sum_{q=1}^{Q} n_q^{-1}S_{qq}^{\myparallel} - n^{-1} S_{\tau_f\tau_f}^{\myparallel}
	}{
		2^{-2(K-1)}\sum_{q=1}^{Q} n_q^{-1}S_{qq} - n^{-1}S_{\tau_f\tau_f}
	} .
	\end{align*}
It reduces to $R^2_f = S_{11}^{\myparallel}/S_{11}$, the finite population squared multiple correlation between $Y(1)$ and $\bm{x}$
	under the additivity in Definition \ref{def:additive}.  
\end{corollary}

The proof of Corollary \ref{cor:R2_f} is in Appendix A2 of the Supplementary Material \citep{rerandfacsupp}.

\section{Rerandomization using the Mahalanobis distance}\label{sec:ReFM}

As shown in Section \ref{sec:asym_dist_crfe}, although $\hat{\bm{\tau}}_{\bm{x}}$ has mean $\bm{0}$, its realized value can be very different from $\bm{0}$ for a particular treatment allocation. Rerandomization can avoid this drawback. 
In the design stage, we can force balance of the covariate means by ensuring $\hat{\bm{\tau}}_{\bm{x}}$ to be ``small."

\subsection{Mahalanobis distance criterion}
A measure of the magnitude of $\hat{\bm{\tau}}_{\bm{x}}$ is the Mahalanobis distance 
$
M \equiv 
\hat{\bm{\tau}}_{\bm{x}}' \bm{V}_{\bm{xx}}^{-1} \hat{\bm{\tau}}_{\bm{x}}.
$
We further let $a$ be a positive constant predetermined in the design stage. 
Using $M$ as the balance criterion,  we accept a treatment assignment vector $\bm{Z}$ from the CRFE if and only if $M\leq a$. 
Below we use ReFM to denote $2^K$ rerandomized factorial experiments using $M$ as the criterion, and $\mathcal{M}$ to denote the event that the treatment vector $\bm{Z}$ satisfies this criterion. From Proposition \ref{prop:mean_var_cre}, $M$ is asymptotically $\chi^2_{LF}$, 
and therefore 
the asymptotic acceptance probability is $p_{a} = P(\chi^2_{LF}\leq a)$ under ReFM. 
In practice, we usually choose a small threshold $a$, or equivalently a small $p_a$, e.g., $p_a=0.001$. 
However, we do not advocate choosing $p_a$ to be too small, because an extremely small $p_a$ may lead to too few configurations of treatment allocations in ReFM.

\subsection{Asymptotic sampling distribution of $\hat{\boldsymbol{\tau}}$ under ReFM}\label{sec:asym_dist}
Rerandomization in the design stage accepts only the treatment assignments resulting in covariate balance, which consequently changes the sampling distribution of $\hat{\bm{\tau}}$. Understanding the asymptotic sampling distribution of $\hat{\bm{\tau}}$ is crucial for conducting the classical repeated sampling inference of $\bm{\tau}$. 
Intuitively, $\hat{\bm{\tau}}$ has two parts: one part is orthogonal to $\hat{\bm{\tau}}_{\bm{x}}$ and thus unaffected by ReFM, and the other part is the linear projection onto $\hat{\bm{\tau}}_{\bm{x}}$ and thus affected by ReFM. 
Let $\bm{\varepsilon}\sim \mathcal{N}(\bm{0}, \bm{I}_F)$ be an $F$ dimensional standard Gaussian random vector, and  
$\bm{\zeta}_{LF,a} \sim \bm{D} \mid \bm{D}'\bm{D}\leq a$ be an $LF$ dimensional truncated Gaussian random vector, where $\bm{D} = (D_1, \ldots, D_{LF})'\sim \mathcal{N}(\bm{0}, \bm{I}_{LF})$. The following theorem shows the asymptotic sampling distribution of $\hat{\bm{\tau}}$.

\begin{theorem}\label{thm:joint_refm}
	Under ReFM and Condition \ref{cond:fp}, 
	\begin{eqnarray}\label{eq:joint_refm}
	\hat{\bm{\tau}} - \bm{\tau} \mid  \mathcal{M} 
	& \apprsim & 
	\left(\bm{V}_{\bm{\tau\tau}}^\myperp\right)^{1/2} \bm{\varepsilon} +  
	\left(\bm{V}_{\bm{\tau\tau}}^\myparallel\right)^{1/2}_{LF}
	\bm{\zeta}_{LF,a}, 
	\end{eqnarray}
	where $\bm{\varepsilon}$ and $\bm{\zeta}_{LF,a}$ are independent. 
\end{theorem}

Theorem \ref{thm:joint_refm} holds because the sampling distribution of $\hat{\bm{\tau}}$ under rerandomization is the same as the conditional distribution of $\hat{\bm{\tau}}$ given $M\leq a$. Its proof  is in Appendix A3 of the Supplementary Material \citep{rerandfacsupp}. We emphasize that, although the matrix $\left(\bm{V}_{\bm{\tau\tau}}^\myparallel\right)^{1/2}_{LF}$ may not be unique, the asymptotic sampling distribution \eqref{eq:joint_refm} is. 
Therefore, the asymptotic sampling distribution of $\hat{\bm{\tau}} - \bm{\tau}$ under ReFM depends only on $L,F,a$,  $\bm{V}_{\bm{\tau\tau}}^\myperp$, and $\bm{V}_{\bm{\tau\tau}}^\myparallel$. 
Theorem \ref{thm:joint_refm} immediately implies the asymptotic sampling distribution of a single factorial effect estimator. 
Let $\varepsilon_0\sim \mathcal{N}(0,1)$, and ${\eta}_{LF,a} \sim D_1 \mid \bm{D}'\bm{D}\leq a$ be the first coordinate of $\bm{\zeta}_{LF,a}$. 

\begin{corollary}\label{cor:dist_refm_one_linear}
	Under ReFM and Condition \ref{cond:fp}, for $1\leq f\leq F$, 
	\begin{eqnarray}\label{eq:dist_refm_single}
	\hat{\tau}_f - \tau_f \mid  \mathcal{M}
	& \apprsim & 
	\sqrt{V_{\tau_f\tau_f}}
	\left(
	\sqrt{1-R^2_f}\cdot \varepsilon_0 + \sqrt{R^2_f} \cdot \eta_{LF,a}
	\right). 
	\end{eqnarray}
\end{corollary}

The proof of Corollary \ref{cor:dist_refm_one_linear} is in Appendix A3 of the Supplementary Material \citep{rerandfacsupp}.
The marginal asymptotic sampling distribution \eqref{eq:dist_refm_single} under ReFM has the same form as that under rerandomized treatment-control experiments using the Mahalanobis distance \citep{asymrerand2016}.

\subsection{Review of the central convex unimodality}
In this subsection, we review a generalization of unimodality to multivariate distributions and apply it to study the asymptotic sampling distribution \eqref{eq:joint_refm}. This property will be important for constructing conservative large-sample confidence sets later. 

Although the definition of symmetric unimodality for univariate distribution is simple and intuitive, it is nontrivial to generalize it to multivariate distribution. Here we adopt the {\it central convex unimodality} proposed by \citet{dharmadhikari1976} based on the results of \citet{sherman1955}, which is also equivalent to the symmetric unimodality in \citet{kanter1977}. 
For a set $\mathcal{B}$ of distributions on $\mathbb{R}^m$, we say that $\mathcal{B}$ is {\it closed convex} if it satisfies two conditions: (i) for any distributions $\nu_1,\nu_2\in \mathcal{B}$ and for any $\lambda\in (0,1)$, the distribution $(1-\lambda)\nu_1+\lambda\nu_2$ is in $\mathcal{B}$, and (ii) a distribution $\nu$ is in $\mathcal{B}$ if there exists a sequence of distributions in $\mathcal{B}$ converging weakly to $\nu$. 
For any set $\mathcal{C}$ of distributions, let the {\it closed convex hull} of $\mathcal{C}$ be the smallest closed convex set containing $\mathcal{C}$. 
A compact convex set in Euclidean space $\mathbb{R}^m$ is called a {\it convex body} if it has a nonempty interior. A set $\mathcal{K}\subset\mathbb{R}^m$ is {\it symmetric} if $\mathcal{K} =  \{-\bm{a}: \bm{a}\in \mathcal{K}\}$. 
Below we introduce the definition. 

\begin{definition}\label{def:ccu}
	A distribution on $\mathbb{R}^m$ is {\it central convex unimodal} if it is in the closed convex hull of $\mathcal{U}$, where $\mathcal{U}$ is the set of all uniform distributions on symmetric convex bodies in $\mathbb{R}^m$. 
\end{definition}

The class of central convex unimodal distributions is closed under convolution, marginality, product measure, and weak convergence \citep{kanter1977}. A sufficient condition for the central convex unimodality is having a log-concave probability density function \citep{kanter1977, dharmadhikari1988}. The following proposition states the central convex unimodality of the asymptotic sampling distribution of $\hat{\bm{\tau}}-\bm{\tau}$ under ReFM.

\begin{proposition}\label{prop:ccu_refm}
	The standard Gaussian random vector $\bm{\varepsilon}$, the truncated Gaussian random vector $\bm{\zeta}_{LF,a}$, and the asymptotic sampling distribution \eqref{eq:joint_refm} are all central convex unimodal.
\end{proposition}

Proposition \ref{prop:ccu_refm} follows from the log-concavity of the densities of $\bm{\varepsilon}$ and $\bm{\zeta}_{LF,a}$ and the closedness of the class of central convex unimodal distributions under linear transformation and convolution. Its proof is in Appendix A3 of the Supplementary Material \citep{rerandfacsupp}.

\subsection{Representation for the asymptotic sampling distribution of $\hat{\boldsymbol{\tau}}$}

In this subsection, we further represent \eqref{eq:joint_refm} using well-known distributions to gain more insights. Let $\chi^2_{LF,a} \sim \chi^2_{LF}\mid \chi^2_{LF}\leq a$ be a truncated $\chi^2$ random variable, $\bm{S}$ be an $LF$ dimensional random vector whose coordinates are independent random signs with probability $1/2$ of being $\pm 1$, and $\bm{\beta}$ be an $LF$ dimensional Dirichlet random vector with parameters $(1/2, \ldots, 1/2)$. 
Let $\sqrt{\bm{\beta}}$ be the element-wise square root of the vector $\bm{\beta}$, and 
$v_{LF,a} = P(\chi^2_{LF+2}\leq a)/P(\chi^2_{LF}\leq a)
\leq 1$.

\begin{proposition}\label{prop:represent}
$\bm{\zeta}_{LF,a}$ is spherically symmetric with covariance $v_{LF,a}\bm{I}_{LF}$.  
It follows 
	$
	\bm{\zeta}_{LF,a} \sim \chi_{LF,a} \cdot \bm{S}\circ\sqrt{\bm{\beta}}, 
	$
	where $(\chi_{LF,a}, \bm{S}, \bm{\beta})$ are jointly independent. 
\end{proposition}

Proposition \ref{prop:represent} follows from the spherical symmetry of the standard multivariate Gaussian random vector, with the proof   in Appendix A3 of the Supplementary Material \citep{rerandfacsupp}.
Proposition \ref{prop:represent} allows for easy simulations of the asymptotic sampling distribution \eqref{eq:joint_refm}, which is useful for the repeated sampling inference discussed shortly. 
For simplicity, in the remaining paper, we assume that $\bm{V}_{\bm{\tau\tau}}$ is invertible whenever we mention its inverse; otherwise we can focus on a lower dimensional linear transformation of $\hat{\bm{\tau}}$ \citep{rerandfacsupp}.
Let $\bm{R} = \bm{V}_{\bm{\tau\tau}}^{-1/2}\bm{V}_{\bm{\tau\tau}}^{\myparallel}\bm{V}_{\bm{\tau\tau}}^{-1/2}$ be the matrix measuring the relative sampling covariance of $\hat{\bm{\tau}}$ explained by $\hat{\bm{\tau}}_{\bm{x}}$, 
and $\bm{R} = \bm{\Gamma}\bm{\Pi}^2 \bm{\Gamma}'$ be its eigen-decomposition, where $\bm{\Gamma}\in \mathbb{R}^{F\times F}$ is an orthogonal matrix and $\bm{\Pi}^2 = \textup{diag}(\pi_1^2, \ldots, \pi_F^2)\in \mathbb{R}^{F\times F}$ is a diagonal matrix with nonnegative elements. 
The eigenvalues $(\pi_1^2, \ldots, \pi_F^2)$ are the canonical correlations between the sampling distributions of $\hat{\bm{\tau}}$ and $\hat{\bm{\tau}}_{\bm{x}}$ under the CRFE, which measure the association between the potential outcomes and covariates. 
Under the additivity, $\pi_1^2 =  \cdots = \pi_F^2 = S_{11}^\myparallel/S_{11}$. 
The following corollary gives an equivalent form of \eqref{eq:joint_refm} highlighting the dependence on the canonical correlations $(\pi_1^2, \ldots, \pi_F^2)$.

\begin{corollary}\label{cor:another_form_refm}
	Under ReFM and Condition \ref{cond:fp}, \eqref{eq:joint_refm} is equivalent to 
	\begin{equation}\label{eq:equiva_dist_ReFM}
	\hat{\bm{\tau}} - \bm{\tau} \mid \mathcal{M}
	\  \apprsim \ 
	\bm{V}_{\bm{\tau\tau}}^{1/2}\bm{\Gamma}
	\left\{
	\left(\bm{I}_F - \bm{\Pi}^2\right)^{1/2}\bm{\varepsilon} + \left(\bm{\Pi}, \bm{0}_{F\times (L-1)F}\right) \bm{\zeta}_{LF,a}
	\right\}. 
	\end{equation} 
\end{corollary}

The proof of Corollary \ref{cor:another_form_refm} is in Appendix A3 of the Supplementary Material \citep{rerandfacsupp}.
The second term in \eqref{eq:equiva_dist_ReFM}, affected by rerandomization, depends on the canonical correlations $(\pi_1^2, \ldots, \pi_F^2)$ and the asymptotic acceptance probability $p_a$ of ReFM. Below we use a numerical example to illustrate such dependence. 
\begin{example}
	We consider the case with $L=1$, $K=2$ and $F=3$, and focus on the standardized distribution
	$
	\left(\bm{I}_3 - \bm{\Pi}^2\right)^{1/2}\bm{\varepsilon} + \bm{\Pi} \bm{\zeta}_{3,a},
	$ 
	which depends on $\bm{\Pi}^2=\textup{diag}(\pi_1^2, \pi_2^2, \pi_3^2)$ and $p_a = P(\chi^2_3 \leq a)$. 
	First, we fix $(\pi_2^2, \pi_3^2, p_a) = (0.5, 0.5, 0.001)$. Figure \ref{fig:std_dist_pi2_all} shows the density of the first two coordinates of $\bm{\zeta}_{3,a}$ for different $\pi_1^2$. As $\pi_1^2$ increases, the density becomes more concentrated around zero, showing that the stronger the association is between the potential outcomes and covariates, the more precise the factorial effect estimators are.

	Second, we fix $(\pi_1^2, \pi_2^2, \pi_3^2)=(0.5, 0.5, 0.5)$. Figure \ref{fig:std_dist_pa_all} shows the density of the first two coordinates of $\bm{\zeta}_{3,a}$ for different $p_a$. As the asymptotic acceptance probability $p_a$ decreases, the density becomes more concentrated around zero, 
	confirming the intuition that a smaller asymptotic acceptance probability gives us more precise factorial effect estimators. 
	Note that the first $\bm{\varepsilon}$ component in the asymptotic sampling distribution \eqref{eq:equiva_dist_ReFM} does not depend on $p_a$ and is usually nonzero. For example, when $\bm{V}_{\bm{\tau\tau}}^\myperp$ is positive definite, 
	$\bm{I} - \bm{R} = \bm{V}_{\bm{\tau\tau}}^{-1/2}\bm{V}_{\bm{\tau\tau}}^{\myperp}\bm{V}_{\bm{\tau\tau}}^{-1/2}$ is positive definite, as well as  
	the coefficient of $\bm{\varepsilon}$ in \eqref{eq:equiva_dist_ReFM}. 
	Therefore, the gain of ReFM by decreasing $p_a$ usually becomes smaller as $p_a$ decreases.
	$\hfill\Box$

	\begin{figure}[htb]
		\centering
		\begin{subfigure}{1\textwidth}
			\centering
			\includegraphics[width=0.8\linewidth]{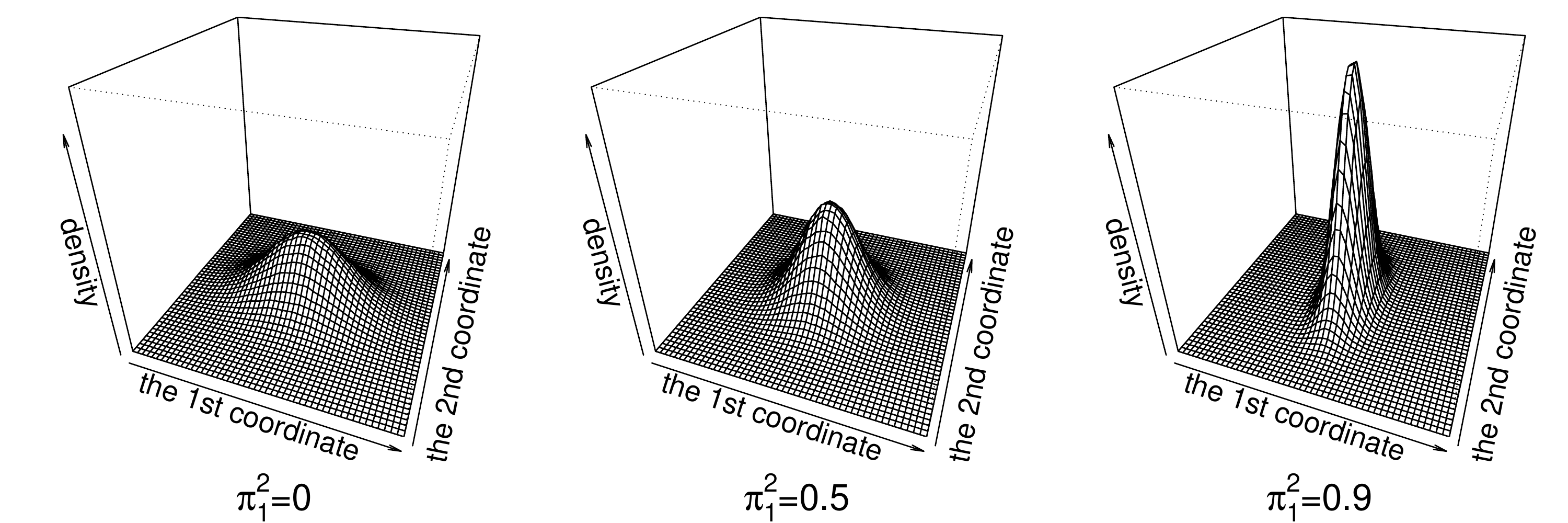}
			\caption{
				$(\pi_2^2, \pi_3^2, p_{\text{a}}) = (0.5, 0.5, 0.001)$}\label{fig:std_dist_pi2_all}
		\end{subfigure}\\
		\begin{subfigure}{1\textwidth}
			\centering
			\includegraphics[width=0.8\linewidth]{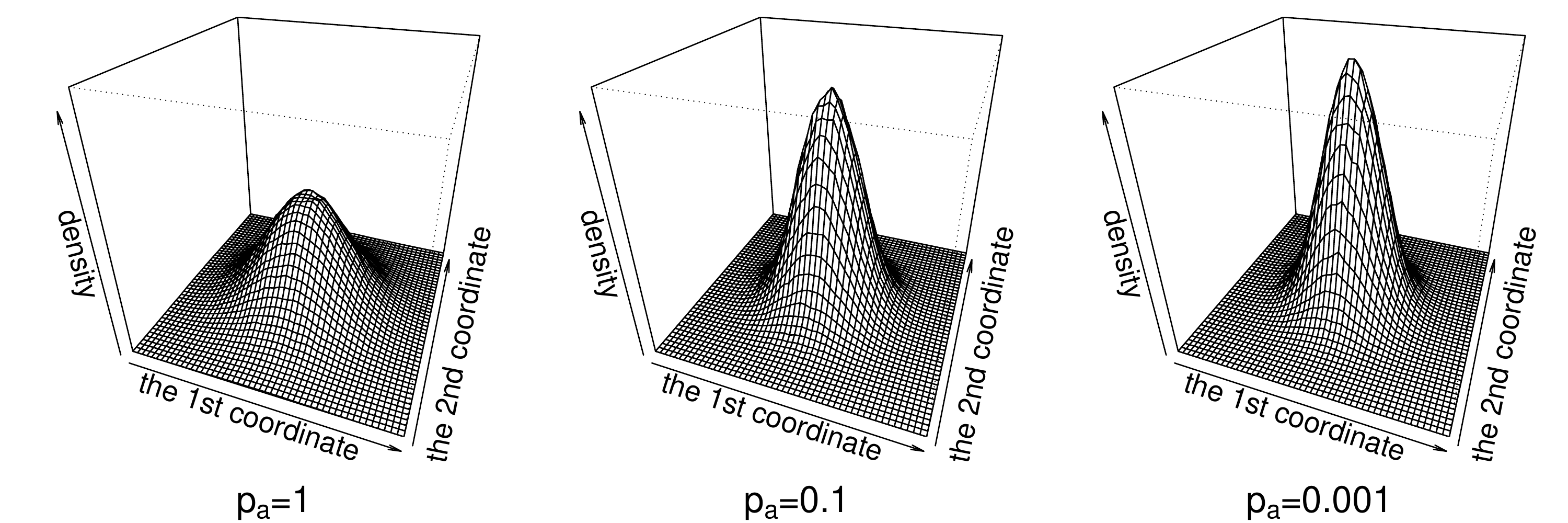}
			\caption{$(\pi_1^2, \pi_2^2, \pi_3^2)=(0.5, 0.5, 0.5)$}\label{fig:std_dist_pa_all}
		\end{subfigure}
		\caption{Joint density of the first two coordinates of $ 
			\left(\bm{I}_3 - \bm{\Pi}^2\right)^{1/2}\bm{\varepsilon} + \bm{\Pi} \bm{\zeta}_{3,a}$}
	\end{figure}
\end{example}

\subsection{Asymptotic unbiasedness, sampling covariance and peakedness}\label{sec:asym_prop}
In this subsection, we further study the asymptotic properties of $\hat{\bm{\tau}}$ under ReFM. 
First, the factorial effects estimator $\hat{\bm{\tau}}$ is consistent for $\bm{\tau}.$ 
Because covariates are potential outcomes unaffected by the treatment, 
the difference-in-means of any observed or unobserved covariate with respect to any factorial effect has asymptotic mean zero.

Second, we compare the asymptotic sampling covariance matrices of $\hat{\bm{\tau}}$ under ReFM and the CRFE, 
which also gives the reduction in asymptotic sampling covariances of difference-in-means of covariates as a special case.

\begin{theorem}\label{thm:red_var_rem}
	Under Condition \ref{cond:fp}, 
	the asymptotic sampling covariance matrix of $\hat{\bm{\tau}}$ under ReFM is smaller than or equal to that under the CRFE, and the reduction in asymptotic sampling covariance is  $(1-v_{LF,a})n\bm{V}_{\bm{\tau\tau}}^{\myparallel}$.
	Specifically, the percentage reduction in asymptotic sampling variance (PRIASV) of $\hat{\tau}_f$ is $(1-v_{LF,a})R^2_f$.
\end{theorem}

Theorem \ref{thm:red_var_rem} follows from Theorem \ref{thm:joint_refm} and Proposition \ref{prop:represent}, with the proof   in Appendix A4 of the Supplementary Material \citep{rerandfacsupp}.
Rigorously, 
the reductions in Theorem \ref{thm:red_var_rem} should be $(1-v_{LF,a})\lim_{n\rightarrow\infty}n\bm{V}_{\bm{\tau\tau}}^{\myparallel}$ and $(1-v_{LF,a})\lim_{n\rightarrow\infty} R^2_f$. However, for descriptive simplicity, we omit the limit signs.
From Theorem \ref{thm:red_var_rem}, the larger the squared multiple correlation $R^2_f$ is, the more PRIASV of the factorial effect estimator is through ReFM.  
When $a$ is close to zero, or equivalently the asymptotic acceptance probability $p_a$ is small, the asymptotic sampling variance of $\hat{\tau}_f$ reduces to $V_{\tau_f\tau_f}(1-R^2_f)$, which is identical to the asymptotic sampling variance of the regression adjusted estimator under the CRFE discussed in \citet{lu2016covadj}.

Third, we compare the peakedness of the asymptotic sampling distributions of $\hat{\bm{\tau}}$ under ReFM and the CRFE, because of its close connection to the volumes of confidence sets for $\bm{\tau}$. 
\citet{birnbaum1948}, \citet{bickel1976} and \citet{shaked1985ordering} proposed some measures of dispersion for univariate distributions. \citet{sherman1955} and \citet{GIOVAGNOLI1995325} generalized them to multivariate distributions.
\citet{marshall2009} discussed some related properties. 
Here we use the definition in \citet{sherman1955}. 

\begin{definition}\label{def:peak}
	For two symmetric random vectors $\bm{\phi}$ and $\bm{\psi}$ in $\mathbb{R}^m$, 
	we say that $\bm{\phi}$ is more peaked than $\bm{\psi}$ and write as $\bm{\phi} \succ \bm{\psi}$, if $P(\bm{\phi}\in \mathcal{K})\geq P(\bm{\psi}\in \mathcal{K})$ for every symmetric convex set $\mathcal{K}\subset \mathbb{R}^m$. 
\end{definition}

From Definition \ref{def:peak}, 
intuitively, 
the more peaked a random vector is, the more ``concentrated" around zero it is. 
Therefore, when comparing two experimental designs, the one with more peaked sampling distribution of the causal estimator gives more precise estimate for the true causal effect. That is, peakedness measures the efficiencies of the designs. 

As a basic fact, the ordering of peakedness directly implies the ordering of the covariance matrices. 

\begin{proposition}\label{prop:peak_cov}
	For two symmetric random vectors $\bm{\phi}$ and $\bm{\psi}$ in $\mathbb{R}^m$ with finite second moments, if $\bm{\phi} \succ \bm{\psi}$, then $\Cov(\bm{\phi})\leq \Cov(\bm{\psi})$.  
\end{proposition}

Proposition \ref{prop:peak_cov} follows from some algebra, 
with the proof in Appendix A5 of the Supplementary Material \citep{rerandfacsupp}.
For two Gaussian vectors $\bm{\phi}$ and $\bm{\psi}$, $\Cov(\bm{\phi})\leq \Cov(\bm{\psi})$ also implies $\bm{\phi} \succ \bm{\psi}$. The reverse of Proposition \ref{prop:peak_cov} does not hold for general random vectors. 
For example, we compare a standard Gaussian random variable $\varepsilon_0$ and a truncated Gaussian random variable $\xi_0 \sim \varepsilon_0 \mid 0.5 \leq \varepsilon_0^2 \leq 1$. 
Both random variables are symmetric around zero and $\Var(\xi_0) < 1  = \Var(\varepsilon_0)$. However, 
$\xi_0$ is not more peaked than $\varepsilon_0$, because
$P(|\xi_0| \leq 0.5) = 0 < P(|\varepsilon_0| \leq 0.5)$.

The following theorem shows that the difference-in-means estimator is more ``concentrated" under ReFM than under the CRFE.

\begin{theorem}\label{thm:red_qr_rem}
	Under Condition \ref{cond:fp}, 
	the asymptotic sampling distribution of $\hat{\bm{\tau}} - \bm{\tau}$ under ReFM is more peaked than that under the CRFE. 
\end{theorem}

Theorem \ref{thm:red_qr_rem}  holds because the truncated Gaussian random vector $\bm{\zeta}_{LF,a}$ is more peaked than the standard Gaussian random vector. Its proof is in Appendix A5 of the Supplementary Material \citep{rerandfacsupp}.
First, Theorem \ref{thm:red_qr_rem}, coupled with Proposition \ref{prop:peak_cov}, implies the asymptotic sampling covariance of $\hat{\bm{\tau}}$ is smaller under ReFM than under the CRFE.
Second, Theorem \ref{thm:red_qr_rem} shows that asymptotically,  $\hat{\bm{\tau}}-\bm{\tau}$ has larger probability to be in any symmetric convex set under ReFM than under the CRFE.  
For a positive definite matrix $\bm{\Lambda} \in \mathbb{R}^{p\times p}$ and $c\geq 0$,  
let $\mathcal{O}(\bm{\Lambda}, c) \equiv \{ \bm{\mu}: \bm{\mu}'\bm{\Lambda}^{-1}\bm{\mu}\leq c \}$. 
The following theorem implies that, for the special class of symmetric convex sets, $\{\mathcal{O}(\bm{V}_{\bm{\tau\tau}}, c): c\geq 0\}$, the asymptotic probability that $\hat{\bm{\tau}}-\bm{\tau}$ lies in $\mathcal{O}(\bm{V}_{\bm{\tau\tau}}, c)$ is nondecreasing in the canonical correlation $\pi_k^2$'s.

\begin{theorem} \label{thm:red_qr_rem_nond_ccorr}
Under ReFM, assume Condition \ref{cond:fp}. Let $c_{1-\alpha}$ be the solution of 
	$
	\lim_{n\rightarrow\infty}P\left\{
	\hat{\bm{\tau}} - \bm{\tau} \in \mathcal{O}(\bm{V}_{\bm{\tau\tau}}, c_{1-\alpha}) \mid \mathcal{M}
	\right\} = 1 - \alpha
	$
	 for any $\alpha\in (0,1)$. It
	depends only on $(L,K,a)$ and the canonical correlation $\pi_k^2$'s, and is nonincreasing in these canonical correlations for fixed $(L,K,a)$. 
\end{theorem}

Theorem \ref{thm:red_qr_rem_nond_ccorr} is a multivariate extension of \citet[][Theorem 2]{asymrerand2016}, 
with the proof in Appendix A5 of the Supplementary Material \citep{rerandfacsupp}.
The set $\mathcal{O}(\bm{V}_{\bm{\tau\tau}}, c_{1-\alpha})$ in Theorem \ref{thm:red_qr_rem_nond_ccorr} is a $1-\alpha$ asymptotic quantile region of $\hat{\bm{\tau}} - \bm{\tau}$ under ReFM. 
From Theorem \ref{thm:red_qr_rem_nond_ccorr}, with larger canonical correlation  $\pi_k^2$'s, ReFM leads to more percentage reduction in volume of the $1-\alpha$ asymptotic quantile region $\mathcal{O}(\bm{V}_{\bm{\tau\tau}}, c_{1-\alpha})$ of $\hat{\bm{\tau}} - \bm{\tau}$.

Moreover, we can establish similar conclusions as Theorems \ref{thm:red_qr_rem} and \ref{thm:red_qr_rem_nond_ccorr} for any linear transformation of $\hat{\bm{\tau}}$. This follows from two facts: (i) the peakedness relationship is invariant under linear transformations \citep[][Lemma 7.2]{dharmadhikari1988}, i.e., for any $\bm{C}\in \mathbb{R}^{p\times m}$, if $\bm{\phi} \succ \bm{\psi}$, then $\bm{C}\bm{\phi} \succ \bm{C}\bm{\psi}$; 
(ii) the asymptotic sampling distribution of any linear transformation of $\hat{\bm{\tau}}$ has the same form as $\hat{\bm{\tau}}$, i.e.,  a linear combination of a standard Gaussian random vector and a truncated Gaussian random vector.  
For conciseness, we relegate the discussion to the Supplementary Material \citep{rerandfacsupp}, and consider only a single factorial effect estimator in the main text. In this case, the comparison between peakedness of two univariate asymptotic sampling distributions under ReMF and the CRFE reduces to the comparison of the lengths of quantile ranges \citep{asymrerand2016}.

\begin{corollary}\label{cor:red_in_qr_single_remf}
	Under Condition \ref{cond:fp}, for any $1\leq f\leq F$ and $\alpha\in (0,1)$, the threshold $c_{1-\alpha}$ 
	for the $1-\alpha$ asymptotic symmetric quantile range 
	$[-c_{1-\alpha}V_{\tau_f\tau_f}^{1/2}, c_{1-\alpha}V_{\tau_f\tau_f}^{1/2}]$
	of 
	$\hat{\tau}_f - \tau_f$ under ReFM is smaller than or equal to that under the CRFE, and is nonincreasing in $R_f^2$.
\end{corollary}

The proof of Corollary \ref{cor:red_in_qr_single_remf} is in Appendix A5 of the Supplementary Material \citep{rerandfacsupp}.
From Corollary \ref{cor:red_in_qr_single_remf}, with larger squared multiple correlation $R^2_f$, ReFM leads to more percentage reductions in lengths of the asymptotic quantile ranges of $\hat{\tau}_f - \tau_f$. 

\subsection{Conservative covariance estimator and confidence sets under ReFM}\label{sec:var_est_refm}

The asymptotic sampling distribution \eqref{eq:joint_refm} of $\hat{\bm{\tau}}$ under ReFM depends on $\bm{V}_{\bm{\tau}\bm{\tau}}^\myperp$ and $(\bm{V}_{\bm{\tau\tau}}^\myparallel)^{1/2}_{LF} = \bm{V}_{\bm{\tau x}}\bm{V}_{\bm{xx}}^{-1/2}$, 
which
further depend on $S_{qq}^{\myperp}, \bm{S}_{\bm{\tau\tau}}^{\myperp}$ and $\bm{S}_{q,\bm{x}}\bm{S}_{\bm{xx}}^{-1/2}$. 
Under treatment combination $q$, define $s_{qq}$ as the sample variance of observed outcomes, $\bm{s}_{q,\bm{x}}$ as the sample covariance between observed outcomes and covariates, 
$\bm{s}_{\bm{xx}}(q)$ as the sample covariance of covariates, and 
$s_{qq}^{\myperp} = s_{qq} - \bm{s}_{q,\bm{x}}\bm{s}_{\bm{xx}}^{-1}(q)\bm{s}_{\bm{x},q}$ as the sample variance of the residuals from the linear projection of observed outcomes on covariates.
We estimate $\bm{V}_{\bm{\tau\tau}}^\myperp$ by 
\begin{align}\label{eq:V_tautau_perp_est}
\hat{\bm{V}}_{\bm{\tau\tau}}^\myperp & = 2^{-2(K-1)} \sum_{q=1}^{Q} n_q^{-1} s_{qq}^{\myperp}\coef_q\coef_q', 
\end{align}
$\bm{V}_{\bm{\tau x}}$ by 
$
\hat{\bm{V}}_{\bm{\tau} \bm{x}} = 
2^{-2(K-1)}
\sum_{q=1}^{Q} n_q^{-1}
(\coef_q \coef_q') \otimes
\left\{ \bm{s}_{q,\bm{x}}\bm{s}_{\bm{xx}}^{-1/2}(q)\bm{S}_{\bm{xx}}^{1/2}
\right\}, 
$
and $(\bm{V}_{\bm{\tau\tau}}^\myparallel)^{1/2}_{LF}$ by $\hat{\bm{V}}_{\bm{\tau x}}\bm{V}_{\bm{xx}}^{-1/2}$. 
We can then 
obtain a covariance estimator and 
construct confidence sets for $\bm{\tau}$ or its linear transformations. 
When the threshold $a$ is small, 
$\bm{\zeta}_{LF,a}$ is close to zero, and
the distribution \eqref{eq:joint_refm} of $\hat{\bm{\tau}}$ is close to the Gaussian distribution with mean $\bm{\tau}$ and covariance matrix $\bm{V}_{\bm{\tau\tau}}^\myperp$. 
Therefore, for a parameter of interest $\bm{C}\bm{\tau}$, we recommend confidence sets of the form 
$
\bm{C}\hat{\bm{\tau}} + \mathcal{O}(\bm{C}\hat{\bm{V}}_{\bm{\tau\tau}}^\myperp\bm{C}', c).  
$
We choose the threshold $c$ based on simulation from the estimated asymptotic sampling distribution,  
and let 
$\hat{c}_{1-\alpha}$ be the $1-\alpha$ quantile of 
$ (\bm{C}\bm{\phi})'(\bm{C}\hat{\bm{V}}_{\bm{\tau\tau}}^\myperp\bm{C}')^{-1}(\bm{C}\bm{\phi})
$
with $\bm{\phi}$ following the estimated asymptotic sampling distribution of $\hat{\bm{\tau}}-\bm{\tau}$.

\begin{theorem}\label{thm:conser_conf_set_refm}
	Under ReFM and Condition \ref{cond:fp},  consider inferring $\bm{C}\bm{\tau}$, where
 $\bm{C}$ has full row rank. The probability limit of the covariance estimator for $\bm{C}\hat{\bm{\tau}}$, 
	$
	\bm{C}\hat{\bm{V}}_{\bm{\tau\tau}}^\myperp\bm{C}' + 
	v_{LF,a}
	\bm{C}\hat{\bm{V}}_{\bm{\tau x}}\bm{V}_{\bm{xx}}^{-1}\hat{\bm{V}}_{\bm{x\tau}}\bm{C}',
	$ 
	is larger than or equal to the sampling covariance, and 
	the $1-\alpha$ confidence set, 
	$
	\bm{C}\hat{\bm{\tau}} + \mathcal{O}(\bm{C}\hat{\bm{V}}_{\bm{\tau\tau}}^\myperp\bm{C}', \hat{c}_{1-\alpha}), 
	$
	has asymptotic coverage rate $\geq 1-\alpha$, 
	with equality holding if $\bm{S}_{\bm{\tau\tau}}^\myperp\rightarrow 0$ as $n\rightarrow \infty$. 
\end{theorem}

Theorem \ref{thm:conser_conf_set_refm} holds because the ordering of peakedness still holds by adding an independent central convex unimodal random vector. Its proof is in Appendix A6 of the Supplementary Material \citep{rerandfacsupp}.
The above confidence sets will be similar to the ones based on regression adjustment if the threshold $a$ is small. Theoretically, we can extend Theorem \ref{thm:conser_conf_set_refm} to general symmetric convex confidence sets, and we relegate this discussion to the Supplementary Material \citep{rerandfacsupp}.

\section{Rerandomization with tiers of factorial effects}\label{sec:ReFTF}

From Corollary \ref{cor:R2_f}, under the additivity, the squared multiple correlations between $\hat{\tau}_f$ and $\hat{\bm{\tau}}_{\bm{x}}$ are the same for all $f$:
$R^2_1=\cdots=R^2_F = S_{11}^{\myparallel}/S_{11}$. 
From Section  \ref{sec:asym_prop}, 
under the additivity, the improvement of the $f$th factorial effect estimator $\hat{\tau}_f$ under ReFM compared to the CRFE is asymptotically the same for all $f$. However, in practice, we are sometimes more interested in some factorial effects than others. For example, the main effects are often more important than higher-order interactions. Therefore, we need a balance criterion resulting in more precise  estimators for the more important factorial effects.

\subsection{Tiers of factorial effects criterion}\label{sec:tier_factor}

Let $\mathcal{F}=\{1,2,\ldots, F\}$ be the set of all factorial effects. We partition $\mathcal{F}$ into $H$ tiers $(\mathcal{F}_1, \ldots,\mathcal{F}_{H})$ with decreasing importance,  
where the $\mathcal{F}_h$'s are disjoint and $\mathcal{F} = \bigcup_{h=1}^H \mathcal{F}_h$. 
The cardinality $F_h \equiv |\mathcal{F}_h|$ represents the number of factorial effects in tier $h$. 
For example, we can partition $\mathcal{F}$ into three tiers:  
$\mathcal{F}_1$ contains the $K$ main effects, $\mathcal{F}_2$ 
contains the $\binom{K}{2}$ interaction effects between two factors, and 
$\mathcal{F}_3$ contains the remaining factorial effects with higher-order interactions.

Define $\gamma^2_{fk} = \text{Corr}^2( \hat{\tau}_f,  \hat{\bm{\tau}}_{\bm{x},k} )$.
When the $f$th factorial effect is more important, we would like to put more restriction on the difference-in-means vector $\hat{\bm{\tau}}_{\bm{x},k}$ with larger squared multiple correlation $\gamma^2_{fk}.$ 
Although general results for the relative magnitudes of the $\gamma_{fk}^2$'s appear too complicated, below we give a proposition under the additivity, which serves as a guideline for the choice of the balance criterion. 

\begin{proposition}\label{prop:comp_R_fm_add}
Under the CRFE, assume the additivity in Definition \ref{def:additive}.
The squared multiple correlations 
	satisfy  
$
	\max_{1\leq k\leq F}\gamma^2_{fk} = \gamma^2_{ff} = R^2_f = S_{11}^{\myparallel}/S_{11}
	$
for $1\leq f \leq F$.
The squared multiple partial correlation between 
	$\hat{\tau}_f$ and $\hat{\bm{\tau}}_{\bm{x}}$ given $\hat{\bm{\tau}}_{\bm{x},f}$ is zero, i.e., 
	the residuals from the linear projections of $\hat{\tau}_f$ and $\hat{\bm{\tau}}_{\bm{x}}$ on $\hat{\bm{\tau}}_{\bm{x},f}$ are uncorrelated. 
	If further $n_1=\cdots=n_{Q} = n/Q$, 
	then 
	$\gamma^2_{fk}=0$ for $k\neq f$.
\end{proposition}

Proposition \ref{prop:comp_R_fm_add} follows from some algebra, with the proof   in Appendix A2 of the Supplementary Material \citep{rerandfacsupp}. 
From Proposition \ref{prop:comp_R_fm_add}, with the additivity and under the CRFE, $\hat{\bm{\tau}}_{\bm{x}}$ explains $\hat{\tau}_f$ in the linear projection only through $\hat{\bm{\tau}}_{\bm{x},f}$.
Therefore, it is desirable to impose more restriction on the difference-in-means of covariates with respect to more important factorial effects under rerandomization.

\subsection{Orthogonalization with tiers of factorial effects}\label{sec:tiers_cov_criterion}
For $1\leq h\leq H$, let $\hat{\bm{\tau}}_{\bm{x}}[\mathcal{F}_h]$ be the subvector of $\hat{\bm{\tau}}_{\bm{x}}$, consisting of the difference-in-means of covariates $\hat{\bm{\tau}}_{\bm{x},f}$ with respect to factorial effect $f\in \mathcal{F}_h$.
From Section \ref{sec:tier_factor}, the smaller the $h$ is, the more restriction we want to impose on $\hat{\bm{\tau}}_{\bm{x}}[\mathcal{F}_h]$. 
However, due to the correlations among the $\hat{\bm{\tau}}_{\bm{x}}[\mathcal{F}_h]$'s, 
restrictions on one also restrict others. 
For example, balancing $\hat{\bm{\tau}}_{\bm{x}}[\mathcal{F}_1]$ partially balances
$\hat{\bm{\tau}}_{\bm{x}}[\mathcal{F}_2]$. 
Therefore, instead of unnecessarily balancing for all factorial effects in tier $h$, 
we balance only the part that is orthogonal to the factorial effects in previous tiers.

Let $\covcoef=2^{-2(K-1)}\sum_{q=1}^{Q} n_q^{-1}\coef_q \coef_q'$.
From Proposition \ref{prop:mean_var_cre}, the sampling covariance  of $\hat{\bm{\tau}}_{\bm{x}}$ under the CRFE, 
$\bm{V}_{\bm{xx}}= \covcoef \otimes \bm{S}_{\bm{xx}}$, 
contains two components: 
$\covcoef$ determined by the coefficient vector $\bm{b}_q$'s 
and $\bm{S}_{\bm{xx}}$ determined by the covariates. 
Below we introduce a block-wise Gram--Schmidt orthogonalization of the coefficient vector $\coef_q$'s, taking into account the tiers of factorial effects. 
Let $\mathcal{F}_{\overline{h}} = \bigcup_{l=1}^h \mathcal{F}_{l}$ be the factorial effects in the first $h$ tiers. 
We use $\coef_q[\mathcal{F}_h]$ and 
$\coef_q[\mathcal{F}_{\overline{h}}]$ to denote the subvectors of $\coef_q$ with indices in $\mathcal{F}_h$ and $\mathcal{F}_{\overline{h}}$, 
and 
$\covcoef[\mathcal{F}_h, \mathcal{F}_{\overline{h}}]$ and $\covcoef[\mathcal{F}_{\overline{h}}, \mathcal{F}_{\overline{h}}]$
to denote the submatrices of $\covcoef$ with indices in $\mathcal{F}_h\times \mathcal{F}_{\overline{h}}$ and $\mathcal{F}_{\overline{h}}\times  \mathcal{F}_{\overline{h}}$. 
For each $1\leq q\leq Q$, we define 
the orthogonalized coefficient vector $\orthocoef_q = (\orthocoef_q'[1], \ldots, \orthocoef_q'[H])'$ as 
$\orthocoef_q[1] = \coef_q[\mathcal{F}_1],$ and
\begin{align}\label{eq:b_q}
\orthocoef_q[h] & = \coef_q[\mathcal{F}_h] - 
\covcoef[\mathcal{F}_h, \mathcal{F}_{\overline{h-1}}]
\left\{
\covcoef[\mathcal{F}_{\overline{h-1}}, \mathcal{F}_{\overline{h-1}}]
\right\}^{-1}
\coef_q[\mathcal{F}_{\overline{h-1}}], \quad (2\leq h\leq H). 
\end{align} 
The difference-in-means vector of covariates with respect to orthogonalized coefficient vectors is 
\begin{align}\label{eq:theta_X}
\hat{\bm{\theta}}_{\bm{x}} \equiv 
\begin{pmatrix}
\hat{\bm{\theta}}_{\bm{x}}[1]\\
\vdots\\
\hat{\bm{\theta}}_{\bm{x}}[H]
\end{pmatrix}
=
2^{-(K-1)} \sum_{q=1}^{Q} 
\begin{pmatrix}
\orthocoef_q[1]\\
\vdots\\
\orthocoef_q[H]
\end{pmatrix}
\otimes \hat{\bar{\bm{x}}}(q). 
\end{align}
By construction, $\tilde{\bm{C}} \equiv 2^{-2(K-1)} \sum_{q=1}^{Q} n_q^{-1}\orthocoef_q \orthocoef_q'$ is block diagonal, and thus  
the sampling covariance of $\hat{\bm{\theta}}_{\bm{x}}$ under the CRFE, $\tilde{\bm{C}} \otimes \bm{S}_{\bm{xx}}$, is also block diagonal. The following proposition summarizes these results. 

\begin{proposition}\label{prop:var_cov_tier_factor}
	Under the CRFE,
	$(\hat{\bm{\tau}}'-\bm{\tau}', \hat{\bm{\theta}}_{\bm{x}}')'$ has mean zero and sampling covariance:
	\begin{align*} 
	\Cov
	\left(\hat{\bm{\tau}}-\bm{\tau}, \hat{\bm{\theta}}_{\bm{x}}[h]
	\right) & \equiv \bm{W}_{\bm{\tau x}}[h] = 2^{-2(K-1)}\sum_{q=1}^{Q}
	n_q^{-1}(\coef_q\orthocoef'_q[h]) \otimes \bm{S}_{q, \bm{x}}, \\
	\Cov\left(\hat{\bm{\theta}}_{\bm{x}}[h]\right) & \equiv 
	\bm{W}_{\bm{xx}}[h] = 
	2^{-2(K-1)}\sum_{q=1}^{Q}
	n_q^{-1}(\orthocoef_q[h]\orthocoef'_q[h]) \otimes \bm{S}_{\bm{xx}}, \nonumber
	\end{align*}
	and $\Cov(\hat{\bm{\theta}}_{\bm{x}}[h], \hat{\bm{\theta}}_{\bm{x}}[\tilde{h}])=\bm{0}$ if $h\neq \tilde{h}$. 
\end{proposition}

Proposition \ref{prop:var_cov_tier_factor} follows from some algebra, with the proof   in Appendix A2 of the Supplementary Material \citep{rerandfacsupp}.
From Proposition \ref{prop:var_cov_tier_factor}, $(\hat{\bm{\theta}}_{\bm{x}}[1], \ldots, $ $\hat{\bm{\theta}}_{\bm{x}}[H])$ are mutually uncorrelated under the CRFE, and thus are essentially from a block-wise Gram--Schmidt orthogonalization of $(\hat{\bm{\tau}}_{\bm{x}}[\mathcal{F}_1], \ldots, \hat{\bm{\tau}}_{\bm{x}}[\mathcal{F}_H])$. 
We define the Mahalanobis distance in tier $h$ as
\begin{align}\label{eq:M_h_tier_factor}
M_h =  \hat{\bm{\theta}}_{\bm{x}}'[h]
\left(
\bm{W}_{\bm{xx}}[h]
\right)^{-1}
\hat{\bm{\theta}}_{\bm{x}}[h], \quad (1\leq h\leq H). 
\end{align}
Let $(a_1, \ldots, a_H)$ be $H$ positive constants predetermined in the design stage. Under rerandomization with tiers of factorial effects, denoted by $\text{ReFMT}_\text{F}$, a randomization is acceptable if and only if  $M_h\leq a_h$ for all $1\leq h\leq H$. 
Below we use $\mathcal{T}_{\text{F}}$ to denote the event that the treatment vector $\bm{Z}$ satisfies this criterion. From the finite population central limit theorem, asymptotically, $M_h$ is $\chi^2_{LF_h}$, and $(M_1, \ldots,M_H)$ are jointly independent. Therefore, the asymptotic acceptance probability under $\text{ReFMT}_\text{F}$ is $p_{a} = \prod_{h=1}^{H}P(\chi^2_{LF_h}\leq a_h).$ 
We usually choose small $a_h$'s. The relative magnitude of $a_h$'s depend on our prior knowledge or belief on the relative importance of covariates in all tiers. See \citet{morgan2015rerandomization} for a more detailed discussion.

With equal treatment groups sizes, $M_h$ has simpler form. 
\begin{proposition}\label{prop:simpler_refmt_f}
	When $n_1=\cdots=n_Q=n/Q$, the coefficient $\orthocoef_q[h]$ in \eqref{eq:b_q} reduces to $\coef_q[\mathcal{F}_h]$, 
	the difference-in-means of covariates $\hat{\bm{\theta}}_{\bm{x}}[h]$ in \eqref{eq:theta_X} reduces to $\hat{\bm{\tau}}_{\bm{x}}[\mathcal{F}_h]$, 
	and the Mahalanobis distance $M_h$ in \eqref{eq:M_h_tier_factor} reduces to 
	$M_h = n/4 \cdot
	\sum_{f\in \mathcal{F}_h} 
	\hat{\bm{\tau}}_{\bm{x},f}'
	\bm{S}_{\bm{xx}}^{-1} \hat{\bm{\tau}}_{\bm{x},f}.$ 
\end{proposition}

Proposition \ref{prop:simpler_refmt_f} follows from some algebra, with the proof   in Appendix A2 of the Supplementary Material \citep{rerandfacsupp}.
In Proposition \ref{prop:simpler_refmt_f}, if further each tier contains exactly one factorial effect, $\text{ReFMT}_\text{F}$ reduces to the rerandomization scheme discussed in \citet{branson2016improving}.

\subsection{Asymptotic sampling distribution of $\hat{\boldsymbol{\tau}}$}\label{sec:asym_dist_refmt_f}
In this subsection, we study the asymptotic sampling distribution of $\hat{\bm{\tau}}$ under $\text{ReFMT}_\text{F}$.
Let $\bm{W}_{\bm{\tau\tau}}^\myparallel[h]=\bm{W}_{\bm{\tau x}}[h](
\bm{W}_{\bm{xx}}[h]
)^{-1}\bm{W}_{\bm{x\tau}}[h]$ be the sampling covariance matrix of $\hat{\bm{\tau}}$ explained by $\hat{\bm{\theta}}_{\bm{x}}[h]$ in the linear projection under the CRFE. 
Extending earlier notation, 
let $\bm{\zeta}_{LF_h,a_h}\sim \bm{D}_h \mid \bm{D}_h' \bm{D}_h \leq a_h$ be a truncated Gaussian vector with $LF_h$ dimensions, where $\bm{D}_h = (D_{h1}, \ldots, D_{h,LF_h})'\sim \mathcal{N}(\bm{0}, \bm{I}_{LF_h}).$

\begin{theorem}\label{thm:asymp_refmt_f}
	Under $\text{ReFMT}_\text{F}$ and Condition \ref{cond:fp},
	\begin{eqnarray}\label{eq:asymp_joint_refmt}
	\hat{\bm{\tau}} - \bm{\tau}  \mid \mathcal{T}_{\text{F}} 
	& \apprsim & \left( \bm{V}_{\bm{\tau}\bm{\tau}}^\myperp \right)^{1/2} \bm{\varepsilon} + \sum_{h=1}^H
	\left(\bm{W}_{\bm{\tau\tau}}^\myparallel[h] \right)^{1/2}_{LF_h}
	\bm{\zeta}_{LF_h, a_h},
	\end{eqnarray}
	where $(\bm{\varepsilon}, \bm{\zeta}_{LF_1, a_1}, \ldots, \bm{\zeta}_{LF_H, a_H})$ are jointly independent. 
\end{theorem}

The proof of Theorem \ref{thm:asymp_refmt_f}, similar to that of Theorem \ref{thm:joint_refm}, is in Appendix A3 of the Supplementary Material \citep{rerandfacsupp}.

Let 
$W_{\tau_f\tau_f}^\myparallel[h]$ be the $f$th diagonal element of $\bm{W}_{\bm{\tau\tau}}^\myparallel[h]$. 
The squared multiple correlation between $\hat{\tau}_f$ and $\hat{\bm{\theta}}_{\bm{x}}[h]$ under the CRFE is then  $\rho^2_{f}[h]=W_{\tau_f\tau_f}^\myparallel[h]/V_{\tau_f\tau_f}$. When treatment group sizes are equal, $\rho^2_f[h]$ reduces to 
$\rho^2_{f}[h] = \sum_{k:k\in \mathcal{F}_h}\gamma^2_{fk}$ for all $f$; if further the additivity holds, $\rho^2_{f}[h]$ reduces to $S_{11}^{\myparallel}/S_{11}$ if $f\in \mathcal{F}_h$, and zero otherwise. 
Because the $\hat{\bm{\theta}}_{\bm{x}}[h]$'s are from a block-wise Gram--Schmidt orthogonalization of $\hat{\bm{\tau}}_{\bm{x}}$, the squared multiple correlation between $\hat{\tau}_f$ and $\hat{\bm{\tau}}_{\bm{x}}$ can be decomposed as $R_f^2 = \sum_{h=1}^{H}\rho^2_{f}[h]$. 
The following corollary shows the marginal asymptotic sampling distribution of a single factorial effect estimator. Let ${\eta}_{LF_h,a_h} \sim D_{h1} \mid \bm{D}_h'\bm{D}_h \leq a_h$ be the first coordinate of $\bm{\zeta}_{LF_h,a_h}$.

\begin{corollary}\label{cor:dist_remt}
	Under $\text{ReFMT}_\text{F}$ and Condition \ref{cond:fp}, for $1\leq f\leq F$, 
	\begin{equation}\label{eq:dist_remt}
	\hat{\tau}_f - \tau_f \mid  \mathcal{T}_{\text{F}} 
	\  \apprsim \ 
	\sqrt{V_{\tau_f\tau_f}}
	\left(
	\sqrt{1-R^2_f} \cdot \varepsilon_0 + \sum_{h=1}^H \sqrt{\rho_{f}^2[h]} \cdot \eta_{LF_h,a_h}
	\right),
	\end{equation}
	where $(\varepsilon_0, \eta_{LF_1,a_1}, \ldots, \eta_{LF_H,a_H})$ are jointly independent. 
\end{corollary}

The proof of Corollary \ref{cor:dist_remt} is in Appendix A3 of the Supplementary Material \citep{rerandfacsupp}.

\subsection{Asymptotic unbiasedness, sampling covariance and peakedness}\label{sec:property_refmt_f}
Based on the asymptotic distributions in Section \ref{sec:asym_dist_refmt_f}, we study the asymptotic properties of the factorial effect estimators. 
First, $(\bm{\varepsilon}, \bm{\zeta}_{LF_1,a_1}, \ldots, \bm{\zeta}_{LF_H,a_H})$ are all central convex unimodal from Proposition \ref{prop:ccu_refm}, 
and thus the asymptotic sampling distribution \eqref{eq:asymp_joint_refmt} of $\hat{\bm{\tau}}$ under $\text{ReFMT}_\text{F}$ is also central convex unimodal. The symmetry of the asymptotic sampling distributions ensures that the factorial effect estimator
$\hat{\bm{\tau}}$ is consistent for $\bm{\tau}$ under $\text{ReFMT}_\text{F}$, which implies that the difference-in-means of any observed or unobserved covariate with respect to any factorial effect has asymptotic mean zero.

Second, we compare the asymptotic sampling covariance matrices of $\hat{\bm{\tau}}$ under $\text{ReFMT}_\text{F}$ and the CRFE. 
For each $1\leq h\leq H$, 
let $v_{LF_h, a_h}=P(\chi^2_{LF_h+2}\leq a_h)/P(\chi^2_{LF_h}\leq a_h)\leq 1$.

\begin{theorem}\label{thm:red_var_remft_f}
	Under Condition \ref{cond:fp}, 
	$\hat{\bm{\tau}}$ has smaller asymptotic sampling covariance under $\text{ReFMT}_\text{F}$ than that under the CRFE, and the reduction in asymptotic sampling covariance is 
	$
	n\sum_{h=1}^{H}(1-v_{LF_h,a_h}) \bm{W}_{\bm{\tau\tau}}^\myparallel[h]. 
	$
	Specifically,  
	for each $1\leq f\leq F,$ the PRIASV of $\hat{\tau}_f$ is 
	$
	\sum_{h=1}^H (1-v_{LF_h, a_h})\rho_{f}^2[h]. 
	$
\end{theorem}

Theorem \ref{thm:red_var_remft_f} follows from Theorem \ref{thm:asymp_refmt_f} and Proposition \ref{prop:represent}, with the proof   in Appendix A4 of the Supplementary Material \citep{rerandfacsupp}. When the threshold $a_h$'s are close to zero, the asymptotic sampling variance of $\hat{\tau}_f$ reduces to $V_{\tau_f\tau_f}(1-R^2_f)$, which is identical to the asymptotic sampling variance of the regression adjusted estimator under the CRFE \citep{lu2016covadj}. 

Third, we compare the peakedness of asymptotic sampling distributions of $\hat{\bm{\tau}}$ under $\text{ReFMT}_\text{F}$ and the CRFE. 
\begin{theorem}\label{thm:quant_region_refmt_f}
	Under Condition \ref{cond:fp}, the asymptotic sampling distribution of $\hat{\bm{\tau}}-\bm{\tau}$ under $\text{ReFMT}_\text{F}$ is more peaked than that under the CRFE. 
\end{theorem}

The proof of Theorem \ref{thm:quant_region_refmt_f}, similar to that of Theorem \ref{thm:red_qr_rem}, is in Appendix A5 of the Supplementary Material \citep{rerandfacsupp}. We then consider the specific symmetric convex set $\mathcal{O}(\bm{V}_{\bm{\tau\tau}}, c)$. Unfortunately, considering joint quantile region for $\bm{\tau}$ is technically challenging in general, and we consider the case where the following condition holds.

\begin{condition}\label{cond:simultaneous_ortho}
	There exists an orthogonal matrix $\bm{\Gamma}\in \mathbb{R}^{F\times F}$ such that 
$$
	\bm{\Gamma}' 
	\bm{V}_{\bm{\tau\tau}}^{-1/2}
	\bm{W}_{\bm{\tau\tau}}^\myparallel[h]
	\bm{V}_{\bm{\tau\tau}}^{-1/2} \bm{\Gamma}=\textup{diag}(\omega_{h1}^2, \ldots, \omega_{hF}^2), 
	\quad 
	(1\leq h\leq H)
$$
	where $(\omega_{h1}^2, \ldots, \omega_{hF}^2)$ are the canonical correlations between $\hat{\bm{\tau}}$ and $\hat{\bm{\theta}}_{\bm{x}}[h]$ under the CRFE. 
\end{condition}

Condition \ref{cond:simultaneous_ortho} holds automatically when $H=1$. 
Moreover, the additivity in Definition \ref{def:additive} implies Condition \ref{cond:simultaneous_ortho} for general $H\geq 1$. The following proposition states this result. 
Let $\bm{\Psi}\in \mathbb{R}^{F\times F}$ be the common linear transformation matrix from $\orthocoef_q$ to $\coef_q$, i.e., 
$\coef_q = \bm{\Psi}\orthocoef_q$ for all $1\leq q\leq Q$. 
Recall that 
$\tilde{\bm{B}} = 2^{-2(K-1)}\sum_{q=1}^{Q} n_q^{-1}\coef_q \coef_q'$,  
and 
$\tilde{\bm{C}} = 2^{-2(K-1)}\sum_{q=1}^{Q} n_q^{-1}\orthocoef_q \orthocoef_q'$. 

\begin{proposition}\label{prop:add_simu_diag}
	Under the additivity in Definition \ref{def:additive}, Condition \ref{cond:simultaneous_ortho} holds with orthogonal matrix $\bm{\Gamma} = \tilde{\bm{B}}^{-1/2}\bm{\Psi}  \tilde{\bm{C}}^{1/2}$, and the canonical correlations between $\hat{\bm{\tau}}$ and $\hat{\bm{\theta}}_{\bm{x}}[h]$ have exactly $F_h$ nonzero elements, which are all equal to $S_{11}^\myparallel/S_{11}$. 
\end{proposition}

Proposition \ref{prop:add_simu_diag} follows from some algebra, with the proof in Appendix A5 of the Supplementary Material \citep{rerandfacsupp}.

\begin{theorem}\label{thm:quant_region_refmt_f_monotone_joint}
	Under $\text{ReFMT}_\text{F}$, assume  that Conditions \ref{cond:fp} and  \ref{cond:simultaneous_ortho} hold. Let
 $c_{1-\alpha}$ be the solution of
$
	\lim_{n\rightarrow\infty}P\left\{
	\hat{\bm{\tau}} - \bm{\tau} \in \mathcal{O}(\bm{V}_{\bm{\tau\tau}}, c_{1-\alpha}) \mid \mathcal{T}_{\text{F}} 
	\right\} = 1 - \alpha .
$
It	depends only on  $L$, $F_h$'s, $a_h$'s, and $(\omega_{h1}^2, \ldots, \omega_{hF}^2)$'s, 
	and is nonincreasing in $\omega_{hf}^2$ for $1\leq h \leq H$ and $1\leq f\leq F$. 
\end{theorem}

The proof of Theorem \ref{thm:quant_region_refmt_f_monotone_joint}, similar to that of Theorem \ref{thm:red_qr_rem_nond_ccorr}, is in Appendix A5 of the Supplementary Material \citep{rerandfacsupp}.

Because the peakedness relationship is invariant under linear transformations, and any linear transformation of $\hat{\bm{\tau}}$ has an asymptotic sampling distribution of the same form as $\hat{\bm{\tau}}$,  
we can establish similar conclusions as Theorems \ref{thm:quant_region_refmt_f} and \ref{thm:quant_region_refmt_f_monotone_joint} for any linear transformations of $\hat{\bm{\tau}}$. 
We relegate the details to the Supplementary Material \citep{rerandfacsupp}, and consider only the asymptotic sampling distribution of a single factorial effect estimator below.

\begin{corollary}\label{cor:quant_region_refmt_f_monotone_marg}
	Under Condition \ref{cond:fp}, for any $1\leq f\leq F$ and $\alpha\in (0,1)$, the threshold $c_{1-\alpha}$ for
	$1-\alpha$ asymptotic symmetric quantile range  
	$[-c_{1-\alpha}V_{\tau_f\tau_f}^{1/2}, c_{1-\alpha}V_{\tau_f\tau_f}^{1/2}]$
	of $\hat{\tau}_f-\tau_f$ under $\text{ReFMT}_\text{F}$ is smaller than or equal to that under the CRFE, and is nonincreasing in $\rho^2_f[h]$ for $1\leq h\leq H$.
\end{corollary}

The proof of Corollary \ref{cor:quant_region_refmt_f_monotone_marg} is in Appendix A5 of the Supplementary Material \citep{rerandfacsupp}.
From Corollary \ref{cor:quant_region_refmt_f_monotone_marg}, with larger squared multiple correlation $\rho^2_f[h]$, $\text{ReFMT}_\text{F}$ yields more percentage reductions of quantile ranges. 

The example below shows the advantage of $\text{ReFMT}_\text{F}$ over ReFM.

\begin{example}\label{eg:simple_refmt_f}
	We consider experiments with $K$ factors and $L$ dimensional covariates. Assume the additivity, which implies that 
	$R_f^2$ is the same for all factorial effects $f$. 
	Suppose that we are more interested in the $K$ main effects than the interaction effects.
	We divide the $F$ effects into $2$ tiers, where tier 1 contains the $F_1 = K$ main effects and tier 2 contains the remaining $F_2 = 2^K-1-K$ interaction effects. 
	From Proposition \ref{prop:comp_R_fm_add}, we can derive
	$\rho^2_k[1]=R_f^2$ and $\rho^2_k[2]=0$
	for the main effect $1\leq k\leq K$. 
	We compare two rerandomization schemes with the same asymptotic acceptance probability: ReFM with $p_{a} = 0.001$ and $\text{ReFMT}_\text{F}$ with thresholds $(a_1,a_2)$ satisfying $P(\chi^2_{LF_1}\leq a_1) = 0.002$ and $P(\chi^2_{LF_2}\leq a_2)=0.5$. 
	Figure \ref{fig:simple_compare_refmt_f} shows 
	the PRIASV, divided by $R^2_f$, of the main effect estimators for both rerandomization schemes. It shows that the advantage of $\text{ReFMT}_\text{F}$ increases as the numbers of factors and covariates increase. 
	$\hfill\Box$ 
	\begin{figure}[htb]
		\centering
		\includegraphics[width=0.5\textwidth]{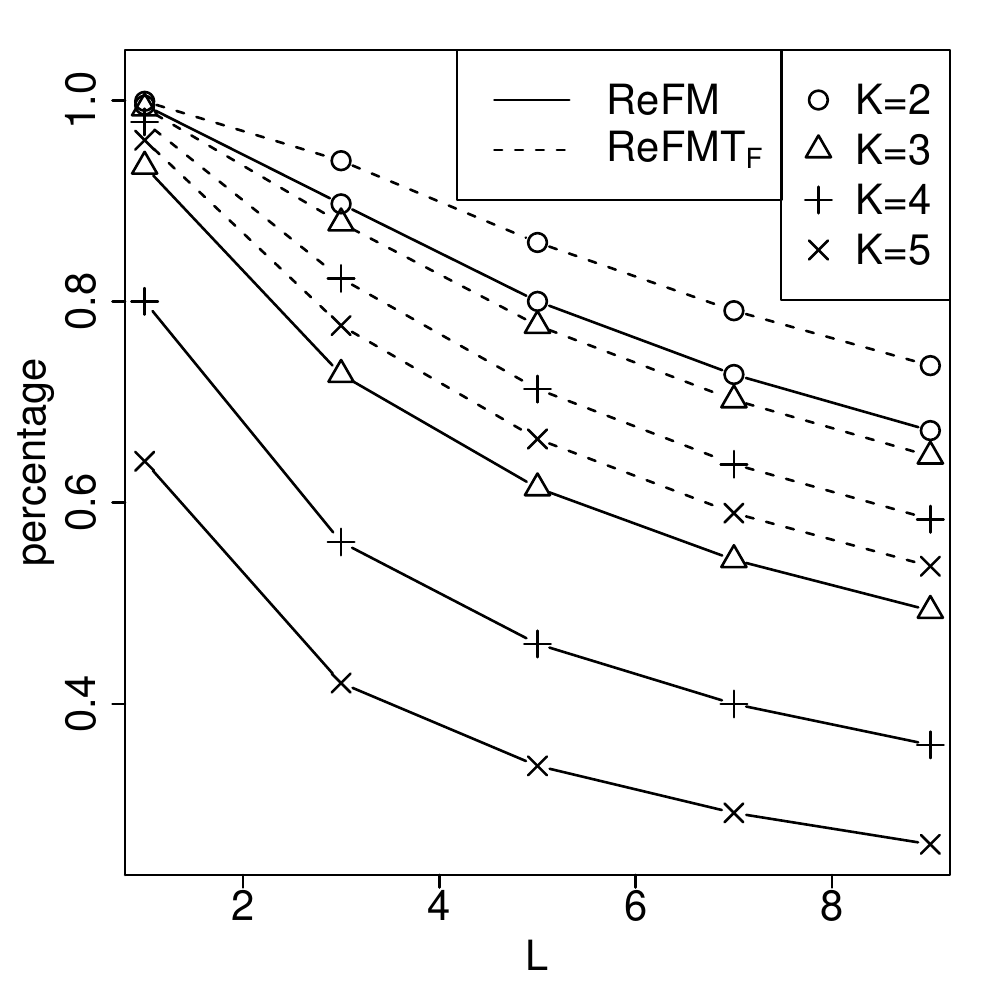}
		\caption{PRIASV of main effect estimators divided by $R^2_f$}\label{fig:simple_compare_refmt_f}
	\end{figure}
\end{example}

\subsection{Conservative covariance estimator and confidence sets under $\text{ReFMT}_\text{F}$}\label{sec:var_est_refmt_f}
We estimate $\bm{V}_{\bm{\tau\tau}}^\myperp$ by 
$\hat{\bm{V}}_{\bm{\tau\tau}}^\myperp$ 
in \eqref{eq:V_tautau_perp_est}, $\bm{W}_{\bm{\tau x}}[h]$ by 
$
\hat{\bm{W}}_{\bm{\tau x}}[h] = 2^{-2(K-1)}\sum_{q=1}^{Q}
n_q^{-1}(\coef_q\orthocoef'_q[h]) \otimes \{\bm{s}_{q,\bm{x}}\bm{s}_{\bm{xx}}^{-1/2}(q)\bm{S}_{\bm{xx}}^{1/2}\},
$
and $
(\bm{W}_{\bm{\tau\tau}}^\myparallel[h] )^{1/2}_{LF_h}$ 
by
$\hat{\bm{W}}_{\bm{\tau x}}[h]
(
\bm{W}_{\bm{xx}}[h]
)^{-1/2}$.
We can then obtain a covariance estimator and 
construct confidence sets for $\bm{\tau}$ or its linear transformations. 
Similar to ReFM, for a parameter of interest $\bm{C}\bm{\tau}$, we recommend confidence sets of the form $
\bm{C}\hat{\bm{\tau}} + \mathcal{O}(\bm{C}\hat{\bm{V}}_{\bm{\tau\tau}}^\myperp\bm{C}', c),   
$  
where we choose the threshold $c$ by simulating random draws from the estimated asymptotic sampling distribution. 
Let 
$\hat{c}_{1-\alpha}$ be the $1-\alpha$ quantile of 
$ (\bm{C}\bm{\phi})'(\bm{C}\hat{\bm{V}}_{\bm{\tau\tau}}^\myperp\bm{C}')^{-1}(\bm{C}\bm{\phi})
$
with $\bm{\phi}$ following the estimated asymptotic sampling distribution of $\hat{\bm{\tau}}-\bm{\tau}$ under $\text{ReFMT}_\text{F}.$

\begin{theorem}\label{thm:conserv_cs_remft_f}
Under $\text{ReFMT}_\text{F}$ and Condition \ref{cond:fp},  consider inferring $\bm{C}\bm{\tau}$, where
 $\bm{C}$ has full row rank. The probability limit of the covariance estimator,
	$
	\bm{C}\hat{\bm{V}}_{\bm{\tau\tau}}^\myperp\bm{C}' + \sum_{h=1}^{H} v_{LF_h, a_h} \bm{C}\hat{\bm{W}}_{\bm{\tau x}}[h]
	(
	\bm{W}_{\bm{xx}}[h]
	)^{-1}
	\hat{\bm{W}}_{\bm{x\tau}}[h]\bm{C}', 
	$
	for $\bm{C}\hat{\bm{\tau}}$ is larger than or equal to the actual sampling covariance, and 
	the $1-\alpha$ confidence set, 
	$
	\bm{C}\hat{\bm{\tau}} + \mathcal{O}(\bm{C}\hat{\bm{V}}_{\bm{\tau\tau}}^\myperp\bm{C}', \hat{c}_{1-\alpha}), 
	$
	has asymptotic coverage rate $\geq 1-\alpha$, 
	with equality holding if $\bm{S}_{\bm{\tau\tau}}^\myperp\rightarrow 0$ as $n\rightarrow \infty$. 
\end{theorem}

The proof of Theorem \ref{thm:conserv_cs_remft_f}, similar to that of Theorem \ref{thm:conser_conf_set_refm}, is in Appendix A6 of the Supplementary Material \citep{rerandfacsupp}. The above confidence sets will be similar to the ones based on regression adjustment if the threshold $a_h$'s are small \citep{lu2016covadj}.
Moreover,  we can also extend Theorem \ref{thm:conserv_cs_remft_f} to general symmetric convex confidence sets \citep{rerandfacsupp}.

\section{An education example}\label{sec:edu_example}

We illustrate the theory of rerandomization using a dataset from the  Student Achievement and Retention Project \citep{angrist2009incentives}, a $2^2$ CRFE at one of the satellite campuses of a large Canadian university. One treatment factor is the Student Support Program (SSP), which provides students some services for study. The other treatment factor is the Student Fellowship Program (SFP), which awards students scholarships for achieving a target first year grade point average (GPA). There were 1006 students in the control group receiving neither SSP nor SFP (i.e., $(-1,-1)$), 250 students offered only SFP (i.e., $(-1,+1)$), 250 students offered only SSP (i.e., $(+1,-1)$), and 150 students offered both SSP and SFP (i.e., $(+1,+1)$). 
We include $L=5$ pretreatment covariates: high school GPA, gender, age, indicators for whether the student was living at home and whether the student rarely put off studying for tests, and exclude students with missing covariate values. This results in treatment groups of sizes $(856, 216, 208, 118)$ for treatment combinations $(-1,-1), (-1,+1), (+1,-1)$ and $(+1,+1)$, respectively.

To demonstrate the advantage of rerandomization, we compare the CRFE and $\text{ReFMT}_\text{F}$
in terms of the sampling distributions of the factorial effects estimator.
However, the sampling distributions depend on all the potential outcomes including the missing ones. 
To make the simulation more realistic, we impute all of the missing potential outcomes based on simple model fitting. 
Specifically, 
we fit a linear regression of the observed GPA on the levels of two treatment factors, all covariates and the interactions between these covariates, and then impute all the missing potential outcomes based on the fitted model. 
We further truncate all the potential outcomes to $[0,4]$ to mimic the values of GPA. 
Note that the generating models for the missing potential outcomes are not linear in the covariates. 
For the simulated data set, the sampling squared multiple correlations between factorial effect estimators and the difference-in-means of covariates are $(R_1^2, R_2^2, R_3^2)=(0.247,  0.244, 0.245)$.

We divide the three factorial effects into two tiers, where tier 1 contains $F_1=2$ main effects, and tier 2 contains $F_2=1$ interaction effect, and choose thresholds $(a_1,a_2)$ such that 
$P(\chi^2_{LF_1}\leq a_1) = 0.002$ and $P(\chi^2_{LF_2}\leq a_2)=0.5$.
Table \ref{tab:star_comp_refmt_f} shows the empirical and theoretical percentage reductions in the sampling variances and the lengths of $95\%$ symmetric quantile ranges for the three factorial effect estimators under $\text{ReFMT}_\text{F}$, compared to the CRFE. 
From Table \ref{tab:star_comp_refmt_f}, the asymptotic approximations work fairly well, and 
$\text{ReFMT}_\text{F}$ improves the precision of the two average main effects estimators more than that of the average interaction effect estimator.

\begin{table}
	\centering
	\caption{Comparison of the factorial effect estimators between the CRFE and $\text{ReFMT}_\text{F}$. The second and third columns show the percentage reductions in variances, and the fourth and fifth columns show the percentage reductions in the lengths of quantile ranges.}\label{tab:star_comp_refmt_f}
	\begin{tabular}{cccccc}
		\toprule
		Factorial effect & \multicolumn{2}{c}{Reduction in variance}  & & \multicolumn{2}{c}{Reduction in $95\%$ quantile range}\\
		\cline{2-3} \cline{5-6}
		&  empirical & theoretical & &  empirical & theoretical\\ 
		\midrule 
		Main effect of SSP & $20.2\%$ & $21.2\%$ & &  $10.7\%$ & $11.2\%$\\
		Main effect of SFP & $20.4\%$ & $20.9\%$ & & $10.8\%$ & $11.1\%$\\
		Interaction effect & $14.4\%$ & $14.9\%$ & & $7.7\%$ & $7.8\%$\\
		\bottomrule
	\end{tabular}
\end{table}

We then consider confidence sets for the two average main effects $(\tau_1, \tau_2)$ under both designs. The empirical coverage probabilities of $95\%$ confidence sets discussed in Sections \ref{sec:asym_dist_crfe} and \ref{sec:var_est_refmt_f} under the CRFE and $\text{ReFMT}_\text{F}$ are, respectively, $96.4\%$ and $96.5\%$, showing that both confidence sets are slightly conservative. Moreover, 
the percentage reduction in the average volume of $95\%$ confidence sets under $\text{ReFMT}_\text{F}$ compared to the CRFE is $20.5\%$, and the corresponding percentage increase in sample size needed for the CRFE to obtain $95\%$ confidence set of the same average volume as $\text{ReFMT}_\text{F}$ is about $25.8\%$.  

To end this section, we investigate the dependence of the PRIASVs on the choices of thresholds 
$(a_1, a_2)$. Let $p_{ah} \equiv P(\chi^2_{LF_h}\leq a_h)$ be the asymptotic acceptance probability for tier $h$ $(h=1,2)$. 
Fixing the overall asymptotic acceptance probability $p_a \equiv p_{a1}p_{a2}$ at $0.001$, 
Figure \ref{fig:choice_a} shows the PRIASVs of all factorial effect estimators as functions of $p_{a1}$. 
We can see that (1) more stringent restrictions on the first tier of factorial effects (i.e., the two main effects) lead to larger PRIASVs of the corresponding estimators, but (2) the PRIASV of the estimator of the second tier of factorial effect (i.e., the interaction effect) is a non-monotone function of $p_{a1}$. Therefore, in practice, we are facing a trade-off, which depends on the a priori relative importance of the factorial effects.

\begin{figure}[htb]
	\centering
	\includegraphics[width=0.5\textwidth]{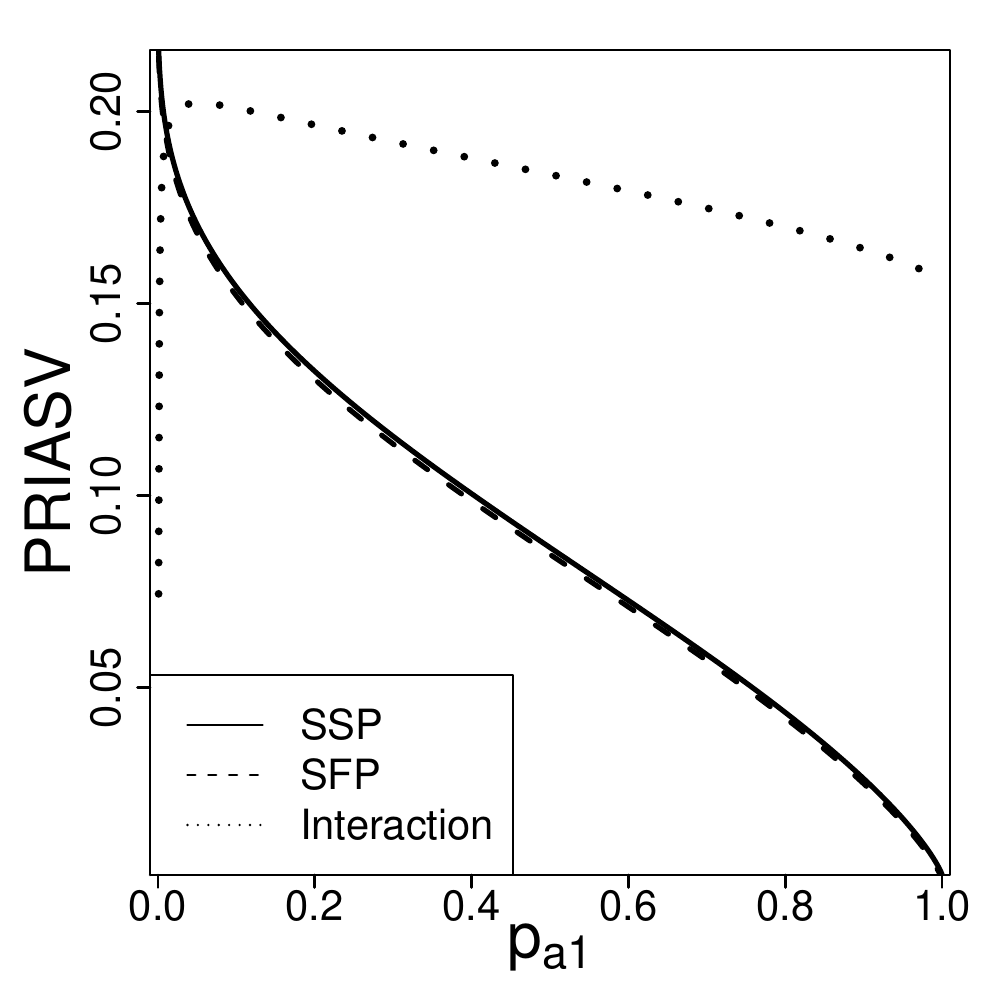}
	\caption{PRIASVs of all factorial effects as $p_{a1}$, with the overall acceptance probability $p_a = p_{a1}p_{a2} $ fixed at 0.001.}\label{fig:choice_a}
\end{figure}

\section{Extensions}\label{sec:extension}
When covariates have varying importance for the potential outcomes, we can further consider balance criterion using tiers of covariates, i.e., rerandomized factorial experiments with tiers of both covariates and factorial effects. We discuss in detail this balance criterion in the Supplementary Material \citep{rerandfacsupp}. There we study the asymptotic sampling distribution and properties of the difference-in-means estimator, and the repeated sampling inference for the average factorial effects. We demonstrate that the advantage of further considering tiers of covariates can increase with the numbers of covariates and factors.

We focused on improving completely randomized factorial experiments using rerandomization. Our future work will focus on using rerandomization to ensure covariate balance in more complex experiments, for example, randomized block, Latin squares and split-plot designs \citep{cox2000theory, Hinkelmann2007, zhao2018randomization}.

\section*{Acknowledgments}
We thank the Associate Editor and two reviewers for helpful comments. Zach Branson at Harvard made constructive suggestions on an early version of our paper.

\begin{supplement}
	\textbf{Supplementary Material to  ``Rerandomization in $\boldsymbol{2^K}$ Factorial Experiments''}
	\slink[url]{DOI: 10.1214/00-AOSXXXXSUPP; .pdf}
	\sdescription{
		We study the theoretical properties of $2^K$ rerandomized factorial experiments with tiers of both covariates and factorial effects, and prove all the theorems, corollaries and propositions 
		in the Supplementary Material \citep{rerandfacsupp}.  
	}
\end{supplement}

\bibliographystyle{plainnat}
\bibliography{causal}

\newpage
\setcounter{page}{1}
\begin{center}
	\bf \huge 
	Supplementary Material
\end{center}

\bigskip

\setcounter{equation}{0}
\setcounter{section}{0}
\setcounter{figure}{0}
\setcounter{example}{0}
\setcounter{proposition}{0}
\setcounter{corollary}{0}
\setcounter{theorem}{0}
\setcounter{table}{0}

\renewcommand {\theproposition} {A\arabic{proposition}}
\renewcommand {\theexample} {A\arabic{example}}
\renewcommand {\thefigure} {A\arabic{figure}}
\renewcommand {\thetable} {A\arabic{table}}
\renewcommand {\theequation} {A\arabic{equation}}
\renewcommand {\thelemma} {A\arabic{lemma}}
\renewcommand {\thesection} {A\arabic{section}}
\renewcommand {\thetheorem} {A\arabic{theorem}}
\renewcommand {\thecorollary} {A\arabic{corollary}}
\renewcommand {\thecondition} {A\arabic{condition}}

Appendix \ref{app:refmt_cf} studies the theoretical properties of $2^K$ rerandomized factorial experiments with tiers of both covariates and factorial effects. 

Appendix \ref{app:covariance} shows the sampling covariances, asymptotic sampling distributions, and squared multiple correlations under the CRFE. 
It includes the proofs of 
Theorem \ref{thm:joint_var_expl_unexpl}, 
Corollary \ref{cor:R2_f}, 
and 
Propositions \ref{prop:mean_var_cre},  \ref{prop:comp_R_fm_add}--\ref{prop:simpler_refmt_f} and \ref{prop:var_cov_tier_factor_covariate}.

Appendix \ref{app:dist} proves the asymptotic sampling distributions of $\hat{\bm{\tau}}$ under ReFM, $\text{ReFMT}_\text{F}$ and $\text{ReFMT}_\text{CF}$. 
It includes the proofs of 
Theorems \ref{thm:joint_refm}, \ref{thm:asymp_refmt_f} and \ref{thm:asym_refmt_cf},  
Corollaries \ref{cor:dist_refm_one_linear}, \ref{cor:another_form_refm},  \ref{cor:dist_remt} and \ref{cor:single_refmt_cf}, 
and 
Propositions \ref{prop:ccu_refm} and \ref{prop:represent}.

Appendix \ref{app:red_cov} compares the asymptotic sampling covariances of $\hat{\bm{\tau}}$ under rerandomizations and the CRFE.  It includes the proofs of 
Theorems \ref{thm:red_var_rem}, \ref{thm:red_var_remft_f} and 
\ref{thm:red_var_refmt_cf}. 

Appendix \ref{app:peak} compares the peakedness of the asymptotic sampling distributions under rerandomizations and the CRFE. 
It includes the proofs of 
Theorems \ref{thm:red_qr_rem}, \ref{thm:red_qr_rem_nond_ccorr}, \ref{thm:quant_region_refmt_f} and  \ref{thm:quant_region_refmt_f_monotone_joint}, 
Corollaries \ref{cor:red_in_qr_single_remf} and \ref{cor:quant_region_refmt_f_monotone_marg}, 
and 
Propositions \ref{prop:peak_cov}, \ref{prop:add_simu_diag} and \ref{prop:add_simu_diag_refmt_cf}.

Appendix \ref{app:inference} proves the asymptotic conservativeness of covariance estimators and symmetric convex confidence sets under rerandomizations. 
It includes the proof of Theorem \ref{thm:conser_conf_set_refm}.

\section{Tiers of both covariates and factorial effects}\label{app:refmt_cf}
\subsection{Tiers and orthogonalized covariates}
When covariates have varying importance for the potential outcomes, \citet{morgan2015rerandomization} proposed rerandomization with tiers of covariates in treatment-control experiments. It is important to consider the tiers of covariates with tiers of factorial effects. We first partition the covariates into $T$ tiers with decreasing importance, and use $\bm{x}_i[t]$ to denote the $L_t$ dimensional covariates in tier $t$.   
Let 
$\bm{x}_i[\overline{t}]=(\bm{x}_i[1], \ldots, \bm{x}_i[t])$ be the covariates in the first $t$ tiers, 
$\bm{S}_{\bm{x}[\overline{t}]\bm{x}[\overline{t}]}$ be the finite population covariance of the covariates in the first $t$ tiers, and $\bm{S}_{\bm{x}[t], \bm{x}[\overline{t-1}]}$ be the finite population covariance between $\bm{x}[t]$ and $\bm{x}[\overline{t-1}]$.
We then apply a block-wise Gram--Schmidt orthogonalization to the  $\bm{x}_i[t]$'s: $\bm{e}_i[1] = \bm{x}_i[1]$, and 
\begin{eqnarray*}
	\bm{e}_i[t] = \bm{x}_i[t] - \bm{S}_{\bm{x}[t], \bm{x}[\overline{t-1}]}  \bm{S}_{\bm{x}[\overline{t-1}]\bm{x}[\overline{t-1}]}^{-1}\bm{x}_i[\overline{t-1}], \quad ( 2\leq t\leq \Tau ) 
\end{eqnarray*}
where $\bm{e}_i[t]$ is the residual from the linear projection of the covariates $\bm{x}_i[t]$ in tier $t$ onto the space spanned by the covariates in previous tiers. \citet{asymrerand2016} call $\bm{e}_i[t]$ the orthogonalized covariates in tier $t$.

\subsection{Tiers of covariates and factorial effects criterion}
Using the notation in Section \ref{sec:ReFTF}, we partition the $F$ factorial effects into $H$ tiers with decreasing importance, i.e., $\mathcal{F} = \bigcup_{h=1}^H \mathcal{F}_t,$ 
and the corresponding block-wise Gram--Schmidt orthogonalization of $(\coef_q[\mathcal{F}_1], \ldots, \coef_q[\mathcal{F}_H])$ is $(\orthocoef_q[1], \ldots, \orthocoef_q[H])$ defined in \eqref{eq:b_q}. 
Let  
$\hat{\bar{{\bm{e}}}}_{[t]}(q)=n_q^{-1}\sum_{i:Z_i=q}\bm{e}_i[t]$ be the mean of the orthogonalized covariates in tier $t$ under treatment combination $q$. 
The difference-in-means of orthogonalized covariates in tier $t$ with coefficients $\orthocoef_q$'s is 
\begin{align}\label{eq:theta_E}
\hat{\bm{\theta}}_{\bm{e}[t]} = 
\begin{pmatrix}
\hat{\bm{\theta}}_{\bm{e}[t]}[1]\\
\vdots\\
\hat{\bm{\theta}}_{\bm{e}[t]}[H]
\end{pmatrix}
=
2^{-(K-1)} \sum_{q=1}^{Q} 
\begin{pmatrix}
\orthocoef_q[1]\\
\vdots\\
\orthocoef_q[H]
\end{pmatrix}
\otimes \hat{\bar{\bm{e}}}_{[t]}(q), \quad (1\leq t\leq T).  
\end{align}
Let 
$\bm{S}_{\bm{e}[{t}]\bm{e}[{t}]}$ be the finite population covariance of the orthogonalized covariates in tier $t$, and $\bm{S}_{q, \bm{e}[{t}]}$ be the finite population covariance between the ${Y}_i(q)$'s and $\bm{e}_i[{t}]$'s.

\begin{proposition}\label{prop:var_cov_tier_factor_covariate}
	Under the CRFE, 
	$(\hat{\bm{\tau}}'-\bm{\tau}', \hat{\bm{\theta}}_{\bm{e}[1]}', \ldots, \hat{\bm{\theta}}_{\bm{e}[T]}')'$ has mean zero and
	sampling covariance:  
	\begin{align*} 
	\Cov
	\left(\hat{\bm{\tau}}-\bm{\tau}, \hat{\bm{\theta}}_{\bm{e}[t]}[h]
	\right) & = \bm{W}_{\bm{\tau e}[t]}[h] = 2^{-2(K-1)}\sum_{q=1}^{Q}
	n_q^{-1}(\coef_q\orthocoef'_q[h]) \otimes \bm{S}_{q, \bm{e}[t]}, \\
	\Cov\left(\hat{\bm{\theta}}_{\bm{e}[t]}[h]\right) & = 
	\bm{W}_{\bm{e}[t]\bm{e}[t]}[h] = 
	2^{-2(K-1)}\sum_{q=1}^{Q}
	n_q^{-1}(\orthocoef_q[h]\orthocoef'_q[h]) \otimes \bm{S}_{\bm{e}[t]\bm{e}[t]}, 
	\end{align*}
	and $\Cov(\hat{\bm{\theta}}_{\bm{e}[t]}[h], \hat{\bm{\theta}}_{\bm{e}[\tilde{t}]}[\tilde{h}])=\bm{0}$ if $t \neq \tilde{t}$ or $h \neq \tilde{h}$, i.e., $\hat{\bm{\theta}}_{\bm{e}[t]}[h]$'s are mutually uncorrelated. 
\end{proposition}

From Proposition \ref{prop:var_cov_tier_factor_covariate}, we define the Mahalanobis distance for orthogonalized covariates in tier $t$ with respect to factorial effects in tier $h$ as  
\begin{align*}
M_{t,h} = \hat{\bm{\theta}}_{\bm{e}[t]}'[h]
\left(\bm{W}_{\bm{e}[t] \bm{e}[t]}[h]\right)^{-1}
\hat{\bm{\theta}}_{\bm{e}[t]}[h], \quad (1\leq h 
\leq H, 1\leq t\leq T). 
\end{align*}
Table \ref{tab:tiers_cov_fac} displays the Mahalanobis distances for all tiers of covariates and factorial effects.
As discussed in Section \ref{sec:ReFTF}, we should balance more the orthogonalized covariates with respect to more important factorial effects, e.g., putting decreasing restrictions on $M_{t,1}, M_{t,2}, \ldots, M_{t,H}$ for a given $t$.  
With the pre-assumed importance of covariates, we should also balance more the covariates in more important tiers, e.g., putting decreasing restrictions on $M_{1,h}, M_{2,h}, \ldots, M_{T,h}$ for a given $h$. 
Generally, the restrictions on the Mahalanobis distances in Table \ref{tab:tiers_cov_fac} should decrease from left to right and from top to bottom, i.e.,
if $h\leq \tilde{h}$ and $t\leq \tilde{t}$, we should put more restriction on $M_{t,h}$ than that on $M_{\tilde{t}, \tilde{h}}$. This implies a partial order of importance on the set $\mathcal{S}=\{(t,h): 1\leq t\leq T, 1\leq h\leq H\}$ of all combinations of tiers of covariates and factorial effects.
In practice, we can divide the set $\mathcal{S}$ into $J$ tiers $(\mathcal{S}_1, \ldots, \mathcal{S}_J)$ with decreasing importance, which are coherent with the partial order on $\mathcal{S}.$ 
\begin{table}
	\centering
	\caption{Mahalanobis distances: covariates and factorial effects in different tiers}
	\label{tab:tiers_cov_fac}
	\begin{tabular}{cc|cccc}
		& & \multicolumn{4}{c}{Factorial effects}\\
		& & Tier 1 & Tier 2 & $\cdots$ & Tier $H$\\
		\hline
		\multirow{4}{*}{Covariates} & Tier 1 & $M_{1,1}$ & $M_{1,2}$ & $\cdots$ & $M_{1,H}$\\
		& Tier 2 & $M_{2,1}$ & $M_{2,2}$ & $\cdots$ & $M_{2,H}$\\
		& $\vdots$ & $\vdots$ & $\vdots$ & $\ddots$ & $\vdots$\\
		& Tier $T$ & $M_{T,1}$ & $M_{T,2}$ & $\cdots$ & $M_{T,H}$
	\end{tabular}
\end{table}
\begin{example}\label{eg:triang_tier}
	A choice of the $\mathcal{S}_j$'s is the triangular tiers, where  $J=\min\{T,H\}$, and  
	\begin{align*}
	\mathcal{S}_j & = \left\{(t,h): h+t=j+1, 1\leq h\leq H, 1\leq t\leq T\right\}, \quad (1\leq j\leq J-1)  
	\\
	\mathcal{S}_{J} & = \left\{(t,h): h+t>J, 1\leq h\leq H, 1\leq t\leq T\right\}.
	\end{align*}
\end{example}

Let $(a_1, \ldots, a_J)$ be $J$ positive constants predetermined in the design stage. 
Under rerandomized factorial experiments with tiers of covariates and factorial effects, denoted by $\text{ReFMT}_\text{CF}$, 
we accept only those treatment assignments with
$
\sum_{(t,h)\in \mathcal{S}_j} M_{t,h} \leq a_j,
$
for all $1\leq j\leq J$. 
Below we use $\mathcal{T}_{\text{CF}}$ to denote the event that the treatment vector $\bm{Z}$ is accepted under $\text{ReFMT}_\text{CF}$. 
From the finite population central limit theorem, asymptotically, $M_{t,h}$ is $\chi^2_{L_t F_h}$, and the $M_{t,h}$'s are jointly independent. 
For each $1\leq j\leq J$, let $\lambda_j = \sum_{(t,h)\in \mathcal{S}_j}L_tF_h$. 
The asymptotic acceptance probability under $\text{ReFMT}_\text{CF}$ is then 
$
p_a = \prod_{j=1}^{J} P(\chi^2_{\lambda_j}\leq a_j ). 
$

By the same logic as Proposition \ref{prop:simpler_refmt_f}, 
with equal treatment groups sizes, $M_{t,h}$ reduces to 
$M_{t,h} = n/4 \cdot
\sum_{f\in \mathcal{F}_h} 
\hat{\bm{\tau}}_{\bm{e}[t],f}'
\bm{S}_{\bm{e}[t]\bm{e}[t]}^{-1} \hat{\bm{\tau}}_{\bm{e}[t],f},$ 
where $\hat{\bm{\tau}}_{\bm{e}[t],f} = 
2^{-(K-1)} \sum_{q=1}^{Q} g_{fq} \hat{\bar{{\bm{e}}}}_{[t]}(q)$ is the difference-in-means of orthogonalized covariate $\bm{e}[t]$ with respect to the $f$th factorial effect. 

\subsection{Asymptotic sampling distribution of $\hat{\boldsymbol{\tau}}$}
For each tier $\mathcal{S}_j$, let $\bm{U}_{\bm{\tau e}}[j] \in \mathbb{R}^{F\times \lambda_j}$ be a matrix consisting of the columns of matrices $\{\bm{W}_{\bm{\tau}\bm{e}[t]}[h]\}_{(t,h)\in \mathcal{S}_j}$,
$\bm{U}_{\bm{ee}}[j]\in \mathbb{R}^{\lambda_j\times \lambda_j}$ be a block diagonal matrix with $\{\bm{W}_{\bm{e}[t]\bm{e}[t]}[h]\}_{(t,h)\in \mathcal{S}_j}$ as the diagonal components,  
and 
$\bm{U}_{\bm{\tau\tau}}^\myparallel[j] = \bm{U}_{\bm{\tau e}}[j]
(
\bm{U}_{\bm{ee}}[j]
)^{-1}\bm{U}_{\bm{e\tau}}[j]$ be the sampling covariance matrix of $\hat{\bm{\tau}}$ explained by $\{\hat{\bm{\theta}}_{\bm{e}[t]}[h]\}_{(t,h)\in \mathcal{S}_j}$ in the linear projection under the CRFE. 
Recall $\bm{\zeta}_{\lambda_j,a_j}\sim \bm{D}_j \mid \bm{D}_j'\bm{D}_j \leq a_j$, 
where $\bm{D}_j = (D_{j1}, \ldots, D_{j,\lambda_j})' \sim \mathcal{N}(\bm{0}, \bm{I}_{\lambda_j})$. 

\begin{theorem}\label{thm:asym_refmt_cf}
	Under $\text{ReFMT}_\text{CF}$ and Condition \ref{cond:fp},
	\begin{eqnarray}\label{eq:asym_refmt_cf}
	\hat{\bm{\tau}} - \bm{\tau}  \mid \mathcal{T}_{\text{CF}} 
	& \apprsim &
	\left(\bm{V}_{\bm{\tau}\bm{\tau}}^\myperp\right)^{1/2} \bm{\varepsilon} + 
	\sum_{j=1}^{J}
	\left(\bm{U}_{\bm{\tau\tau}}^\myparallel[j]\right)^{1/2}_{\lambda_j}
	\bm{\zeta}_{\lambda_j,a_j},
	\end{eqnarray}
	where $(\bm{\varepsilon}, \bm{\zeta}_{\lambda_1,a_1}, \ldots, \bm{\zeta}_{\lambda_J,a_J})$ are jointly independent. 
\end{theorem}

Let 
$U_{\tau_f\tau_f}^\myparallel[j]$ be the $f$th diagonal element of $\bm{U}_{\bm{\tau\tau}}^\myparallel[j]$,   
and 
$\beta_f^2[j] = U_{\tau_f\tau_f}^\myparallel[j]/V_{\tau_f\tau_f}$ be the proportion of variance of $\hat{\tau}_f$ explained by $
\{\hat{\bm{\theta}}_{\bm{e}[t]}[h]\}_{(t,h)\in \mathcal{S}_j}$
in the linear projection. Because the $
\{\hat{\bm{\theta}}_{\bm{e}[t]}[h]\}_{(t,h)\in \mathcal{S}_j}$'s are essentially from a block-wise Gram--Schmidt orthogonalization of $\hat{\bm{\tau}}_{\bm{x}}$, we have  $\sum_{j=1}^{J}\beta_f^2[j]=R_f^2.$ 
Let $\eta_{\lambda_j,a_j}\sim {D}_{j1}\mid \bm{D}_j'\bm{D}_j\leq a_j$ be the first coordinate of  $\bm{\zeta}_{\lambda_j,a_j}$. 
Theorem \ref{thm:asym_refmt_cf} immediately implies the following asymptotic sampling distribution of a single factorial effect estimator. 
\begin{corollary}\label{cor:single_refmt_cf}
	Under $\text{ReFMT}_\text{CF}$ and Condition \ref{cond:fp}, for $1\leq f\leq F$, 
	\begin{eqnarray}\label{eq:refmt_cf}
	\hat{\tau}_f - \tau_f \mid \mathcal{T}_{\text{CF}}
	& \apprsim & 
	\sqrt{V_{\tau_f\tau_f}}
	\left(
	\sqrt{1-R^2_f} \cdot \varepsilon_0 + \sum_{j=1}^J \sqrt{ \beta_{f}^2[j] }
	\cdot \eta_{\lambda_j,a_j}\right),
	\end{eqnarray}
	where $(\varepsilon_0, \eta_{\lambda_1,a_1}, \ldots,\eta_{\lambda_J,a_J})$ are jointly independent. 
\end{corollary}

\subsection{Asymptotic unbiasedness, sampling covariance and peakedness}

First, the asymptotic sampling distribution \eqref{eq:asym_refmt_cf} is central convex unimodal.  
Therefore, the factorial effect estimator
$\hat{\bm{\tau}}$ is asymptotically unbiased for $\bm{\tau}$ under $\text{ReFMT}_\text{CF}$, and the difference-in-means of covariates with respect to any factorial effect has mean zero asymptotically. 

Second, we consider the asymptotic sampling covariance 
of $\hat{\bm{\tau}}$ under $\text{ReFMT}_\text{CF}$. 
\begin{theorem}\label{thm:red_var_refmt_cf}
	Under Condition \ref{cond:fp}, $\hat{\bm{\tau}}$ has smaller asymptotic sampling covariance matrix under $\text{ReFMT}_\text{CF}$ than that under the CRFE, and the reduction in asymptotic sampling covariance is 
	$
	n\sum_{j=1}^{J}(1-v_{\lambda_j,a_j}) \bm{U}_{\bm{\tau\tau}}^\myparallel[j]. 
	$ 
	For each $1\leq f\leq F,$ the PRIASV of $\hat{\tau}_f$ is 
	$
	\sum_{j=1}^J (1-v_{\lambda_j,a_j})\beta_{f}^2[j]. 
	$
\end{theorem}

Third, we consider the peakedness of \eqref{eq:asym_refmt_cf}.

\begin{theorem}\label{thm:quant_region_refmt_cf}
	Under Condition \ref{cond:fp}, the asymptotic sampling distribution of $\hat{\bm{\tau}}-\bm{\tau}$ under $\text{ReFMT}_\text{CF}$ is more peaked than that under the CRFE. 
\end{theorem}

We then consider specific symmetric convex sets of form $\mathcal{O}(\bm{V}_{\bm{\tau\tau}}, c)$. 
Similar to Section \ref{sec:property_refmt_f}, 
due to some technical difficulties, we consider only the case under the following condition. 

\begin{condition}\label{cond:simultaneous_ortho_refmt_cf}
	There exists an orthogonal matrix $\bm{\Gamma}\in \mathbb{R}^{F\times F}$ such that 
	\begin{align*} 
	\bm{\Gamma}' 
	\bm{V}_{\bm{\tau\tau}}^{-1/2}\bm{U}_{\bm{\tau\tau}}^\myparallel[j]
	\bm{V}_{\bm{\tau\tau}}^{-1/2} \bm{\Gamma}=\textup{diag}(\kappa_{j1}^2, \ldots, \kappa_{jF}^2), 
	\quad (1\leq j\leq J)
	\end{align*}
	where $(\kappa_{j1}^2, \ldots, \kappa_{jF}^2)$ are the canonical correlations between $\hat{\bm{\tau}}$ and $\{\hat{\bm{\theta}}_{\bm{e}[t]}[h]\}_{(t,h)\in \mathcal{S}_j}$ under the CRFE. 
\end{condition}

Condition \ref{cond:simultaneous_ortho_refmt_cf} holds automatically when $J=1$. Moreover, the following proposition shows that the additivity is a sufficient condition for Condition \ref{cond:simultaneous_ortho_refmt_cf}. Recall that $\bm{\Psi}$ is the linear transformation from $\orthocoef_q$ to $\coef_q$, 
$\tilde{\bm{B}} = 2^{-2(K-1)}\sum_{q=1}^{Q} n_q^{-1}\coef_q \coef_q'$,  
and 
$\tilde{\bm{C}} = 2^{-2(K-1)}\sum_{q=1}^{Q} n_q^{-1}\orthocoef_q \orthocoef_q'$. 

\begin{proposition}\label{prop:add_simu_diag_refmt_cf}
	Under the additivity in Definition \ref{def:additive}, Condition \ref{cond:simultaneous_ortho_refmt_cf} holds with orthogonal matrix $\bm{\Gamma}=\tilde{\bm{B}}^{-1/2}\bm{\Psi}  \tilde{\bm{C}}^{1/2}$. 
\end{proposition}

\begin{theorem}\label{thm:quant_region_refmt_cf_monotone_joint}
	Under $\text{ReFMT}_\text{CF}$, assume that Conditions \ref{cond:fp} and \ref{cond:simultaneous_ortho_refmt_cf} hold. Let
	$c_{1-\alpha}$ be the solution of 
	$
	\lim_{n\rightarrow\infty}P\left\{
	\hat{\bm{\tau}} - \bm{\tau} \in \mathcal{O}(\bm{V}_{\bm{\tau\tau}}, c_{1-\alpha}) \mid \mathcal{T}_{\text{CF}} 
	\right\} = 1 - \alpha .
	$
	It depends only on  $K$, $\lambda_j$'s, $a_j$'s, and $(\kappa_{j1}^2, \ldots, \kappa_{jF}^2)$'s, and is nonincreasing in $\kappa_{jf}^2$ for $1\leq j\leq J$ and $1\leq f\leq F$. 
\end{theorem}

Because the peakedness relationship is invariant under linear transformations, and any linear transformation of $\hat{\bm{\tau}}$ has asymptotic sampling distribution of the same form as $\hat{\bm{\tau}}$, we can establish similar conclusions as Theorems \ref{thm:quant_region_refmt_cf} and \ref{thm:quant_region_refmt_cf_monotone_joint} for any linear transformations of $\hat{\bm{\tau}}$. 
We relegate the details to Appendix \ref{app:peak}, and consider only the marginal asymptotic sampling distribution of a single factorial effect estimator here. 

\begin{corollary}\label{cor:red_qr_refmt_cf}
	Under Condition \ref{cond:fp}, for any $1\leq f\leq F$ and $\alpha\in (0,1)$, the threshold $c_{1-\alpha}$ for 
	$1-\alpha$ asymptotic symmetric quantile range  
	$[-c_{1-\alpha}V_{\tau_f\tau_f}^{1/2}, c_{1-\alpha}V_{\tau_f\tau_f}^{1/2}]$
	of $\hat{\tau}_f-\tau_f$ under $\text{ReFMT}_\text{CF}$ is smaller than or equal to that under the CRFE, and is nonincreasing in $\beta^2_f(j)$, for any $1\leq j\leq J$. 
\end{corollary}

\begin{example}
	We consider again the setting in Example \ref{eg:simple_refmt_f} in the main text, and further assume that the finite population partial covariance between potential outcome $Y(1)$ and other covariates given the first covariate is zero, i.e., only the first covariate is important. 
	We divide the covariates into two tiers, where tier $1$ contains only the first covariate and tier $2$ contains the remaining covariates, and then construct triangular tiers in Example \ref{eg:triang_tier}, where $\mathcal{S}_1$ consists of the combination of main effects and first covariate with  $\lambda_1=K$, and $\mathcal{S}_2$ consists of the remaining combinations of factorial effects and covariates with $\lambda_2 = (2^K-1) L-K$. 
	We choose thresholds $(a_1,a_2)$ such that 
	$P(\chi^2_{\lambda_1}\leq a_1) = 0.002$ and $P(\chi^2_{\lambda_2}\leq a_2)=0.5$. 
	Then for the main effect $1\leq k \leq K,$ $\beta_k^2[1]=R_k^2$ and $\beta_k^2[2]=0.$ 
	Figure \ref{fig:simple_compare_refmt_cf} shows the PRIASV, divided by $R^2_f$, of the main effect estimators for $\text{ReFMT}_\text{F}$ and $\text{ReFMT}_\text{CF}$. It shows that the advantage of further using tiers of covariates increases as numbers of factors and covariates increase. 
	$\hfill\Box$
	\begin{figure}[htb]
		\centering
		\includegraphics[width=0.5\textwidth]{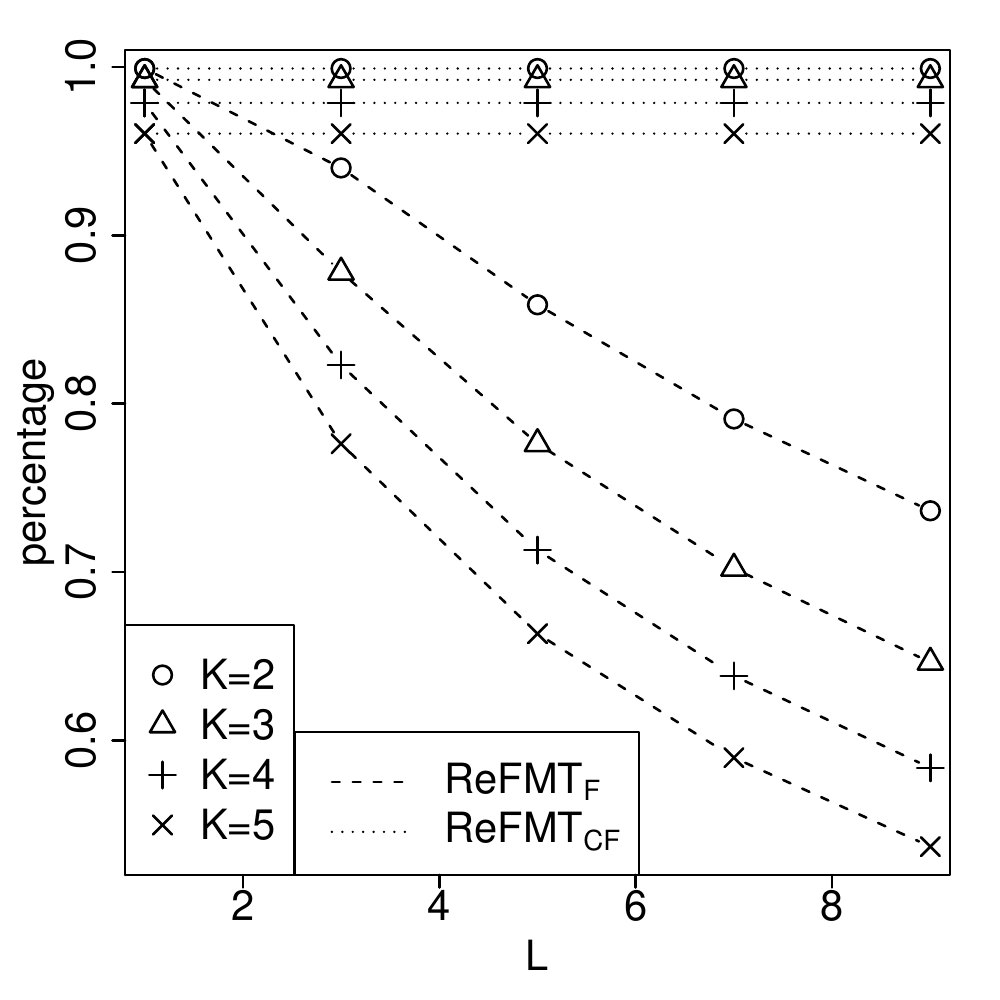}
		\caption{PRIASV of main effect estimators under $\text{ReFMT}_\text{F}$ and $\text{ReFMT}_\text{CF}$, divided by $R^2_f$}\label{fig:simple_compare_refmt_cf}
	\end{figure}
\end{example}

\subsection{Conservative covariance estimator and confidence set under $\text{ReFMT}_\text{CF}$}\label{sec:var_est_refmt_cf}

We define $\bm{s}_{q,\bm{e}[t]}$ and $\bm{s}_{\bm{e}[t]\bm{e}[t]}(q)$ as the sample covariances 
of the observed outcomes and orthogonalized covariates under treatment combination $q$.  
We estimate $\bm{V}_{\bm{\tau\tau}}^\myperp$ by 
$\hat{\bm{V}}_{\bm{\tau\tau}}^\myperp$ 
in \eqref{eq:V_tautau_perp_est}, 
$\bm{W}_{\bm{\tau e}[t]}[h]$ by 
$\hat{\bm{W}}_{\bm{\tau e}[t]}[h] = 2^{-2(K-1)}\sum_{q=1}^{Q}
n_q^{-1}(\coef_q\orthocoef'_q[h]) \otimes \left\{\bm{s}_{q,\bm{e}[t]}\bm{s}_{\bm{e}[t]\bm{e}[t]}^{-1/2}(q)\bm{S}_{\bm{e}[t]\bm{e}[t]}^{1/2}\right\}$, $\bm{U}_{\bm{\tau e}}[j]$ by $\hat{\bm{U}}_{\bm{\tau e}}[j]$ consisting of the corresponding 
$\hat{\bm{W}}_{\bm{\tau e}[t]}[h]$'s, and $
(\bm{U}_{\bm{\tau\tau}}^\myparallel[j])^{1/2}_{\lambda_j} = 
\bm{U}_{\bm{\tau e}}[j]
(
\bm{U}_{\bm{ee}}[j]
)^{-1/2}$ by $\hat{\bm{U}}_{\bm{\tau e}}[j]
(
\bm{U}_{\bm{ee}}[j]
)^{-1/2}$.  
We can then obtain a covariance estimator and 
construct  confidence sets for $\bm{\tau}$ or its lower dimensional linear transformation.  
Similar to ReFM, 
for a parameter of interest $\bm{C}\bm{\tau}$, we recommend to use confidence sets of the form 
$
\bm{C}\hat{\bm{\tau}} + \mathcal{O}(\bm{C}\hat{\bm{V}}_{\bm{\tau\tau}}^\myperp\bm{C}', c), 
$
and choose the threshold $c$ by simulating the estimated asymptotic distribution. 
Let 
$\hat{c}_{1-\alpha}$ be the $1-\alpha$ quantile of 
$ (\bm{C}\bm{\phi})'(\bm{C}\hat{\bm{V}}_{\bm{\tau\tau}}^\myperp\bm{C}')^{-1}(\bm{C}\bm{\phi})
$
with $\bm{\phi}$ following the estimated asymptotic sampling distribution of $\hat{\bm{\tau}}-\bm{\tau}$ under $\text{ReFMT}_\text{CF}.$

\begin{theorem}\label{thm:conserv_cs_remft_cf}
	Under $\text{ReFMT}_\text{CF}$ and Condition \ref{cond:fp},  consider inferring $\bm{C}\bm{\tau}$, where
	$\bm{C}$ has full row rank. The probability limit of covariance estimator for $\bm{C}\hat{\bm{\tau}}$, 
	$
	\bm{C}\hat{\bm{V}}_{\bm{\tau\tau}}^\myperp\bm{C}' + 
	\sum_{j=1}^{J} 
	v_{\lambda_j, a_j}
	\bm{C} \hat{\bm{U}}_{\bm{\tau e}}[j]
	(
	\bm{U}_{\bm{ee}}[j]
	)^{-1}
	\hat{\bm{U}}_{\bm{e\tau}}[j]\bm{C}', 
	$
	is larger than or equal to the actual sampling covariance, and 
	the $1-\alpha$ confidence set, 
	$
	\bm{C}\hat{\bm{\tau}} + \mathcal{O}(\bm{C}\hat{\bm{V}}_{\bm{\tau\tau}}^\myperp\bm{C}', \hat{c}_{1-\alpha}), 
	$
	has asymptotic coverage rate $\geq 1-\alpha$, 
	with equality holding if $\bm{S}_{\bm{\tau\tau}}^\myperp\rightarrow 0$ as $n\rightarrow \infty$. 
\end{theorem}

The above confidence sets is similar to the ones based on regression adjustment if the threshold $a_j$'s are small \citep{lu2016covadj}.  Moreover,  we will extend Theorem \ref{thm:conserv_cs_remft_cf} to general symmetric convex confidence sets in Appendix \ref{app:inference}.

\section{Sampling properties under the CRFE}\label{app:covariance}

\subsection{Lemmas for matrices}
For any positive integer $m$, we use $\bm{1}_{m}$ to denote an $m$ dimensional column vector with all elements one, 
$\bm{I}_m$ to denote an $m\times m$ identity matrix, 
and $\bm{J}_m$ to denote an $m\times m$ matrix with all elements one.

\begin{lemma}\label{lemma:kronecker_product}
	Let $\bm{A}, \bm{B}, \bm{C}$ and $\bm{D}$ be four matrices.  
	\begin{itemize}
		\item[(1)] $(\bm{A}\otimes \bm{B})' = \bm{A}' \otimes \bm{B}'.$
		\item[(2)] If $\bm{A}$ and $\bm{B}$ have the same dimension, then 
		$
		(\bm{A}+\bm{B})\otimes \bm{C} = \bm{A}\otimes \bm{C} + \bm{B}\otimes \bm{C}, $
		and 
		$ 
		\bm{C} \otimes (\bm{A}+\bm{B}) = \bm{C} \otimes \bm{A} + \bm{C} \otimes \bm{B}. 
		$
		\item[(3)] If one can form the matrix products $\bm{AC}$ and $\bm{BD}$, then
		$
		(\bm{A}\otimes \bm{B})(\bm{C} \otimes \bm{D}) = 
		(\bm{AC}) \otimes (\bm{BD}). 
		$
		\item[(4)] If $\bm{A}$ and $\bm{B}$ are invertible, then 
		$
		(\bm{A}\otimes \bm{B})^{-1} = \bm{A}^{-1} \otimes \bm{B}^{-1}. 
		$
	\end{itemize}
\end{lemma}

Recall that $\bm{g}_f = (g_{f1}, \ldots, g_{fQ})'$ is the generating vector for the $f$th factorial effect, 
$\bm{b}_q = (g_{1q}, \ldots, g_{Fq})'$ is the coefficient vector for the treatment combination $q$, and 
$\covcoef = 2^{-2(K-1)}\sum_{q=1}^{Q}n_q^{-1}\coef_q\coef_q'$. 
\begin{lemma}\label{lemma:rqrmhq_inv_hm} 
	For any $1\leq q,k\leq Q$, 
	\begin{align}\label{eq:rqrmhq_inv_hm}
	2^{-2(K-1)} n_q^{-1}n_k^{-1}\coef_q' 
	\covcoef^{-1}
	\coef_k = n_q^{-1} \times 1\{ q=k \} - n^{-1} .
	\end{align} 
\end{lemma}
\begin{proof}[Proof of Lemma \ref{lemma:rqrmhq_inv_hm}]
	Let $\bm{B}=(\coef_1, \ldots, \coef_{Q}) = (
	\bm{g}_1, \ldots, \bm{g}_{F}
	)'\in \mathbb{R}^{F\times Q}$, and $\bm{N} = \textup{diag}(n_1, \ldots, n_{Q})\in \mathbb{R}^{Q\times Q}$. Then $\covcoef = 2^{-2(K-1)} \sum_{q=1}^{Q}n_q^{-1}\coef_q\coef_q'= 2^{-2(K-1)} \bm{B}\bm{N}^{-1}\bm{B}',$ and $2^{-2(K-1)} n_q^{-1}n_k^{-1}\coef_q' 
	\covcoef^{-1}
	\coef_k$
	is the $(q,k)$th element of matrix
	$
	\bm{C} = 
	\bm{N}^{-1}\bm{B}'
	\left(
	\bm{B}\bm{N}^{-1}\bm{B}'
	\right)^{-1}
	\bm{B} \bm{N}^{-1}. 
	$
	
	First, we show that there exists a constant $c$ such that 
	$
	\bm{C} = \bm{N}^{-1} + c \bm{J}_{Q}. 
	$
	By definition,  
	$
	\bm{B}\bm{C} = 
	\bm{B}\bm{N}^{-1}\bm{B}'
	\left(
	\bm{B}\bm{N}^{-1}\bm{B}'
	\right)^{-1}
	\bm{B} \bm{N}^{-1}
	= 
	\bm{B}\bm{N}^{-1}.
	$
	Thus, 
	\begin{align}\label{eq:A_C_N_inv}
	\bm{0} = \bm{B}(\bm{C}-\bm{N}^{-1}) = 
	\left(
	\bm{g}_1, \ldots, \bm{g}_{F}
	\right)' (\bm{C}-\bm{N}^{-1}).
	\end{align}
	By the properties of generating vectors, 
	$(\bm{1}_{Q}, \bm{g}_1, \ldots, \bm{g}_{F})$ constitute an  orthogonal basis of $\mathbb{R}^{Q}$. 
	Equation \eqref{eq:A_C_N_inv} implies that each column of $\bm{C}-\bm{N}^{-1}$ is orthogonal to $\bm{g}_1, \ldots, \bm{g}_{F}$, and thus has to be $c \bm{1}_{Q\times 1}$ for some constant $c$. Because $\bm{C}-\bm{N}^{-1}$ is a symmetric matrix, $\bm{C}-\bm{N}^{-1}$ must be $c\bm{J}_{Q}$ for some constant $c.$ 
	
	Second, we show that $c=-n^{-1}$ and  
	$
	\bm{C} = \bm{N}^{-1} - n^{-1}\bm{J}_{Q}.  
	$
	On the one hand,
	$
	\bm{C} = \bm{N}^{-1} + c \bm{J}_{Q}
	$
	implies 
	$$
	\text{tr}(\bm{CN}) = \text{tr}(\bm{I})+ c \cdot \text{tr}(\bm{J}_Q\bm{N}) = Q+c\sum_{q=1}^Q n_q = Q + cn;
	$$
	on the other hand, the definition of $\bm{C}$ implies 
	\begin{align*}
	\text{tr}(\bm{CN}) & = \text{tr}\left\{\bm{N}^{-1}\bm{B}'
	\left(
	\bm{B}\bm{N}^{-1}\bm{B}'
	\right)^{-1}
	\bm{B}\right\} = 
	\text{tr}\left\{
	\bm{B}\bm{N}^{-1}\bm{B}'
	\left(
	\bm{B}\bm{N}^{-1}\bm{B}'
	\right)^{-1}
	\right\} \\
	& = F = Q-1.
	\end{align*}
	Therefore, 
	$c=-n^{-1},$ 
	$
	\bm{C} = \bm{N}^{-1} - n^{-1}\bm{J}_{Q},
	$
	and 
	Lemma \ref{lemma:rqrmhq_inv_hm} holds.  
\end{proof}

\begin{lemma}\label{lemma:ortho_AApm_BBpm}
	If two matrices $\bm{A}$ and $\bm{B}$ in $\mathbb{R}^{p\times m}$ satisfy $\bm{A}\bm{A}'=\bm{B}\bm{B'}$, then there exists an orthogonal matrix $\bm{\Gamma}\in \mathbb{R}^{m\times m}$ such that $\bm{A}=\bm{B}\bm{\Gamma}$. 
\end{lemma}

\begin{proof}[Proof of Lemma \ref{lemma:ortho_AApm_BBpm}]
	First, we consider the case with $p=m$, i.e., $\bm{A}$ and $\bm{B}$ are square matrices. 
	From the polar decomposition, $\bm{A}=(\bm{A}\bm{A}')^{1/2}\bm{\Gamma}_1$ and 
	$\bm{B}=(\bm{B}\bm{B}')^{1/2}\bm{\Gamma}_2$, 
	where $\bm{\Gamma}_1$ and $\bm{\Gamma}_2$ are orthogonal matrices. 
	Therefore, 
	\begin{align*}
	\bm{B}\bm{\Gamma}_2'\bm{\Gamma}_1 
	= (\bm{B}\bm{B}')^{1/2}\bm{\Gamma}_2\bm{\Gamma}_2'\bm{\Gamma}_1 
	= (\bm{A}\bm{A}')^{1/2}\bm{\Gamma}_1  
	= \bm{A},
	\end{align*}
	where $\bm{\Gamma}_2'\bm{\Gamma}_1$ is an orthogonal matrix. 
	Thus, Lemma \ref{lemma:ortho_AApm_BBpm} holds when $p=m$. 
	
	Second, we consider the case with $p<m$. Define two square matrices:  
	\begin{align*}
	\bm{A}_1 = 
	\begin{pmatrix}
	\bm{A}\\
	\bm{0}_{(m-p)\times m}
	\end{pmatrix} \in \mathbb{R}^{m\times m}, 
	\quad 
	\bm{B}_1 = 
	\begin{pmatrix}
	\bm{B}\\
	\bm{0}_{(m-p)\times m}
	\end{pmatrix} \in \mathbb{R}^{m\times m}. 
	\end{align*}
	We can verify $\bm{A}_1\bm{A}_1'=\bm{B}_1\bm{B}_1'$. 
	From the first case for square matrices, there exists an orthogonal matrix $\bm{\Gamma}\in \mathbb{R}^{m\times m}$ such that $\bm{A}_1 = \bm{B}_1\bm{\Gamma}$, which immediately implies $\bm{A} = \bm{B}\bm{\Gamma}$. 
	Thus, Lemma \ref{lemma:ortho_AApm_BBpm} holds when $p<m$. 
	
	Third, we consider the case with $p>m$. Let $\bm{W}_0\in \mathbb{R}^{p\times p_0}$ be a matrix whose columns are an orthonormal basis of $\{\bm{x}: \bm{A}'\bm{x}=0 \}\subset \mathbb{R}^p$, and $\bm{W}_1 \in \mathbb{R}^{p\times (p-p_0)}$ be the matrix such that $\bm{W}=(\bm{W}_1, \bm{W}_0) \in \mathbb{R}^{p\times p}$ is an orthogonal matrix. 
	We can verify $\bm{W}_0'\bm{A}=\bm{0}_{p_0\times m}$. 
	Because 
	$$
	\bm{W}_0'\bm{B} (\bm{W}_0'\bm{B})' = \bm{W}_0'\bm{B} \bm{B}' \bm{W}_0 
	= \bm{W}_0'\bm{A} \bm{A}' \bm{W}_0 = \bm{0}_{p_0\times p_0}, 
	$$
	we have $\bm{W}_0'\bm{B}=\bm{0}_{p_0\times m}$.
	Because $\{\bm{x}: \bm{A}'\bm{x}=0 \}$ is of dimension at least $p-m$, we have  $p_0\geq p-m$ and therefore $p-p_0\leq m$. 
	Because the two matrices $\bm{W}_1'\bm{A}$ and $\bm{W}_1'\bm{B}$ in $\mathbb{R}^{(p-p_0)\times m}$ satisfy 
	$
	\bm{W}_1'\bm{A}\bm{A}'\bm{W}_1 = \bm{W}_1'\bm{B}\bm{B}'\bm{W}_1, 
	$
	from the second case, there exists an orthogonal matrix $\bm{\Gamma}\in \mathbb{R}^{m\times m}$ such that 
	$
	\bm{W}_1'\bm{A} = \bm{W}_1'\bm{B}\bm{\Gamma}. 
	$
	Thus, 
	\begin{align*}
	\bm{W}' \bm{B}\bm{\Gamma} = 
	\begin{pmatrix}
	\bm{W}_1'\\
	\bm{W}_0'
	\end{pmatrix}
	\bm{B}\bm{\Gamma} 
	= 
	\begin{pmatrix}
	\bm{W}_1'\bm{B}\bm{\Gamma} \\
	\bm{W}_0'\bm{B}\bm{\Gamma} 
	\end{pmatrix}
	= 
	\begin{pmatrix}
	\bm{W}_1'\bm{A} \\
	\bm{0}_{p_0\times m}
	\end{pmatrix}
	= 
	\begin{pmatrix}
	\bm{W}_1'\bm{A}\\
	\bm{W}_0'\bm{A}
	\end{pmatrix}
	= \bm{W}'\bm{A}.  
	\end{align*}
	Because $\bm{W}$ is an orthogonal matrix, we have 
	$\bm{B}\bm{\Gamma} = \bm{A}$. Thus, Lemma \ref{lemma:ortho_AApm_BBpm} holds when $p>m$.
\end{proof}

\subsection{Covariances between $\hat{\bm{\tau}}$ and $\hat{\bm{\tau}}_{\bm{x}}$}

\begin{proof}[{   Proof of Proposition \ref{prop:mean_var_cre}}]
	We first write the vector $(\hat{\bm{\tau}}', \hat{\bm{\tau}}_{\bm{x}}')'$ as a linear combination of average observed outcomes and covariates: 
	\begin{align}\label{eq:vector_Y_and_X}
	\begin{pmatrix}
	\hat{\bm{\tau}}\\
	\hat{\bm{\tau}}_{\bm{x}}
	\end{pmatrix}
	= 
	2^{-(K-1)}\sum_{q=1}^Q 
	\begin{pmatrix}
	\coef_q\hat{\bar{Y}}(q)\\
	g_{1q} \hat{\bar{{\bm{x}}}}(q)\\
	\vdots\\
	g_{Fq} \hat{\bar{{\bm{x}}}}(q)
	\end{pmatrix}
	= 2^{-(K-1)}\sum_{q=1}^{Q}
	\begin{pmatrix}
	\coef_q & \bm{0}_{F\times L}\\
	\bm{0}_{FL \times 1} & \coef_q \otimes \bm{I}_L
	\end{pmatrix}
	\begin{pmatrix}
	\hat{\bar{Y}}(q)\\
	\hat{\bar{{\bm{x}}}}(q)
	\end{pmatrix}. 
	\end{align}
	We can view covariates as ``outcomes" unaffected by the treatment, and view $(Y_i(q), \bm{x}_i')'$ as the potential outcome vector under treatment combination $q$. Using  \citet[][Theorem 3]{fcltxlpd2016}, $(\hat{\bm{\tau}}',\hat{\bm{\tau}}_{\bm{x}}')'$ has sampling mean 
	\begin{align*}
	2^{-(K-1)}\sum_{q=1}^{Q}
	\begin{pmatrix}
	\coef_q & \bm{0}_{F\times L}\\
	\bm{0}_{FL \times 1} & \coef_q \otimes \bm{I}_L
	\end{pmatrix}
	\begin{pmatrix}
	\bar{Y}(q)\\
	\bar{{\bm{x}}}
	\end{pmatrix} = 
	\begin{pmatrix}
	\bm{\tau}\\
	\bm{0}_{FL\times 1}
	\end{pmatrix},
	\end{align*} 
	and sampling covariance matrix
	\begin{eqnarray*}
		& & 2^{-2(K-1)}\sum_{q=1}^{Q}n_q^{-1}
		\begin{pmatrix}
			\coef_q & \bm{0}_{F\times L}\\
			\bm{0}_{FL \times 1} & \coef_q \otimes \bm{I}_L
		\end{pmatrix}
		\begin{pmatrix}
			S_{qq} & \bm{S}_{q,\bm{x}}\\
			\bm{S}_{\bm{x},q} & \bm{S}_{\bm{xx}}
		\end{pmatrix}
		\begin{pmatrix}
			\coef_q & \bm{0}_{F\times L}\\
			\bm{0}_{FL \times 1} & \coef_q \otimes \bm{I}_L
		\end{pmatrix}' \\
		& & - 
		n^{-1}
		\begin{pmatrix}
			\bm{S}_{\bm{\tau\tau}} & \bm{0}\\
			\bm{0} & \bm{0}
		\end{pmatrix}\\
		& = & 
		2^{-2(K-1)}
		\sum_{q=1}^{Q}
		n_q^{-1}
		\begin{pmatrix}
			\coef_q \coef_q' S_{qq} & \coef_q \bm{S}_{q, \bm{x}} \cdot (\coef_q \otimes \bm{I}_L)'\\
			\left(\coef_q \otimes \bm{I}_L\right) \cdot \bm{S}_{\bm{x}, q} \coef_q'  & \left(\coef_q \otimes \bm{I}_L\right) \bm{S}_{\bm{xx}} \left(\coef_q \otimes \bm{I}_L\right)'
		\end{pmatrix} \\
		& & - 
		n^{-1}
		\begin{pmatrix}
			\bm{S}_{\bm{\tau\tau}} & \bm{0}\\
			\bm{0} & \bm{0}
		\end{pmatrix}. 
	\end{eqnarray*}
	From Lemma \ref{lemma:kronecker_product}, 
	\begin{align*}
	\coef_q \bm{S}_{q, \bm{x}} \cdot \left(\coef_q \otimes \bm{I}_L\right)' & = (\coef_q \otimes \bm{S}_{q, \bm{x}})\left(\coef_q' \otimes \bm{I}_L\right) = (\coef_q\coef_q') \otimes \bm{S}_{q, \bm{x}}, \\
	(\coef_q \otimes \bm{I}_L) \bm{S}_{\bm{xx}} (\coef_q \otimes \bm{I}_L)' & = 
	(\coef_q \otimes \bm{I}_L) (1 \otimes \bm{S}_{\bm{xx}}) (\coef_q' \otimes \bm{I}_L) = 
	(\coef_q \coef_q') \otimes \bm{S}_{\bm{xx}}.  
	\end{align*}
	Therefore, we deduce the mean and sampling covariance of $(\hat{\bm{\tau}}',\hat{\bm{\tau}}_{\bm{x}}')$ in the form of Proposition \ref{prop:mean_var_cre}. The asymptotic Gaussianity of $(\hat{\bm{\tau}}',\hat{\bm{\tau}}_{\bm{x}}')'$ follows directly from Condition \ref{cond:fp} and \citet[][Theorem 5]{fcltxlpd2016}.
\end{proof}

\begin{proof}[{   Proof of Proposition \ref{prop:var_cov_tier_factor}}]
	We first rewrite 
	\begin{align*}
	\begin{pmatrix}
	\hat{\bm{\tau}}-\bm{\tau}\\
	\hat{\bm{\theta}}_{\bm{x}}
	\end{pmatrix}
	= 
	2^{-(K-1)}
	\sum_{q=1}^{Q}
	\begin{pmatrix}
	\coef_q \hat{\bar{Y}}(q) - \coef_q \bar{Y}(q)\\
	\orthocoef_q\otimes\hat{\bar{\bm{x}}}(q)
	\end{pmatrix}
	= 
	2^{-(K-1)}
	\sum_{q=1}^{Q}
	\begin{pmatrix}
	\coef_q & \bm{0}\\
	\bm{0} & \orthocoef_q\otimes\bm{I}_L
	\end{pmatrix}
	\begin{pmatrix}
	\hat{\bar{Y}}(q) - \bar{Y}(q)\\
	\hat{\bar{\bm{x}}}(q)
	\end{pmatrix}.
	\end{align*}
	From \citet[][Theorem 3]{fcltxlpd2016}, $(\hat{\bm{\tau}}'-\bm{\tau}', \hat{\bm{\theta}}_{\bm{x}}')'$ has mean zero and sampling covariance
	\begin{eqnarray*}
		& & 
		2^{-2(K-1)}
		\sum_{q=1}^{Q}
		n_q^{-1}
		\begin{pmatrix}
			\coef_q & \bm{0}\\
			\bm{0} & \orthocoef_q\otimes\bm{I}_L
		\end{pmatrix}
		\begin{pmatrix}
			S_{qq} & \bm{S}_{q, \bm{x}}\\
			\bm{S}_{\bm{x},q} & \bm{S}_{\bm{xx}}
		\end{pmatrix}
		\begin{pmatrix}
			\coef_q & \bm{0}\\
			\bm{0} & \orthocoef_q\otimes\bm{I}_L
		\end{pmatrix}'
		- 
		n^{-1}
		\begin{pmatrix}
			\bm{S}_{\bm{\tau\tau}} & \bm{0}\\
			\bm{0} & \bm{0}
		\end{pmatrix}\\
		& = & 2^{-2(K-1)}
		\sum_{q=1}^{Q}
		n_q^{-1}
		\begin{pmatrix}
			\coef_q\coef_q' S_{qq} & (\coef_q\orthocoef_q')\otimes \bm{S}_{q,\bm{x}}\\
			(\orthocoef_q\coef_q') \otimes \bm{S}_{\bm{x},q} & (\orthocoef_q\orthocoef'_q) \otimes \bm{S}_{\bm{xx}} 
		\end{pmatrix} - 
		n^{-1}
		\begin{pmatrix}
			\bm{S}_{\bm{\tau\tau}} & \bm{0} \\
			\bm{0} & \bm{0} 
		\end{pmatrix}.
	\end{eqnarray*}
	From the Gram--Schmidt orthogonalization of the $\orthocoef_q$'s, 
	\begin{align*}
	\sum_{q=1}^{Q}
	n_q^{-1}\orthocoef_q\orthocoef_q' & = \sum_{q=1}^{Q}
	n_q^{-1}
	\begin{pmatrix}
	\orthocoef_q[1]\orthocoef_q'[1] & \bm{0} & \ldots & \bm{0}\\
	\bm{0} & \orthocoef_q[2]\orthocoef_q'[2] & \ldots & \bm{0}\\
	\vdots & \vdots & \ddots & \vdots\\
	\bm{0} & \bm{0} & \ldots & \orthocoef_q[H]\orthocoef_q'[H]
	\end{pmatrix}. 
	\end{align*}
	Therefore, Proposition \ref{prop:var_cov_tier_factor} holds. 
\end{proof}

\begin{proof}[{   Proof of Proposition \ref{prop:var_cov_tier_factor_covariate}}]
	We first rewrite
	\begin{align*}
	\begin{pmatrix}
	\hat{\bm{\tau}}-\bm{\tau}\\
	\hat{\bm{\theta}}_{\bm{e}[1]}\\
	\vdots\\
	\hat{\bm{\theta}}_{\bm{e}[T]}
	\end{pmatrix}
	& = 2^{-(K-1)}
	\sum_{q=1}^{Q}
	\begin{pmatrix}
	\coef_q \hat{\bar{Y}}(q) - \coef_q \bar{Y}(q)\\
	\orthocoef_q\otimes\hat{\bar{\bm{e}}}_{[1]}(q)\\
	\vdots\\
	\orthocoef_q\otimes\hat{\bar{\bm{e}}}_{[T]}(q)
	\end{pmatrix}\\
	& = 2^{-(K-1)}
	\sum_{q=1}^{Q}
	\begin{pmatrix}
	\coef_q & \bm{0} & \ldots & \bm{0}\\
	\bm{0} & \orthocoef_q\otimes\bm{I}_{L_1} & \ldots & \bm{0}\\
	\vdots &\vdots & \ddots & \vdots\\
	\bm{0} &  \bm{0} & \ldots & \orthocoef_q\otimes \bm{I}_{L_T}
	\end{pmatrix}
	\begin{pmatrix}
	\hat{\bar{Y}}(q) - \bar{Y}(q)\\
	\hat{\bar{\bm{e}}}_{[1]}(q)\\
	\vdots\\
	\hat{\bar{\bm{e}}}_{[T]}(q)
	\end{pmatrix}.
	\end{align*}
	From \citet[][Theorem 3]{fcltxlpd2016}, $(\hat{\bm{\tau}}'-\bm{\tau}', \hat{\bm{\theta}}_{\bm{e}[1]}', \ldots, \hat{\bm{\theta}}_{\bm{e}[T]}')'$ has mean zero and sampling covariance
	\begin{align*}
	& \quad \ 2^{-2(K-1)}\sum_{q=1}^{Q} n_q^{-1}
	\left\{
	\begin{pmatrix}
	\coef_q & \bm{0} & \ldots & \bm{0}\\
	\bm{0} & \orthocoef_q\otimes\bm{I}_{L_1} & \ldots & \bm{0}\\
	\vdots &\vdots & \ddots & \vdots\\
	\bm{0} &  \bm{0} & \ldots & \orthocoef_q\otimes \bm{I}_{L_T}
	\end{pmatrix}
	\right.
	\\
	& \quad   \ 
	\left.
	\cdot
	\begin{pmatrix}
	S_{qq} & \bm{S}_{q, \bm{e}[1]} & \ldots & \bm{S}_{q, \bm{e}[T]}\\
	\bm{S}_{\bm{e}[1], q} & \bm{S}_{\bm{e}[1]\bm{e}[1]} & \ldots & \bm{0}\\
	\vdots & \vdots & \ddots & \vdots\\
	\bm{S}_{\bm{e}[T],q} & \bm{0} & \ldots & \bm{S}_{\bm{e}[T]\bm{e}[T]}
	\end{pmatrix}
	\cdot
	\begin{pmatrix}
	\coef_q & \bm{0} & \ldots & \bm{0}\\
	\bm{0} & \orthocoef_q\otimes\bm{I}_{L_1} & \ldots & \bm{0}\\
	\vdots &\vdots & \ddots & \vdots\\
	\bm{0} &  \bm{0} & \ldots & \orthocoef_q\otimes \bm{I}_{L_T}
	\end{pmatrix}'
	\right\}\\
	& \quad \ 
	- n^{-1}
	\begin{pmatrix}
	\bm{S}_{\bm{\tau\tau}} & \bm{0} & \ldots & \bm{0}\\
	\bm{0} & \bm{0} & \ldots & \bm{0}\\
	\vdots &\vdots & \ddots & \vdots\\
	\bm{0} &  \bm{0} & \ldots & \bm{0}
	\end{pmatrix}\\
	& = 2^{-2(K-1)}
	\sum_{q=1}^{Q} n_q^{-1}
	\begin{pmatrix}
	\coef_q\coef_q'S_{qq} & (\coef_q\orthocoef_q')\otimes\bm{S}_{q, \bm{e}[1]} & \ldots & (\coef_q\orthocoef_q')\otimes\bm{S}_{q, \bm{e}[T]}\\
	(\orthocoef_q\coef_q')\otimes\bm{S}_{\bm{e}[1], q} & (\orthocoef_q\orthocoef_q')\otimes\bm{S}_{\bm{e}[1]\bm{e}[1]} & \ldots & \bm{0}\\
	\vdots & \vdots & \ddots & \vdots\\
	(\orthocoef_q\coef_q')\otimes \bm{S}_{\bm{e}[T],q} & \bm{0} & \ldots & (\orthocoef_q\orthocoef_q')\otimes\bm{S}_{\bm{e}[T]\bm{e}[T]}
	\end{pmatrix}
	\\
	& \quad \ -
	n^{-1}
	\begin{pmatrix}
	\bm{S}_{\bm{\tau\tau}} & \bm{0} & \ldots & \bm{0}\\
	\bm{0} & \bm{0} & \ldots & \bm{0}\\
	\vdots &\vdots & \ddots & \vdots\\
	\bm{0} &  \bm{0} & \ldots & \bm{0}
	\end{pmatrix}. 
	\end{align*}
	Therefore, Proposition \ref{prop:var_cov_tier_factor_covariate} holds.
\end{proof}

\subsection{Linear projections and squared multiple correlations}
\begin{proof}[{   Proof of Theorem \ref{thm:joint_var_expl_unexpl}}]
	From Lemma \ref{lemma:kronecker_product}, 
	the sampling covariance matrix of $\hat{\bm{\tau}}$ explained by  $\hat{\boldsymbol{\tau}}_{\boldsymbol{x}}$ in the linear projection satisfies 
	\begin{align*}
	& \quad \ \bm{V}_{\bm{\tau x}} \bm{V}_{\bm{xx}}^{-1} \bm{V}_{\bm{x\tau}} \\
	& = 
	2^{-4(K-1)} 
	\left\{\sum_{q=1}^{Q} n_q^{-1}(\coef_q\coef_q')\otimes \bm{S}_{q, \bm{x}}
	\right\}
	\left(
	\covcoef \otimes \bm{S}_{\bm{xx}}
	\right)^{-1}
	\left\{
	\sum_{k=1}^{Q} n_k^{-1}(\coef_k\coef_k')\otimes \bm{S}_{\bm{x}, k}
	\right\}\\
	& = 
	2^{-4(K-1)} 
	\sum_{q=1}^{Q}\sum_{k=1}^{Q} 
	\left\{\sum_{q=1}^{Q} n_q^{-1}(\coef_q\coef_q')\otimes \bm{S}_{q, \bm{x}}
	\right\}
	\left(
	\covcoef^{-1}
	\otimes \bm{S}_{\bm{xx}}^{-1}
	\right) 
	\left\{
	\sum_{k=1}^{Q} n_k^{-1}(\coef_k\coef_k')\otimes \bm{S}_{\bm{x}, k}
	\right\}\\
	& = 2^{-4(K-1)} 
	\sum_{q=1}^{Q}\sum_{k=1}^{Q}
	\left\{
	n_q^{-1}(\coef_q\coef_q') \cdot \covcoef^{-1}\cdot
	n_k^{-1}(\coef_k\coef_k')
	\right\} \otimes 
	\left(
	\bm{S}_{q, \bm{x}}
	\bm{S}_{\bm{xx}}^{-1}
	\bm{S}_{\bm{x}, k}
	\right)\\
	& = 2^{-2(K-1)} 
	\sum_{q=1}^{Q}\sum_{k=1}^{Q}
	\left\{ \coef_q \left(
	2^{-2(K-1)}
	n_q^{-1}n_k^{-1}
	\coef_q' 
	\covcoef^{-1}
	\coef_k \right) \coef_k'
	\right\}
	\cdot
	\left(
	\bm{S}_{q, \bm{x}}
	\bm{S}_{\bm{xx}}^{-1}
	\bm{S}_{\bm{x}, k}
	\right),  
	\end{align*}
	where in the last equality the Kronecker product reduces to the matrix product because $\bm{S}_{q, \bm{x}}
	\bm{S}_{\bm{xx}}^{-1}
	\bm{S}_{\bm{x}, k}$ is a scalar. 
	Using Lemma \ref{lemma:rqrmhq_inv_hm}, we have
	\begin{align*}
	& \quad \ \bm{V}_{\bm{\tau x}} \bm{V}_{\bm{xx}}^{-1} \bm{V}_{\bm{x \tau}}
	\\
	& = 
	2^{-2(K-1)} 
	\left[
	\sum_{q=1}^{Q}
	(n_q^{-1} \coef_q\coef_q') \cdot \left( \bm{S}_{q, \bm{x}}
	\bm{S}_{\bm{xx}}^{-1}
	\bm{S}_{\bm{x}, q} \right)  
	-
	\sum_{q=1}^{Q}\sum_{k=1}^{Q} n^{-1}
	\coef_q\coef_k' \cdot
	\left(
	\bm{S}_{q, \bm{x}}
	\bm{S}_{\bm{xx}}^{-1}
	\bm{S}_{\bm{x}, k} \right)
	\right]
	\\
	& =  2^{-2(K-1)} \sum_{q=1}^{Q} n_q^{-1} \coef_q\coef_q' \cdot S_{qq}^{\myparallel} 
	- 
	n^{-1}
	\left(2^{-(K-1)}\sum_{q=1}^{Q}\coef_q\bm{S}_{q, \bm{x}}\right)
	\bm{S}_{\bm{xx}}^{-1}
	\left(
	2^{-(K-1)}\sum_{k=1}^{Q}\bm{S}_{\bm{x}, k}\coef_{k}'
	\right)\\
	& = 2^{-2(K-1)} \sum_{q=1}^{Q} n_q^{-1} \coef_q\coef_q' \cdot S_{qq}^{\myparallel}  - n^{-1}\bm{S}_{ \bm{\tau},  \bm{x}}\bm{S}_{\bm{xx}}^{-1}\bm{S}_{\bm{x},\bm{\tau}}
	\\
	& = 2^{-2(K-1)} \sum_{q=1}^{Q} n_q^{-1} \coef_q\coef_q' \cdot S_{qq}^{\myparallel} - n^{-1}\bm{S}_{\bm{\tau\tau}}^{\myparallel} \equiv \bm{V}_{\bm{\tau}\bm{\tau}}^\myparallel. 
	\end{align*}
	Therefore, the sampling covariance of the residual from the linear projection of $\hat{\bm{\tau}}$ on $\hat{\boldsymbol{\tau}}_{\boldsymbol{x}}$ satisfies 
	\begin{align*}
	& \quad \ \bm{V}_{\bm{\tau \tau}} - \bm{V}_{\bm{\tau x}}\bm{V}_{\bm{xx}}^{-1}\bm{V}_{\bm{x\tau}} \\
	& =
	\bm{V}_{\bm{\tau \tau}} - \bm{V}_{\bm{\tau}\bm{\tau}}^\myparallel
	= 
	2^{-2(K-1)}
	\sum_{q=1}^{Q} n_q^{-1}
	\coef_q \coef_q' \cdot \left(S_{qq}-S_{qq}^{\myparallel}\right)
	- 
	n^{-1}\left(\bm{S}_{\bm{\tau\tau}} - \bm{S}_{\bm{\tau\tau}}^{\myparallel}
	\right)\\
	& = 
	2^{-2(K-1)}
	\sum_{q=1}^{Q} n_q^{-1}
	\coef_q \coef_q' \cdot S_{qq}^{\myperp}
	- 
	n^{-1} \bm{S}_{\bm{\tau\tau}}^{\myperp} \equiv
	\bm{V}_{\bm{\tau}\bm{\tau}}^\myperp. 
	\end{align*}
	Theorem \ref{thm:joint_var_expl_unexpl} holds. 
\end{proof}

\begin{proof}[{   Proof of Corollary \ref{cor:R2_f}}]
	From Proposition \ref{prop:mean_var_cre} and Theorem \ref{thm:joint_var_expl_unexpl}, under the CRFE,  
	the variance of $\hat{\tau}_f$ is $V_{\tau_f\tau_f}$, and 
	the variance of $\hat{\tau}_f$ explained by $\hat{\bm{\tau}}_{\bm{x}}$ in the linear projection is $V_{\tau_f\tau_f}^\myparallel$. Therefore, the squared multiple correlation between $\hat{\tau}_f$ and $\hat{\bm{\tau}}_{\bm{x}}$ is
	$$
	R^2(f)  
	= \frac{V_{\tau_f\tau_f}^\myparallel}{V_{\tau_f\tau_f}}
	=
	\frac{
		2^{-2(K-1)}\sum_{q=1}^{Q} n_q^{-1}S_{qq}^{\myparallel} - n^{-1} S_{\tau_f\tau_f}^{\myparallel}
	}{
		2^{-2(K-1)}\sum_{q=1}^{Q} n_q^{-1}S_{qq} - n^{-1}S_{\tau_f\tau_f}
	}.
	$$
	Under the additivity, 
	$S_{11} = \cdots = S_{QQ}$, $S_{11}^{\myparallel} = \cdots = S_{QQ}^{\myparallel}$, and 
	$S_{\tau_f\tau_f} = S_{\tau_f\tau_f}^{\myparallel}=0,$ 
	which further imply $R^2_f=S_{11}^{\myparallel}/S_{11}.$  
\end{proof}

\begin{proof}[{   Proof of Proposition \ref{prop:comp_R_fm_add}}]
	Under the additivity, 
	from Proposition \ref{prop:mean_var_cre}, the sampling variance and covariances are 
	$
	\Var(\hat{\tau}_f)  = 2^{-2(K-1)}\sum_{q=1}^{Q} n_q^{-1} S_{11},
	$
	$
	\Cov(\hat{\boldsymbol{\tau}}_{\boldsymbol{x},k})  = 2^{-2(K-1)}\sum_{q=1}^{Q}n_q^{-1}\bm{S}_{\bm{xx}}, 
	$
	and
	$
	\Cov(\hat{\tau}_f, \hat{\boldsymbol{\tau}}_{\boldsymbol{x},k})  = 
	2^{-2(K-1)}\sum_{q=1}^{Q} n_q^{-1} g_{fq}g_{kq}\bm{S}_{1,\bm{x}}.  
	$
	Let $\bm{g}_0=\bm{1}_{Q}$. 
	By the property of $2^K$ factorial experiments, for any $1\leq f,k\leq F$, there exists $0\leq m\leq F$ such that $\bm{g}_f\circ\bm{g}_k = \bm{g}_m$, recalling that $\circ$ denotes element-wise multiplication. 
	Define $f\star k$ as the index such that $\bm{g}_f\circ\bm{g}_k = \bm{g}_{f\star k}$.
	Let $w_q = n_q^{-1}/\sum_{k=1}^{Q}n_k^{-1}$ be the weight inversely proportional to the number of units under treatment combination $q$. 
	We can simplify
	the variance of $\hat{\tau}_f$ explained by $\hat{\boldsymbol{\tau}}_{\boldsymbol{x},k}$ under the CRFE as 
	\begin{eqnarray*}
		& & \Cov(\hat{\tau}_f, \hat{\boldsymbol{\tau}}_{\boldsymbol{x},k})
		\Cov^{-1}(\hat{\boldsymbol{\tau}}_{\boldsymbol{x},k})
		\Cov(\hat{\boldsymbol{\tau}}_{\boldsymbol{x},k}, \hat{\tau}_f)\\
		& = & 
		2^{-2(K-1)}
		\left(\sum_{q=1}^{Q} n_q^{-1} g_{f\star k, q}\bm{S}_{1,\bm{x}}\right)
		\left(
		\sum_{q=1}^{Q}n_q^{-1}\bm{S}_{\bm{xx}}
		\right)^{-1}
		\left(\sum_{q=1}^{Q} n_q^{-1} g_{f\star k, q}\bm{S}_{\bm{x},1}\right)\\
		& = & 
		2^{-2(K-1)}
		\left(
		\sum_{q=1}^{Q}n_q^{-1}
		\right)
		\left(\sum_{q=1}^{Q} w_q g_{f\star k, q}\right)^2S_{11}^{\myparallel}. 
	\end{eqnarray*}
	Therefore, the squared multiple correlation between $\hat{\tau}_f$ and  $\hat{\boldsymbol{\tau}}_{\boldsymbol{x},k}$ satisfies  
	\begin{align*}
	\gamma^2_{fk} & = 
	\frac{
		\Cov(\hat{\tau}_f, \hat{\boldsymbol{\tau}}_{\boldsymbol{x},k})
		\Cov^{-1}(\hat{\boldsymbol{\tau}}_{\boldsymbol{x},k})
		\Cov(\hat{\boldsymbol{\tau}}_{\boldsymbol{x},k}, \hat{\tau}_f)
	}{\Var(\hat{\tau}_f)}\\
	& = \frac{
		2^{-2(K-1)}
		\left(
		\sum_{q=1}^{Q}n_q^{-1}
		\right)
		\left(\sum_{q=1}^{Q} w_q g_{f\star k, q}\right)^2S_{11}^{\myparallel}
	}{2^{-2(K-1)}\sum_{q=1}^{Q} n_q^{-1} S_{11}}\\
	& = \left(\sum_{q=1}^{Q} w_q g_{f\star k, q}\right)^2S_{11}^{\myparallel}/S_{11}
	= 
	\left(
	\sum_{q:g_{f\star k, q}=1} w_q - \sum_{q:g_{f\star k, q}=-1} w_q
	\right)^2S_{11}^{\myparallel}/S_{11}\\
	& \leq 
	\left(
	\sum_{q=1}^Q w_q
	\right)^2S_{11}^{\myparallel}/S_{11} = S_{11}^{\myparallel}/S_{11}, 
	\end{align*}
	where the equality holds if $f\star k=0$, i.e., $k=f$. Using Corollary \ref{cor:R2_f}, we have $\gamma^2_{fk}\leq \gamma^2_{ff}=R^2_f = S_{11}^{\myparallel}/S_{11}$.  
	Moreover, because $\gamma^2_{ff}=R^2_f$, under the CRFE, the variance of $\hat{\tau}_f$ explained by $\hat{\bm{\tau}}_{\bm{x}}$ is the same as that explained by $\hat{\bm{\tau}}_{\bm{x},f}$. Therefore, the squared multiple partial correlation between 
	$\hat{\tau}_f$ and $\hat{\bm{\tau}}_{\bm{x}}$ given $\hat{\bm{\tau}}_{\bm{x},f}$ is zero. 
	If $n_1=\cdots=n_{Q} = n/Q$, then $w_1=\cdots=w_Q=Q^{-1}$, and thus  
	$
	\sum_{q=1}^{Q} w_q g_{f\star k, q} = Q^{-1} \sum_{q=1}^{Q} g_{f\star k, q} = 0
	$
	if $k\neq f$. 
	Therefore, $\gamma^2_{fk}=0$ for $k\neq f$.
\end{proof}

\begin{proof}[{   Proof of Proposition \ref{prop:simpler_refmt_f}}]
	With equal treatment group sizes, from the definition of coefficient vector $\coef_q$, we have
	\begin{align*}
	\tilde{\bm{B}} & = 2^{-2(K-1)}
	\sum_{q=1}^{Q} n_q^{-1}
	\coef_q \coef_q' 
	= 
	(4/Q^2)
	(n/Q)^{-1} \cdot
	\sum_{q=1}^{Q} 
	\coef_q \coef_q'\\
	& = 
	4/(nQ) \cdot
	\left(
	\coef_1, \ldots, \coef_Q
	\right) \cdot
	\left(
	\coef_1, \ldots, \coef_Q
	\right)' 
	\\
	& = 
	4/(nQ) \cdot 
	\left(
	\bm{g}_1, \ldots, \bm{g}_F
	\right)' \cdot
	\left(
	\bm{g}_1, \ldots, \bm{g}_F
	\right). 
	\end{align*}
	By the property of $2^K$ factorial experiments, these generating vectors satisfy that, for $1\leq f\neq k\leq F$, $\bm{g}_f'\bm{g}_k=0$ and $\bm{g}_f'\bm{g}_f = Q$. Thus, $\covcoef$ further reduces to 
	$
	\covcoef
	= 
	4/(nQ) \cdot Q \bm{I}_F =  (4/n)\bm{I}_F,
	$
	which is a diagonal matrix. Therefore, in \eqref{eq:b_q},  $\covcoef[\mathcal{F}_h, \mathcal{F}_{\overline{h-1}}]=\bm{0}$, and $\orthocoef_q[h]$ reduces to $\coef_q[\mathcal{F}_h]$. Consequently, $\hat{\bm{\theta}}_{\bm{x}}[h]$ in \eqref{eq:theta_X} reduces to $\hat{\bm{\tau}}_{\bm{x}}[\mathcal{F}_h]$. 
	Because 
	\begin{align*}
	\bm{W}_{\bm{xx}}[h] & = 
	2^{-2(K-1)}\sum_{q=1}^{Q}
	n_q^{-1}(\orthocoef_q[h]\orthocoef'_q[h]) \otimes \bm{S}_{\bm{xx}} \\
	&= 
	2^{-2(K-1)}\sum_{q=1}^{Q}
	n_q^{-1}(\coef_q[\mathcal{F}_h]\coef'_q[\mathcal{F}_h]) \otimes \bm{S}_{\bm{xx}} 
	= \covcoef[\mathcal{F}_h, \mathcal{F}_h] \otimes \bm{S}_{\bm{xx}} \\
	& = 4/n \cdot \bm{I}_{F_h} \otimes \bm{S}_{\bm{xx}}, 
	\end{align*}
	we can simplify $M_h$ as
	\begin{align*}
	M_h & =  \hat{\bm{\theta}}_{\bm{x}}'[h]
	\left(
	\bm{W}_{\bm{xx}}[h]
	\right)^{-1}
	\hat{\bm{\theta}}_{\bm{x}}[h] =  
	\hat{\bm{\tau}}_{\bm{x}}'[\mathcal{F}_h]
	\left(
	4/n \cdot \bm{I}_{F_h} \otimes \bm{S}_{\bm{xx}}
	\right)^{-1}
	\hat{\bm{\tau}}_{\bm{x}}[\mathcal{F}_h]\\
	& = n/4 \cdot
	\sum_{f\in \mathcal{F}_h} 
	\hat{\bm{\tau}}_{\bm{x},f}'
	\bm{S}_{\bm{xx}}^{-1} \hat{\bm{\tau}}_{\bm{x},f}.
	\end{align*}
	Therefore, Proposition \ref{prop:simpler_refmt_f} holds. 
\end{proof}

\section{Asymptotic sampling distributions of $\hat{\boldsymbol{\tau}}$}\label{app:dist}
\subsection{Lemmas for central convex unimodality}
We use $\mathcal{B}_m(r)$ to denote the ball in $\mathbb{R}^{m}$ with center zero and radius $r$.

\begin{lemma}\label{lemma:margi_ccu}
	The class of central convex unimodal distributions is closed under convolution, marginality, product measure, and weak convergence. 
\end{lemma}

\begin{proof}
	[Proof of Lemma \ref{lemma:margi_ccu}]
	See \citet{kanter1977}, \citet{dharmadhikari1988} and \citet{dai1989}.
\end{proof}

The following two lemmas extends Lemma \ref{lemma:margi_ccu} to all linear transformations. 
Although this extension is straightforward, we give a proof below for completeness.

\begin{lemma}\label{lemma:ccu_comb_zero}
	If $\bm{\psi} \in \mathbb{R}^m$ is central convex unimodal, then for any $p\geq 1$, $(\bm{\psi}', \bm{0}_{p\times 1}')'\in \mathbb{R}^{m+p}$ is central convex unimodal. 
\end{lemma}
\begin{proof}[Proof of Lemma \ref{lemma:ccu_comb_zero}]
	For any random vector $\bm{\phi}$ uniformly distributed on a symmetric convex body $\mathcal{K}\subset \mathbb{R}^m$, we define $\bm{\phi}_r\in \mathbb{R}^{m+p}$ as the random vector uniformly distributed on $\mathcal{K}\times \mathcal{B}_p(r)\subset \mathbb{R}^{m+p}$. Because  $\mathcal{K}\times \mathcal{B}_p(r)$ is a symmetric convex set, $\bm{\phi}_r$ is central convex unimodal. As $r$ goes to zero, $\bm{\phi}_r$ converges weakly to  $(\bm{\phi}', \bm{0}_{p\times 1}')'$. Therefore, $(\bm{\phi}', \bm{0}_{p\times 1}')'$ is central convex unimodal. 
	By taking mixtures and weak limits, we deduce Lemma \ref{lemma:ccu_comb_zero}. 
\end{proof}

\begin{lemma}\label{lemma:ccu_linear}
	If $\bm{\psi} \in \mathbb{R}^m$ is central convex unimodal, then 
	for any matrix $\bm{C}\in \mathbb{R}^{p\times m}$, $\bm{C} \bm{\psi}\in \mathbb{R}^p$ is central convex unimodal. 
\end{lemma}
\begin{proof}[Proof of Lemma \ref{lemma:ccu_linear}]
	First, we consider the case where $\bm{C}$ is an $m\times m$ invertible matrix. For any random vector $\bm{\phi}$ uniformly distributed on a symmetric convex body $\mathcal{K}\subset \mathbb{R}^m$, $\bm{C} \bm{\phi}$ is uniformly distributed on $\bm{C}\mathcal{K} \equiv \{\bm{Cx}: \bm{x} \in \mathcal{K} \}$. 
	By taking mixtures and weak limits, we can know that 
	Lemma \ref{lemma:ccu_linear} holds when $\bm{C}$ is an $m\times m$ invertible matrix. 
	
	Second, we consider the case where $\bm{C}$ has full row rank, i.e.,  $\text{rank}(\bm{C})=p\leq m$. In this case, there exists an $(m-p)\times m$ matrix $\bar{\bm{C}}$ such that $(\bm{C}', \bar{\bm{C}'})$ is an $m\times m$ invertible matrix. From the first case, we know for any central convex unimodal random vector $\bm{\psi} \in \mathbb{R}^m$,  $(\bm{C}', \bar{\bm{C}'})'\bm{\psi}$ is also central convex unimodal. From
	Lemma \ref{lemma:margi_ccu}, its subvector $\bm{C} \bm{\psi}$ is also central convex unimodal.

	Third, we consider a general matrix $\bm{C}$ of rank $p_1$. 
	Let $\bm{W}$ be a matrix permuting the rows of $\bm{C}$ such that the first $p_1$ rows of $\bm{WC}$ is linearly independent, and the remaining $p_2=p-p_1$ rows of $\bm{WC}$ are all linear combinations of the first $p_1$ rows, i.e., $(\bm{WC})'=(\bm{C}_1', \bm{C}_2')$, where 
	$\bm{C}_1\in \mathbb{R}^{p_1\times m}$, $\bm{C}_2\in \mathbb{R}^{p_2\times m}$, $\text{rank}(\bm{C}_1)=p_1$, and $\bm{C}_2 = \bm{\Gamma}\bm{C}_1$. Thus,
	\begin{align*}
	\bm{C}\bm{\psi} & = 
	\bm{W}^{-1}(\bm{W}
	\bm{C})\bm{\psi}
	=
	\bm{W}^{-1}
	\begin{pmatrix}
	\bm{C}_1 \\
	\bm{C}_2
	\end{pmatrix}
	\bm{\psi} 
	= 
	\bm{W}^{-1}
	\begin{pmatrix}
	\bm{C}_1 \\
	\bm{\Gamma}\bm{C}_1
	\end{pmatrix}
	\bm{\psi} 
	= 
	\bm{W}^{-1}
	\begin{pmatrix}
	\bm{C}_1 \bm{\psi} \\
	\bm{\Gamma}\bm{C}_1\bm{\psi} 
	\end{pmatrix}
	\\
	& = 
	\bm{W}^{-1}
	\begin{pmatrix}
	\bm{I}_{p_1} & \bm{0}_{p_1\times p_2} \\
	\bm{\Gamma}& \bm{I}_{p_2}
	\end{pmatrix}
	\begin{pmatrix}
	\bm{C}_1\bm{\psi} \\
	\bm{0}_{p_2\times 1}
	\end{pmatrix}
	\equiv \tilde{\bm{\Gamma}}\begin{pmatrix}
	\bm{C}_1\bm{\psi} \\
	\bm{0}_{p_2\times 1}
	\end{pmatrix}.
	\end{align*}
	From the second case, $\bm{C}_1\bm{\psi}$ is central convex unimodal.  
	From Lemma \ref{lemma:ccu_comb_zero}, 
	$
	((\bm{C}_1\bm{\psi})',
	\bm{0}_{p_2\times 1}')
	$
	is also central convex unimodal. Because the $p\times p$ matrix 
	$
	\tilde{\bm{\Gamma}}
	$
	is invertible, from the first case, we know $\bm{C}\bm{\psi}$ is central convex unimodal. 
\end{proof}

\begin{lemma}\label{lemma:ccu_sum}
	Let $\bm{\psi}_1$ and $\bm{\psi}_2$ be two independent $m$ dimensional random vectors. If both $\bm{\psi}_1$ and $\bm{\psi}_2$ are central convex unimodal, then $\bm{\psi}_1+\bm{\psi}_2$ is central convex unimodal. 
\end{lemma}

\begin{proof}
	[Proof of Lemma \ref{lemma:ccu_sum}]
	See \citet[][Theorem 2.20]{dharmadhikari1988}. 
\end{proof}

\begin{lemma}\label{lemma:logconcave}
	If a random vector in $\mathbb{R}^m$ has a log-concave density, then it is central convex unimodal. 
\end{lemma}

\begin{proof}
	[Proof of Lemma \ref{lemma:logconcave}]
	See \citet[][Lemma 3.1]{kanter1977} and \citet[][Theorem 2.15]{dharmadhikari1988}. 
\end{proof}

\subsection{Properties of the truncated Gaussian random vectors}
\begin{proof}[{   Proof of Proposition \ref{prop:ccu_refm}}]
	For any $m $, 
	the densities of $\bm{\varepsilon}$ and $\bm{\zeta}_{m,a}$ are
	\begin{align*}
	f(\bm{x}) = (2\pi)^{-m/2}\exp\left(-\bm{x}'\bm{x}/2\right), \quad 
	g(\bm{x}) = \frac{1\{\bm{x}'\bm{x}\leq a\}}{P(\chi^2_m\leq a)}(2\pi)^{-m/2}\exp\left(-\bm{x}'\bm{x}/2\right),
	\end{align*}
	and the log-densities of $\bm{\varepsilon}$ and $\bm{\zeta}_{m,a}$ are 
	\begin{align*}
	\log f(\bm{x}) & = -(m/2)\log(2\pi)-\bm{x}'\bm{x}/2, \\
	\log g(\bm{x}) & = 
	\begin{cases}
	-\log\left\{P(\chi^2_m\leq a)\right\} -(m/2)\log(2\pi)-\bm{x}'\bm{x}/2, & \text{if } \bm{x}'\bm{x}\leq a, \\
	-\infty, & \text{otherwise.}
	\end{cases}
	\end{align*}
	It is straightforward to show that both $\log f(\cdot)$ and $\log g(\cdot)$ are concave functions. From Lemma \ref{lemma:logconcave}, both $\bm{\varepsilon}$ and $\bm{\zeta}_{m,a}$ are central convex unimodal. 
	Because both $\bm{\varepsilon}$ and $\bm{\zeta}_{LF,a}$ are central convex unimodal,  using Lemma \ref{lemma:ccu_linear}, 
	both
	$(\bm{V}_{\bm{\tau\tau}}^\myperp)^{1/2} \bm{\varepsilon}$ and 
	$(\bm{V}_{\bm{\tau\tau}}^\myparallel)^{1/2}_{LF}
	\bm{\zeta}_{LF,a}$ 
	are central convex unimodal. 
	From Lemma \ref{lemma:ccu_sum}, we have Proposition \ref{prop:ccu_refm}. 
\end{proof}

\begin{proof}[{   Proof of Proposition \ref{prop:represent}}]
	First, we prove that $\bm{\zeta}_{LF,a}$ is spherically symmetric. 
	Let $\bm{D}\sim \mathcal{N}(\bm{0}, \bm{I}_{LF})$. 
	For any orthogonal matrix $\bm{\Gamma}$, $\bm{\Gamma}\bm{D} \sim \bm{D}$, and thus, 
	$$
	\bm{\zeta}_{LF,a} \sim \bm{D} \mid \bm{D}'\bm{D}\leq a \sim \bm{\Gamma D} \mid (\bm{\Gamma}\bm{D})'\bm{\Gamma}\bm{D}\leq a 
	\sim \bm{\Gamma D} \mid \bm{D}'\bm{D}\leq a 
	\sim \bm{\Gamma} \bm{\zeta}_{LF,a}. 
	$$
	Second, 
	from \citet[][Theorem 3.1]{morgan2012rerandomization}, $\Cov(\bm{\zeta}_{LF,a})=v_{LF,a}\bm{I}_{LF}.$
	Third, we show the representation for  $\bm{\zeta}_{LF,a}$. 
	By the spherical symmetry of the standard Gaussian random vector $\bm{D}$, $(\bm{D}'\bm{D})^{-1/2}\bm{D} \sim \bm{U}=(U_1, \ldots, U_{LF})'$ is uniformly distributed on the $LF$ dimensional unit sphere, $\bm{D}'\bm{D}$ follows $\chi^2_{LF}$, and they are independent. Therefore, 
	\begin{align*}
	\bm{D} \mid \bm{D}'\bm{D} \leq a 
	\ \sim \ 
	(\bm{D}'\bm{D})^{1/2} \cdot (\bm{D}'\bm{D})^{-1/2}\bm{D} \mid \bm{D}'\bm{D}\leq a 
	\ \sim \ 
	\chi_{LF,a} \bm{U}. 
	\end{align*}
	Let $\text{sign}(U_{i})$ be the sign of $U_{i}$. 
	Given $(|U_{1}|, \ldots, |U_{LF}|)$, 
	the sign vector $(\text{sign}(U_{1}),$ $\ldots, \text{sign}(U_{LF}))$ has the same probability to be any value in $\{-1,1\}^{m}$. 
	Thus, $(\text{sign}(U_{1}), \ldots, \text{sign}(U_{LF})) \sim \bm{S}$, and it is independent of $(|U_{1}|, \ldots, |U_{LF}|)$. 
	Because $\bm{U}  \sim (\bm{D}'\bm{D})^{-1/2}\bm{D}$, and $(D_1^2, \ldots, D_{LF}^2)$ are independent and identically Gamma distributed with shape parameter $1/2$ and rate parameter $1/2$, we have 
	\begin{align*}
	(U_{1}^2, \ldots, U_{LF}^2)' = (\bm{D}'\bm{D})^{-1} \cdot
	(D_1^2, \ldots, D_{LF}^2)' \sim \text{Dirichlet}(1/2, \ldots, 1/2) \sim \bm{\beta}. 
	\end{align*}
	Therefore, $\bm{U} = (\text{sign}(U_{1}), \ldots, \text{sign}(U_{LF}))' \circ (|U_{1}|, \ldots, |U_{LF}|)' \sim \bm{S}\circ \sqrt{\bm{\beta}}$, which further implies $\bm{\zeta}_{LF,a} \sim \chi_{LF,a} \cdot \bm{S}\circ\sqrt{\bm{\beta}}$. 
\end{proof}

The following lemma helps to simplify the asymptotic sampling distributions of $\hat{\bm{\tau}}$ under the CRFE, $\text{ReFMT}_\text{F}$ and $\text{ReFMT}_\text{CF}$.

\begin{lemma}\label{lemma:truncate_AApm_BBpm}
	Let $\bm{\zeta}_{m,a}\sim \bm{D}\mid \bm{D}'\bm{D}\leq a$ be an $m$ dimensional truncated Gaussian random vector, where $\bm{D}\sim \mathcal{N}(\bm{0}, \bm{I}_m)$. If two matrices $\bm{A}$ and $\bm{B}$ in $\mathbb{R}^{p\times m}$ satisfy $\bm{A}\bm{A}'=\bm{B}\bm{B'}$, then $\bm{A}\bm{\zeta}_{m,a}\sim \bm{B}\bm{\zeta}_{m,a}$.
\end{lemma}
\begin{proof}[Proof of Lemma \ref{lemma:truncate_AApm_BBpm}]
	From Lemma \ref{lemma:ortho_AApm_BBpm}, there exists an orthogonal matrix $\bm{\Gamma}\in \mathbb{R}^{m\times m}$ such that $\bm{A}=\bm{B}\bm{\Gamma}$. 
	From Proposition \ref{prop:represent}, by the spherical symmetry of $\bm{\zeta}_{m,a}$, 
	$\bm{A}\bm{\zeta}_{m,a} = \bm{B}\bm{\Gamma}\bm{\zeta}_{m,a} \sim  \bm{B}\bm{\zeta}_{m,a}.$
\end{proof}

\subsection{Asymptotic sampling distributions of $\hat{\boldsymbol{\tau}}$ under ReFM}

To prove Theorem \ref{thm:joint_refm}, we need the following three lemmas.
We further introduce the following regularity condition for a general covariate balance criteria depending only on $\hat{\bm{\tau}}_{\bm{x}}$ and $\bm{V}_{\bm{xx}}$. 
\begin{condition}\label{cond:cov_balance_criteria}
	Let $\bm{B} \sim \mathcal{N}(\bm{0},\bm{\Lambda})$. For any $\bm{\Lambda}>0$, 
	the 0-1 function $\kappa(\sqrt{n}\hat{\bm{\tau}}_{\bm{x}}, n\bm{V}_{\bm{xx}}) \equiv \kappa(\hat{\bm{\tau}}_{\bm{x}}, \bm{V}_{\bm{xx}})$ satisfies
	\begin{itemize}
		\item[(a)] $\kappa$ is almost surely continuous, 
		\item[(b)] $\Var\{\bm{B}\mid \kappa(\bm{B},\bm{\Lambda})=1\}$, as a function of $\bm{\Lambda}$, is continuous, 
		\item[(c)] $P\{\kappa(\bm{B},\bm{\Lambda})=1\}>0$, 
		\item[(d)]$\kappa(\bm{\mu},\bm{\Lambda})=\kappa(-\bm{\mu},\bm{\Lambda})$, for all $\bm{\mu}$. 
	\end{itemize}
\end{condition}

We can verify that the balance criteria for ReFM, $\text{ReFMT}_\text{F}$ and $\text{ReFMT}_\text{CF}$ depend only on $\hat{\bm{\tau}}_{\bm{x}}$ and $\bm{V}_{\bm{xx}}$ and satisfy Condition \ref{cond:cov_balance_criteria}.

\begin{lemma}\label{lemma:condi_converge}
	Let $(\sqrt{n}\tilde{\bm{\tau}}', \sqrt{n}\tilde{\bm{\tau}}_{\bm{x}}')$ be a random vector following $\mathcal{N}(0, n\bm{V})$. Then 
	as $n\rightarrow\infty$, under Conditions \ref{cond:fp} and \ref{cond:cov_balance_criteria}, 
	the two conditional distributions, 
	$
	\sqrt{n}(\hat{\bm{\tau}}'-\bm{\tau}', \hat{\bm{\tau}}_{\bm{x}}') \mid \kappa(\hat{\bm{\tau}}_{\bm{x}}, \bm{V}_{\bm{xx}})=1
	$ 
	and 
	$
	\sqrt{n}(\tilde{\bm{\tau}}', \tilde{\bm{\tau}}_{\bm{x}}') \mid \kappa(\tilde{\bm{\tau}}_{\bm{x}}, \bm{V}_{\bm{xx}})=1,
	$ 
	converge weakly to the same distribution, i.e., 
	\begin{eqnarray*}
		\left.
		\begin{pmatrix}
			\hat{\bm{\tau}}-\bm{\tau} \\
			\hat{\bm{\tau}}_{\bm{x}}
		\end{pmatrix}\ 
		\right| \ \kappa(\hat{\bm{\tau}}_{\bm{x}}, \bm{V}_{\bm{xx}})=1
		& \apprsim & 
		\left.
		\begin{pmatrix}
			\tilde{\bm{\tau}} \\
			\tilde{\bm{\tau}}_{\bm{x}}
		\end{pmatrix}\ 
		\right| \ \kappa(\tilde{\bm{\tau}}_{\bm{x}}, \bm{V}_{\bm{xx}})=1. 
	\end{eqnarray*}
\end{lemma}
\begin{proof}[Proof of Lemma \ref{lemma:condi_converge}]
	The proof of this lemma is almost identical to \citet[][Proposition A1 and Corollary A1]{asymrerand2016}. We omit it. 
\end{proof}

\begin{proof}[{   Proof of Theorem \ref{thm:joint_refm}}]
	From Lemma \ref{lemma:condi_converge}, under Condition \ref{cond:fp}, 
	$
	\hat{\bm{\tau}} - \bm{\tau} \mid  \mathcal{M}  \apprsim  
	\tilde{\bm{\tau}} \mid \tilde{\bm{\tau}}_{\bm{x}}' \bm{V}_{\bm{xx}}^{-1}\tilde{\bm{\tau}}_{\bm{x}}\leq a,  
	$
	where $(\tilde{\bm{\tau}}', \tilde{\bm{\tau}}_{\bm{x}}') \sim \mathcal{N}(\bm{0}, \bm{V}).$ 
	The linear projection of $\tilde{\bm{\tau}}$ on $\tilde{\bm{\tau}}_{\bm{x}}$ is $\bm{V}_{\bm{\tau}\bm{x}}\bm{V}_{\bm{xx}}^{-1}\tilde{\bm{\tau}}_{\bm{x}},$
	and the corresponding residual is 
	$\tilde{\bm{\tau}}_{\bm{\varepsilon}} = \tilde{\bm{\tau}} - \bm{V}_{\bm{\tau}\bm{x}}\bm{V}_{\bm{xx}}^{-1}\tilde{\bm{\tau}}_{\bm{x}}. $
	From Theorem \ref{thm:joint_var_expl_unexpl}, $\tilde{\bm{\tau}}_{\bm{\varepsilon}}$ has variance $\bm{V}_{\bm{\tau\tau}}^\myperp.$ 
	Let $\bm{D}=\bm{V}_{\bm{xx}}^{-1/2}\tilde{\bm{\tau}}_{\bm{x}} \sim \mathcal{N}(\bm{0}, \bm{I}_{LF})$ be the standardization of $\tilde{\bm{\tau}}_{\bm{x}}$. We can then simplify the asymptotic sampling distribution of $\hat{\bm{\tau}}$ under ReFM as:
	\begin{eqnarray*}
		\hat{\bm{\tau}} - \bm{\tau} \mid  \mathcal{M} & \apprsim & 
		\tilde{\bm{\tau}} \mid \tilde{\bm{\tau}}_{\bm{x}}' \bm{V}_{\bm{xx}}^{-1}\tilde{\bm{\tau}}_{\bm{x}}\leq a 
		\ \sim \ 
		\tilde{\bm{\tau}}_{\bm{\varepsilon}} + \bm{V}_{\bm{\tau}\bm{x}}\bm{V}_{\bm{xx}}^{-1}\tilde{\bm{\tau}}_{\bm{x}} 
		\mid \tilde{\bm{\tau}}_{\bm{x}}' \bm{V}_{\bm{xx}}^{-1}\tilde{\bm{\tau}}_{\bm{x}}\leq a \\
		& \sim & 
		\tilde{\bm{\tau}}_{\bm{\varepsilon}} + \bm{V}_{\bm{\tau}\bm{x}}\bm{V}_{\bm{xx}}^{-1/2}\bm{D} 
		\mid \bm{D}'\bm{D}\leq a \\
		& \sim &  
		\left(\bm{V}_{\bm{\tau\tau}}^\myperp\right)^{1/2}\bm{\varepsilon} +  \bm{V}_{\bm{\tau}\bm{x}}\bm{V}_{\bm{xx}}^{-1/2}\bm{\zeta}_{LF,a},
	\end{eqnarray*}
	where $\bm{\varepsilon}\sim \mathcal{N}(0, \bm{I}_F)$ is independent of $\bm{\zeta}_{LF,a}$. 
	From Theorem \ref{thm:joint_var_expl_unexpl}, 
	$
	\bm{V}_{\bm{\tau\tau}}^\myparallel = \bm{V}_{\bm{\tau}\bm{x}}\bm{V}_{\bm{xx}}^{-1}\bm{V}_{\bm{x}\bm{\tau}} = 
	\bm{V}_{\bm{\tau}\bm{x}}\bm{V}_{\bm{xx}}^{-1/2}(\bm{V}_{\bm{\tau}\bm{x}}\bm{V}_{\bm{xx}}^{-1/2})'.
	$
	Using Lemma \ref{lemma:truncate_AApm_BBpm}, we have 
	$\bm{V}_{\bm{\tau}\bm{x}}\bm{V}_{\bm{xx}}^{-1/2}\bm{\zeta}_{LF,a} \sim 
	(\bm{V}_{\bm{\tau\tau}}^\myparallel)^{1/2}_{LF}\bm{\zeta}_{LF,a}.
	$ 
	Therefore, Theorem \ref{thm:joint_refm} holds. 
\end{proof}

\begin{proof}[{   Proof of Corollary \ref{cor:dist_refm_one_linear}}]
	Let $\bm{e}_f=(0, \ldots, 0, 1, 0, \ldots, 0)'$ be an $F$ dimensional unit vector with $f$th element one. From Theorem \ref{thm:joint_refm}, 
	\begin{eqnarray*}
		\hat{\tau}_f - \tau_f \mid  \mathcal{M} & \apprsim & 
		\bm{e}_f'\left(\bm{V}_{\bm{\tau\tau}}^\myperp\right)^{1/2} \bm{\varepsilon} +\bm{e}_f' (\bm{V}_{\bm{\tau\tau}}^\myparallel)^{1/2}_{LF} \bm{\zeta}_{LF,a}\\
		& \sim & 
		\sqrt{\bm{e}_f'\bm{V}_{\bm{\tau\tau}}^\myperp\bm{e}_f} \cdot
		\varepsilon_0 + 
		\sqrt{\bm{e}_f' \bm{V}_{\bm{\tau\tau}}^\myparallel \bm{e}_f} \cdot \bm{c}' \bm{\zeta}_{LF,a}, 
	\end{eqnarray*}
	where 
	$\bm{c}'=(\bm{e}_f'\bm{V}_{\bm{\tau\tau}}^\myparallel\bm{e}_f)^{-1/2}\bm{e}_f' (\bm{V}_{\bm{\tau\tau}}^\myparallel)^{1/2}_{LF}$  is a unit vector with length one. By the spherical symmetry of $\bm{\zeta}_{LF,a}$ in Proposition \ref{prop:represent}, $\bm{c}' \bm{\zeta}_{LF,a}\sim \eta_{LF,a}$ \citep[][Lemma A1]{asymrerand2016}.
	By the definition of $R^2_f$ in Corollary \ref{cor:R2_f}, we can further simplify the asymptotic sampling distribution of $\hat{\tau}_f$ as 
	\begin{eqnarray*}
		\hat{\tau}_f - \tau_f \mid  \mathcal{M} & \apprsim & 
		\sqrt{\bm{e}_f'\bm{V}_{\bm{\tau\tau}}\bm{e}_f - \bm{e}_f'\bm{V}_{\bm{\tau\tau}}^\myparallel\bm{e}_f} \cdot
		\varepsilon_0 + 
		\sqrt{\bm{e}_f' \bm{V}_{\bm{\tau\tau}}^\myparallel\bm{e}_f} \cdot  \eta_{LF,a}\\
		& \sim & 
		\sqrt{V_{\tau_f\tau_f} - V_{\tau_f\tau_f}^\myparallel} \cdot \varepsilon_0 + 
		\sqrt{V_{\tau_f\tau_f}^\myparallel}\cdot  \eta_{LF,a}
		\\
		& \sim & 
		\sqrt{V_{\tau_f\tau_f}}
		\left(
		\sqrt{1 - {V_{\tau_f\tau_f}^\myparallel
			}/{V_{\tau_f\tau_f}}} \cdot \varepsilon_0 + 
		\sqrt{{V_{\tau_f\tau_f}^\myparallel
			}/{V_{\tau_f\tau_f}}} \cdot \eta_{LF,a}
		\right)\\
		& \sim & 
		\sqrt{V_{\tau_f\tau_f}}
		\left(
		\sqrt{1-R^2_f}\cdot \varepsilon_0 + \sqrt{R^2_f} \cdot \eta_{LF,a}
		\right). 
	\end{eqnarray*}
	Therefore, Corollary \ref{cor:dist_refm_one_linear} holds. 
\end{proof}

\begin{proof}[{   Proof of Corollary \ref{cor:another_form_refm}}]
	Recall that $\bm{R} = \bm{V}_{\bm{\tau\tau}}^{-1/2}
	\bm{V}_{\bm{\tau\tau}}^\myparallel\bm{V}_{\bm{\tau\tau}}^{-1/2} = \bm{\Gamma}\bm{\Pi}^2\bm{\Gamma}'$ is the eigen-decomposition of $\bm{R}$, where $\bm{\Gamma}\in \mathbb{R}^{F\times F}$ is an orthogonal matrix, and $\bm{\Pi}^2\in \mathbb{R}^{F\times F}$ is a diagonal matrix containing the eigenvalues of $\bm{R}$. Let $\bm{\Omega} = \bm{\Pi}^{-1}\bm{\Gamma}'\bm{V}_{\bm{\tau\tau}}^{-1/2}\bm{V}_{\bm{\tau x}}\bm{V}_{\bm{xx}}^{-1/2}$. Then
	$
	\bm{\Omega}\bm{\Omega}' = 
	\bm{\Pi}^{-1}\bm{\Gamma}'\bm{R}
	\bm{\Gamma}\bm{\Pi}^{-1} = \bm{I}_F, 
	$
	and 
	$
	\bm{V}_{\bm{\tau\tau}}^{-1/2}\bm{V}_{\bm{\tau x}}\bm{V}_{\bm{xx}}^{-1/2} = \bm{\Gamma}\bm{\Pi}\bm{\Omega}
	$ is the singular value decomposition of $\bm{V}_{\bm{\tau\tau}}^{-1/2}\bm{V}_{\bm{\tau x}}\bm{V}_{\bm{xx}}^{-1/2}$. 
	Note that 
	$
	\bm{V}_{\bm{\tau\tau}}^{-1/2}
	\bm{V}_{\bm{\tau\tau}}^\myperp\bm{V}_{\bm{\tau\tau}}^{-1/2}  = \bm{I}_F - \bm{R} = \bm{\Gamma}(\bm{I}_F-\bm{\Pi}^2)\bm{\Gamma}'.
	$
	We can simplify the asymptotic distribution \eqref{eq:joint_refm} as
	\begin{eqnarray*}
		\hat{\bm{\tau}} - \bm{\tau} \mid  \mathcal{M} & \apprsim & 
		\left(\bm{V}_{\bm{\tau\tau}}^\myperp\right)^{1/2} \bm{\varepsilon} +  \bm{V}_{\bm{\tau x}}\bm{V}_{\bm{xx}}^{-1/2} \bm{\zeta}_{LF,a}\\
		& = & \bm{V}_{\bm{\tau\tau}}^{1/2}\left\{
		\bm{V}_{\bm{\tau\tau}}^{-1/2}\left(\bm{V}_{\bm{\tau\tau}}^\myperp\right)^{1/2} \bm{\varepsilon} + 
		\bm{V}_{\bm{\tau\tau}}^{-1/2}\bm{V}_{\bm{\tau x}}\bm{V}_{\bm{xx}}^{-1/2} \bm{\zeta}_{LF,a}
		\right\}\\
		& \sim & \bm{V}_{\bm{\tau\tau}}^{1/2}\left\{
		\left(\bm{V}_{\bm{\tau\tau}}^{-1/2}
		\bm{V}_{\bm{\tau\tau}}^\myperp\bm{V}_{\bm{\tau\tau}}^{-1/2}\right)^{1/2}
		\bm{\varepsilon} + 
		\bm{V}_{\bm{\tau\tau}}^{-1/2}\bm{V}_{\bm{\tau x}}\bm{V}_{\bm{xx}}^{-1/2} \bm{\zeta}_{LF,a}
		\right\}\\
		& = & 
		\bm{V}_{\bm{\tau\tau}}^{1/2}\left[
		\left\{\bm{\Gamma}(\bm{I}_F-\bm{\Pi}^2)\bm{\Gamma}'\right\}^{1/2}
		\bm{\varepsilon} + 
		\bm{\Gamma}\bm{\Pi}\bm{\Omega} \bm{\zeta}_{LF,a}
		\right]\\
		& \sim & 
		\bm{V}_{\bm{\tau\tau}}^{1/2}\left\{
		\bm{\Gamma}(\bm{I}_F-\bm{\Pi}^2)^{1/2}
		\bm{\varepsilon} + 
		\bm{\Gamma}\bm{\Pi}\bm{\Omega} \bm{\zeta}_{LF,a}
		\right\}. 
	\end{eqnarray*}
	Because $\bm{\Omega}\bm{\Omega}' = \bm{I}_F,$ 
	there exists a matrix $\tilde{\bm{\Omega}}\in \mathbb{R}^{(L-1)F\times LF}$ such that $(\bm{\Omega}', \tilde{\bm{\Omega}}') \in \mathbb{R}^{LF\times LF}$ is orthogonal. From Proposition \ref{prop:represent}, 
	$(\bm{\Omega}', \tilde{\bm{\Omega}}')'\bm{\zeta}_{LF,a} \sim \bm{\zeta}_{LF,a}$. 
	Therefore, we can simplify the asymptotic distribution \eqref{eq:joint_refm} as
	\begin{eqnarray*}
		\hat{\bm{\tau}} - \bm{\tau} \mid  \mathcal{M} & \apprsim & \bm{V}_{\bm{\tau\tau}}^{1/2}\left\{
		\bm{\Gamma}(\bm{I}_F-\bm{\Pi}^2)^{1/2}
		\bm{\varepsilon} + 
		\bm{\Gamma}\bm{\Pi}\bm{\Omega} \bm{\zeta}_{LF,a}
		\right\}\\
		& \sim & 
		\bm{V}_{\bm{\tau\tau}}^{1/2}\bm{\Gamma}\left\{
		(\bm{I}_F-\bm{\Pi}^2)^{1/2}
		\bm{\varepsilon} + 
		(
		\bm{\Pi}, \bm{0}
		)
		\begin{pmatrix}
			\bm{\Omega}\\
			\tilde{\bm{\Omega}}
		\end{pmatrix}
		\bm{\zeta}_{LF,a}
		\right\}\\
		& \sim & 
		\bm{V}_{\bm{\tau\tau}}^{1/2}\bm{\Gamma}\left\{
		(\bm{I}_F-\bm{\Pi}^2)^{1/2}
		\bm{\varepsilon} + 
		(
		\bm{\Pi}, \bm{0}
		)
		\bm{\zeta}_{LF,a}
		\right\}. 
	\end{eqnarray*}
	Therefore, Corollary \ref{cor:another_form_refm} holds. 
\end{proof}


\subsection{Asymptotic sampling distribution of $\hat{\boldsymbol{\tau}}$ under $\text{ReFMT}_\text{F}$}

\begin{proof}[{   Proof of Theorem \ref{thm:asymp_refmt_f}}]
	Let $(\tilde{\bm{\tau}}', \tilde{\bm{\theta}}_{\bm{x}}')' = (\tilde{\bm{\tau}}', \tilde{\bm{\theta}}_{\bm{x}}'[1], \ldots, \tilde{\bm{\theta}}_{\bm{x}}'[H])'$ be a random vector
	following Gaussian distribution with mean zero and covariance matrix
	\begin{align*}
	\begin{pmatrix}
	\bm{V}_{\bm{\tau\tau}} & \bm{W}_{\bm{\tau x}}[1]  & \ldots & \bm{W}_{\bm{\tau x}}[H]\\
	\bm{W}_{\bm{x\tau}}[1] & \bm{W}_{\bm{xx}}[1]  & \ldots & \bm{0}\\
	\vdots & \vdots & \ddots & \vdots\\
	\bm{W}_{\bm{x\tau}}[H] & \bm{0} & \ldots & \bm{W}_{\bm{xx}}[H]
	\end{pmatrix}.
	\end{align*}
	From Lemma \ref{lemma:condi_converge}, 
	under Condition \ref{cond:fp},  
	\begin{eqnarray*}
		\hat{\bm{\tau}} - \bm{\tau}  \mid \mathcal{T}_{\text{F}} 
		& \apprsim &
		\tilde{\bm{\tau}} \  \left| \ 
		\left\{\tilde{\bm{\theta}}_{\bm{x}}'[h]
		\left(
		\bm{W}_{\bm{xx}}[h]
		\right)^{-1}
		\tilde{\bm{\theta}}_{\bm{x}}[h]\leq a_h, 1\leq h\leq H \right\}, \right.   
	\end{eqnarray*}
	where $\tilde{\bm{\theta}}_{\bm{x}}[h]$'s are mutually uncorrelated. The linear projection of $\tilde{\bm{\tau}}$ on $\tilde{\bm{\theta}}_{\bm{x}}$ is $\sum_{h=1}^H \bm{W}_{\bm{\tau x}}[h]
	(
	\bm{W}_{\bm{xx}}[h]
	)^{-1} \tilde{\bm{\theta}}_{\bm{x}}[h]$. 
	Let $\tilde{\bm{\tau}}_{\bm{\varepsilon}}$ be the corresponding residual, which, by the identical covariance structure between $(\tilde{\bm{\tau}}', \tilde{\bm{\theta}}_{\bm{x}}')'$ and  $(\hat{\bm{\tau}}-\bm{\tau}', \hat{\bm{\theta}}_{\bm{x}}')'$, has the same covariance as the sampling covariance of $\hat{\bm{\tau}}_{\bm{\varepsilon}}$, the residual from the linear projection of $\hat{\bm{\tau}}$ on $\hat{\bm{\theta}}_{\bm{x}}$. 
	Because $\hat{\bm{\theta}}_{\bm{x}}$ and $\hat{\bm{\tau}}_{\bm{x}}$ are linear transformations of each other, $\hat{\bm{\tau}}_{\bm{\varepsilon}}$ is the same as the  residual from the linear projection of $\hat{\bm{\tau}}$ on $\hat{\bm{\tau}}_{\bm{x}}$. Thus, from Theorem \ref{thm:joint_var_expl_unexpl}, 
	$
	\Cov(\tilde{\bm{\tau}}_{\bm{\varepsilon}}) = 
	\Cov(\hat{\bm{\tau}}_{\bm{\varepsilon}}) = \bm{V}_{\bm{\tau\tau}}^\myperp. 
	$
	Let $\bm{\varepsilon} = (\bm{V}_{\bm{\tau\tau}}^{\myperp})^{-1/2}\tilde{\bm{\tau}}_{\bm{\varepsilon}}
	\sim \mathcal{N}(0, \bm{I}_F)
	$
	be the standardization of $\tilde{\bm{\tau}}_{\bm{\varepsilon}}$, and $\bm{D}_h = (
	\bm{W}_{\bm{xx}}[h]
	)^{-1/2}\tilde{\bm{\theta}}_{\bm{x}}[h]$ be the standardization of $\tilde{\bm{\theta}}_{\bm{x}}[h]$. We have
	\begin{eqnarray*}
		& & \hat{\bm{\tau}} - \bm{\tau}  \mid \mathcal{T}_{\text{F}} \\
		& \apprsim &
		\tilde{\bm{\tau}} \ \left| \ 
		\left\{
		\tilde{\bm{\theta}}_{\bm{x}}'[h]
		\left(
		\bm{W}_{\bm{xx}}[h]
		\right)^{-1}
		\tilde{\bm{\theta}}_{\bm{x}}[h]\leq a_h, h=1,\ldots, H
		\right\}\right.
		\\
		& \sim & 
		\tilde{\bm{\tau}}_{\bm{\varepsilon}} + \sum_{h=1}^H \bm{W}_{\bm{\tau x}}[h]
		(
		\bm{W}_{\bm{xx}}[h]
		)^{-1} \tilde{\bm{\theta}}_{\bm{x}}[h] \  \left| \ 
		\left\{
		\tilde{\bm{\theta}}_{\bm{x}}'[h]
		\left(
		\bm{W}_{\bm{xx}}[h]
		\right)^{-1}
		\tilde{\bm{\theta}}_{\bm{x}}[h]\leq a_h, 1\leq h\leq H
		\right\} \right.\\
		& \sim & 
		(\bm{V}_{\bm{\tau\tau}}^{\myperp})^{1/2}\bm{\varepsilon} + \sum_{h=1}^H \bm{W}_{\bm{\tau x}}[h]
		(
		\bm{W}_{\bm{xx}}[h]
		)^{-1/2} \bm{D}_h  \ \left| \ 
		\left\{
		\bm{D}_h' \bm{D}_h\leq a_h, 1\leq h\leq H
		\right\}\right.
		\\
		& \sim & 
		(\bm{V}_{\bm{\tau\tau}}^{\myperp})^{1/2}\bm{\varepsilon} + \sum_{h=1}^H \bm{W}_{\bm{\tau x}}[h]
		(
		\bm{W}_{\bm{xx}}[h]
		)^{-1/2} \bm{\zeta}_{LF_h, a_h}. 
	\end{eqnarray*}
	Because 
	$
	\bm{W}_{\bm{\tau \tau}}^\myparallel[h]
	= 
	\bm{W}_{\bm{\tau x}}[h]
	(
	\bm{W}_{\bm{xx}}[h]
	)^{-1/2}
	\{\bm{W}_{\bm{\tau x}}[h]
	(
	\bm{W}_{\bm{xx}}[h]
	)^{-1/2}\}'
	$ 
	by definition,
	Theorem \ref{thm:asymp_refmt_f} holds by Lemma \ref{lemma:truncate_AApm_BBpm}. 
\end{proof}

\begin{proof}[{   Proof of Corollary \ref{cor:dist_remt}}]
	Recall that $\bm{e}_f$ is a unit vector with the $f$th coordinate being one. 
	From Theorem \ref{thm:asymp_refmt_f}, the asymptotic sampling distribution of $\hat{\tau}_f$ under $\text{ReFMT}_\text{F}$ is 
	\begin{eqnarray*}
		& & \hat{\tau}_f - \tau_f \mid \mathcal{T}_{\text{F}} \\
		& \apprsim &
		\bm{e}_f'
		\left(
		\bm{V}_{\bm{\tau}\bm{\tau}}^\myperp \right)^{1/2} \bm{\varepsilon} + 
		\sum_{h=1}^H 
		\bm{e}_f'
		\left(\bm{W}_{\bm{\tau\tau}}^\myparallel[h] \right)^{1/2}_{LF_h}
		\bm{\zeta}_{LF_h,a_h}\\
		& \sim & 
		\sqrt{
			\bm{e}_f'\bm{V}_{\bm{\tau\tau}}^\myperp\bm{e}_f
		}\cdot
		\varepsilon_0 + 
		\sum_{h=1}^H 
		\sqrt{\bm{e}_f' \bm{W}_{\bm{\tau\tau}}^\myparallel[h] \bm{e}_f} \cdot
		\bm{c}_h' \bm{\zeta}_{LF_h,a_h}\\
		& \sim & 
		\sqrt{V_{\tau_f\tau_f}} \left(
		\sqrt{1 - V_{\tau_f\tau_f}^\myparallel/V_{\tau_f\tau_f}}\cdot
		\varepsilon_0 + 
		\sum_{h=1}^H 
		\sqrt{W_{\tau_f\tau_f}^\myparallel[h]/V_{\tau_f\tau_f}} \cdot \bm{c}_h' \bm{\zeta}_{LF_h,a_h}\right)
	\end{eqnarray*}
	where $\bm{c}_h' = (\bm{e}_f' \bm{W}_{\bm{\tau\tau}}^\myparallel[h] \bm{e}_f)^{-1/2}
	\bm{e}_f'
	\left(\bm{W}_{\bm{\tau\tau}}^\myparallel[h] \right)^{1/2}_{LF_h}$ is a unit vector with length one. 
	By the definitions of $R^2_f$ and the $\rho_{f}^2[h]$'s, and the spherical symmetry of the $\bm{\zeta}_{LF_h,a_h}$'s from Proposition \ref{prop:represent}, 
	\begin{eqnarray*}
		\hat{\tau}_f - \tau_f \mid  \mathcal{T}_{\text{F}} 
		& \apprsim & 
		\sqrt{V_{\tau_f\tau_f}}
		\left(
		\sqrt{1-R_f^2} \cdot \varepsilon_0 + \sum_{h=1}^H \sqrt{\rho_{f}^2[h]} \cdot \eta_{LF_h,a_h}
		\right).
	\end{eqnarray*}
	Therefore, Corollary \ref{cor:dist_remt} holds. 
\end{proof}

\subsection{Asymptotic sampling distribution of $\hat{\boldsymbol{\tau}}$ under $\text{ReFMT}_\text{CF}$}
\begin{proof}[{   Proof of Theorem \ref{thm:asym_refmt_cf}}]
	For each tier $\mathcal{S}_j$, 
	let $\hat{\bm{\delta}}_{\bm{e}}[j]$ be the concatenation of $\hat{\bm{\theta}}_{\bm{e}[t]}[h]$ with $(t,h)\in \mathcal{S}_j$, and $\hat{\bm{\delta}}_{\bm{e}}' = (\hat{\bm{\delta}}_{\bm{e}}'[1], \ldots, \hat{\bm{\delta}}_{\bm{e}}'[J])$. From Proposition \ref{prop:var_cov_tier_factor_covariate}, under the CRFE, $(\hat{\bm{\tau}}'-\bm{\tau}', \hat{\bm{\delta}}_{\bm{e}}')'$ has sampling mean zero and sampling covariance matrix
	\begin{equation}\label{eq:cov_refmt_cf}
	\begin{pmatrix}
	\bm{V}_{\bm{\tau\tau}} & \bm{U}_{\bm{\tau e}}[1] & \ldots & \bm{U}_{\bm{\tau e}}[J]\\
	\bm{U}_{\bm{e\tau}}[1] & \bm{U}_{\bm{ee}}[1] & \ldots & \bm{0}\\
	\vdots & \vdots & \ddots & \vdots\\
	\bm{U}_{\bm{e\tau}}[J] & \bm{0} & \ldots & \bm{U}_{\bm{ee}}[J]
	\end{pmatrix}. 
	\end{equation}
	Let $(\tilde{\bm{\tau}}', \tilde{\bm{\delta}}_{\bm{e}}')'$ be a Gaussian random vector with mean zero and covariance matrix \eqref{eq:cov_refmt_cf}. 
	From Lemma \ref{lemma:condi_converge}, under Condition \ref{cond:fp}, 
	\begin{eqnarray*}
		\hat{\bm{\tau}}-\bm{\tau} \mid \mathcal{T}_{\text{CF}}
		& \apprsim & 
		\tilde{\bm{\tau}} 
		\ \left| \ 
		\left\{ \tilde{\bm{\delta}}_{\bm{e}}'[j](\bm{U}_{\bm{ee}}[j])^{-1}\tilde{\bm{\delta}}_{\bm{e}}[j] \leq a_j, j=1,\ldots,J
		\right\}.
		\right. 
	\end{eqnarray*}
	Using the same logic as the proof of Theorem \ref{thm:asymp_refmt_f}, we can simplify the asymptotic sampling distribution of $\hat{\bm{\tau}}$ under $\text{ReFMT}_\text{CF}$ as
	\begin{eqnarray*}
		\hat{\bm{\tau}} - \bm{\tau}  \mid \mathcal{T}_{\text{CF}} 
		& \apprsim &
		\left(\bm{V}_{\bm{\tau}\bm{\tau}}^\myperp\right)^{1/2} \bm{\varepsilon} + 
		\sum_{j=1}^{J}
		\bm{U}_{\bm{\tau e}}[j]
		\left(
		\bm{U}_{\bm{ee}}[j]
		\right)^{-1/2}
		\bm{\zeta}_{\lambda_j,a_j}.
	\end{eqnarray*}
	Because by definition  
	$\bm{U}_{\bm{\tau e}}[j]
	(
	\bm{U}_{\bm{ee}}[j]
	)^{-1/2}
	\{\bm{U}_{\bm{\tau e}}[j]
	(
	\bm{U}_{\bm{ee}}[j]
	)^{-1/2}\}'
	= \bm{U}_{\bm{\tau\tau}}^\myparallel[j],
	$
	from Lemma \ref{lemma:truncate_AApm_BBpm}, 
	Theorem \ref{thm:asym_refmt_cf} holds. 
\end{proof}

\begin{proof}[{   Proof of Corollary \ref{cor:single_refmt_cf}}]
	Recall that $\bm{e}_f$ is a unit vector with the $f$th coordinate being one. 
	From Theorem \ref{thm:asym_refmt_cf}, the asymptotic sampling distribution of $\hat{\tau}_f$ under $\text{ReFMT}_\text{CF}$ is 
	\begin{eqnarray*}
		& & \hat{\tau}_f - \tau_f \mid \mathcal{T}_{\text{CF}} \\
		& \apprsim &
		\bm{e}_f'
		\left(
		\bm{V}_{\bm{\tau}\bm{\tau}}^\myperp \right)^{1/2} \bm{\varepsilon} + 
		\sum_{j=1}^J
		\bm{e}_f'
		\left(\bm{U}_{\bm{\tau\tau}}^\myparallel[j]\right)^{1/2}_{\lambda_j}
		\bm{\zeta}_{\lambda_j,a_j}\\
		& \sim & 
		\sqrt{\bm{e}_f'\bm{V}_{\bm{\tau\tau}}^\myperp\bm{e}_f} \cdot \varepsilon_0 + 
		\sum_{j=1}^{J}
		\sqrt{\bm{e}_f'\bm{U}_{\bm{\tau\tau}}^\myparallel[j] \bm{e}_f} \cdot \bm{c}_j' \bm{\zeta}_{\lambda_j,a_j}
		\\
		& \sim & 
		\sqrt{V_{\tau_f\tau_f}} \left(
		\sqrt{
			1 - V_{\tau_f\tau_f}^\myparallel/V_{\tau_f\tau_f}
		}\cdot
		\varepsilon_0 + 
		\sum_{j=1}^J
		\sqrt{U_{\tau_f\tau_f}^\myparallel[j]/V_{\tau_f\tau_f}}
		\cdot \bm{c}_j' \bm{\zeta}_{\lambda_j,a_j}\right), 
	\end{eqnarray*}
	where $\bm{c}_j' = (\bm{e}_f'\bm{U}_{\bm{\tau\tau}}^\myparallel[j] \bm{e}_f)^{-1/2}
	\bm{e}_f'
	\left(\bm{U}_{\bm{\tau\tau}}^\myparallel[j]\right)^{1/2}_{\lambda_j}$ is a unit vector with length one. 
	By the definitions of $R^2_f$ and $\beta_{f}^2[j]$'s, and the spherical symmetry of $\bm{\zeta}_{\lambda_j,a_j}$'s from Proposition \ref{prop:represent}, 
	\begin{eqnarray*}
		\hat{\tau}_f - \tau_f \mid \mathcal{T}_{\text{CF}}
		& \apprsim & 
		\sqrt{V_{\tau_f\tau_f}}
		\left(
		\sqrt{1-R^2_f} \cdot \varepsilon_0 + \sum_{j=1}^J \sqrt{ \beta_{f}^2[j] }
		\cdot \eta_{\lambda_j,a_j}\right).
	\end{eqnarray*}
	Therefore, Corollary \ref{cor:single_refmt_cf} holds. 
\end{proof}

\section{Reduction in asymptotic sampling covariances}\label{app:red_cov}

We use $\Var_{\text{a}}$ and $\Cov_{\text{a}}$ to denote the variance and covariance of the asymptotic distributions of sequences of random variables and random vectors, respectively.

\begin{proof}[{   Proof of Theorem \ref{thm:red_var_rem}}]
	First, we calculate the reduction in the asymptotic sampling covariance of $\hat{\bm{\tau}}$. 
	From Proposition \ref{prop:mean_var_cre},   
	$
	\Cov_{\text{a}}\{\sqrt{n}(\hat{\bm{\tau}} - \bm{\tau})\} = \lim_{n\rightarrow\infty} n\bm{V}_{\bm{\tau\tau}} = \lim_{n\rightarrow\infty}n\bm{V}_{\bm{\tau\tau}}^\myperp + \lim_{n\rightarrow\infty}n\bm{V}_{\bm{\tau\tau}}^\myparallel  .
	$
	For notational simplicity, we omit the limiting signs. 
	From Theorem \ref{thm:joint_refm} and Proposition \ref{prop:represent}, 
	\begin{eqnarray*}
		& & \Cov_{\text{a}}\{\sqrt{n}(\hat{\bm{\tau}} - \bm{\tau}) \mid \mathcal{M}\}\\
		& = & 
		\Cov\left\{\left(n\bm{V}_{\bm{\tau\tau}}^\myperp\right)^{1/2} \bm{\varepsilon}\right\} + n\left(\bm{V}_{\bm{\tau\tau}}^\myparallel\right)^{1/2}_{LF} \Cov\left(\bm{\zeta}_{LF,a}  \right)
		\left\{\left(\bm{V}_{\bm{\tau\tau}}^\myparallel\right)^{1/2}_{LF}\right\}'
		\\
		& = & 
		n\bm{V}_{\bm{\tau\tau}}^\myperp + v_{LF,a} \cdot
		n\left(\bm{V}_{\bm{\tau\tau}}^\myparallel\right)^{1/2}_{LF}
		\left\{\left(\bm{V}_{\bm{\tau\tau}}^\myparallel\right)^{1/2}_{LF}\right\}'\\
		& = & n\bm{V}_{\bm{\tau\tau}}^\myperp + v_{LF,a} \cdot n\bm{V}_{\bm{\tau\tau}}^\myparallel.
	\end{eqnarray*}
	Therefore, the reduction in asymptotic sampling covariance is  $(1-v_{LF,a})n\bm{V}_{\bm{\tau\tau}}^{\myparallel}$. 
	
	Second, we consider the PRIASV of $\hat{\tau}_f$. 
	From Proposition \ref{prop:mean_var_cre} and Corollary \ref{cor:dist_refm_one_linear}, the asymptotic sampling variance of $\hat{\tau}_f$ are 
	\begin{align*}
	\Var_{\text{a}}\left\{
	\sqrt{n}(\hat{\tau}_f-\tau_f)
	\right\} & =  
	nV_{\tau_f\tau_f}, \\
	\Var_{\text{a}}\left\{
	\sqrt{n}(\hat{\tau}_f-\tau_f)
	\mid \mathcal{M}
	\right\} & = 
	nV_{\tau_f\tau_f} (1-R_f^2 + R_f^2 \cdot v_{LF,a}) \\
	& = nV_{\tau_f\tau_f}\left\{
	1 - (1-v_{LF,a})R_f^2
	\right\}. 
	\end{align*}
	Therefore, the PRIASV of $\hat{\tau}_f$ is $(1-v_{LF,a})R^2_f.$ 
\end{proof}

\begin{proof}[{   Proof of Theorem \ref{thm:red_var_remft_f}}]
	First, we calculate the reduction in the asymptotic sampling covariance of $\hat{\bm{\tau}}$. 
	From Theorem \ref{thm:asymp_refmt_f} and Proposition \ref{prop:represent}, 
	\begin{align*}
	& \quad \ \Cov_{\text{a}}\{
	\sqrt{n}(
	\hat{\bm{\tau}} - \bm{\tau}
	) \mid \mathcal{T}_{\text{F}}
	\}\\
	& = n\bm{V}_{\bm{\tau\tau}}^{\myperp} + n
	\sum_{h=1}^H 
	\left(\bm{W}_{\bm{\tau\tau}}^\myparallel[h] \right)^{1/2}_{LF_h}
	\Cov(\bm{\zeta}_{LF_h,a_h})
	\left\{\left(\bm{W}_{\bm{\tau\tau}}^\myparallel[h] \right)^{1/2}_{LF_h}\right\}'\\
	& = n\bm{V}_{\bm{\tau\tau}}^{\myperp} + 
	n\sum_{h=1}^H v_{LF_h, a_h} 
	\bm{W}_{\bm{\tau \tau}}^\myparallel[h]. 
	\end{align*}
	Because $\hat{\bm{\tau}}_{\bm{x}}$ and $\hat{\bm{\theta}}_{\bm{x}}$ are linear transformations of each other,  the sampling covariances of $\hat{\bm{\tau}}$ explained by $\hat{\bm{\tau}}_{\bm{x}}$ and $\hat{\bm{\theta}}_{\bm{x}}$ are the same, i.e., 
	$
	\bm{V}_{\bm{\tau\tau}}^\myparallel = \sum_{h=1}^H  
	\bm{W}_{\bm{\tau \tau}}^\myparallel[h]. 
	$
	Thus,  
	\begin{align*}
	\Cov_{\text{a}}\{
	\sqrt{n}(
	\hat{\bm{\tau}} - \bm{\tau}
	)
	\}
	& = n\bm{V}_{\bm{\tau\tau}}^{\myperp} + 
	n\bm{V}_{\bm{\tau\tau}}^\myparallel = 
	n\bm{V}_{\bm{\tau\tau}}^{\myperp} + 
	n\sum_{h=1}^H  
	\bm{W}_{\bm{\tau \tau}}^\myparallel[h].
	\end{align*}
	Therefore, the reduction in asymptotic sampling covariance of $\hat{\bm{\tau}}$ is 
	\begin{eqnarray*}
		\Cov_{\text{a}}\{
		\sqrt{n}(
		\hat{\bm{\tau}} - \bm{\tau}
		)
		\} -
		\Cov_{\text{a}}\{
		\sqrt{n}(
		\hat{\bm{\tau}} - \bm{\tau}
		) \mid \mathcal{T}_{\text{F}}
		\} & = n \sum_{h=1}^{H}(1-v_{LF_h, a_h})\bm{W}_{\bm{\tau \tau}}^\myparallel[h]. 
	\end{eqnarray*}

	Second, we consider the PRIASV of $\hat{\tau}_f$. 
	From Proposition \ref{prop:mean_var_cre} 
	and Corollary \ref{cor:dist_remt}, the asymptotic sampling variance of $\hat{\tau}_f$ under the CRFE and $\text{ReFMT}_\text{F}$ are
	\begin{align*}
	\Var_{\text{a}}\left\{
	\sqrt{n}(\hat{\tau}_f-\tau_f)
	\right\} & = nV_{\tau_f\tau_f}, \\
	\Var_{\text{a}}\left\{
	\sqrt{n}(\hat{\tau}_f-\tau_f)
	\mid \mathcal{T}_{\text{F}}
	\right\} & = 
	nV_{\tau_f\tau_f}\left(
	1 - R_f^2 + \sum_{h=1}^{H}\rho^2_f[h] v_{LF_h, a_h}
	\right). 
	\end{align*}
	By definition of $\rho^2_f[h]$ and the fact that 
	$
	\bm{V}_{\bm{\tau\tau}}^\myparallel = \sum_{h=1}^H  
	\bm{W}_{\bm{\tau \tau}}^\myparallel[h], 
	$
	we have 
	$
	R_f^2 = \sum_{h=1}^H \rho_f^2[h]. 
	$
	Therefore, the PRIASV of $\hat{\tau}_f$ under $\text{ReFMT}_\text{F}$ is 
	\begin{align*}
	&~\frac{
		nV_{\tau_f\tau_f} - nV_{\tau_f\tau_f}\left(
		1 - R_f^2 + \sum_{h=1}^{H}\rho^2_f[h] v_{LF_h, a_h}
		\right)
	}{nV_{\tau_f\tau_f}} \\
	=&~  
	R_f^2 - \sum_{h=1}^{H}\rho^2_f[h] v_{LF_h, a_h} 
	= \sum_{h=1}^H (1-v_{LF_h, a_h})\rho_{f}^2[h]. 
	\end{align*}
\end{proof}

\begin{proof}[{   Proof of Theorem \ref{thm:red_var_refmt_cf}}]
	Note that the $
	\{\hat{\bm{\theta}}_{\bm{e}[t]}[h]\}_{(t,h)\in \mathcal{S}_j}$'s are from a block-wise Gram--Schmidt orthogonalization of $\hat{\bm{\tau}}_{\bm{x}}$. The proof is similar to that of Theorem \ref{thm:red_var_remft_f}. We omit it. 
\end{proof}

\section{Peakedness of the asymptotic sampling distributions}\label{app:peak}

\subsection{Lemmas and propositions for peakedness}

Recall that we say a random vector $\bm{\phi} \in \mathbb{R}^m$ is more peaked than another random vector $\bm{\psi} \in \mathbb{R}^m$ and write as $\bm{\phi} \succ \bm{\psi}$, if $P(\bm{\phi}\in \mathcal{K})\geq P(\bm{\psi}\in \mathcal{K})$ for every symmetric convex set $\mathcal{K}\subset \mathbb{R}^m$.

\begin{lemma}\label{lemma:peak_linear}
	If two $m$ dimensional symmetric random vectors $\bm{\phi}_1$ and $\bm{\phi}_2$ satisfy $\bm{\phi}_1 \succ \bm{\phi}_2$, then for any matrix $\bm{C}\in \mathbb{R}^{p\times m}$,  $\bm{C}\bm{\phi}_1 \succ \bm{C}\bm{\phi}_2$. 
\end{lemma}

\begin{lemma}\label{lemma:peak_sum}
	Let $\bm{\psi}, \bm{\phi}_1$ and $\bm{\phi}_2$ be three independent $m$ dimensional symmetric random vectors. If $\bm{\psi}$ is central symmetric unimodal and $\bm{\phi}_1 \succ \bm{\phi}_2$, then $\bm{\psi}+\bm{\phi}_1 \succ \bm{\psi}+\bm{\phi}_2$. 
\end{lemma}

\begin{proof}
	[Proof of Lemmas \ref{lemma:peak_linear} and \ref{lemma:peak_sum}]
	See \citet[][Lemma 7.2 and Theorem 7.5]{dharmadhikari1988}.
\end{proof}

The following lemma states that truncating a standard Gaussian random vector within a ball makes it more peaked.
Although the result seems intuitive, the proof below is a little tedious due to some technical reasons.

\begin{lemma}\label{lemma:zeta_peak_epsilon}
	Let $\bm{\varepsilon} \sim \mathcal{N}(\bm{0}, \bm{I}_m)$ and $\bm{\zeta}_{m,a} \sim \bm{\varepsilon} \mid \bm{\varepsilon}'\bm{\varepsilon}\leq a$. Then $\bm{\zeta}_{m,a} \succ \bm{\varepsilon}$. 
\end{lemma}

\begin{proof}[Proof of Lemma \ref{lemma:zeta_peak_epsilon}]
	For any symmetric convex set $\mathcal{K}$, let $\|\bm{\varepsilon}\|_2=(\bm{\varepsilon}'\bm{\varepsilon})^{1/2}$ be the $l_2$-norm of $\bm{\varepsilon}$, and 
	$
	G(r) = P
	(
	\bm{\varepsilon} \in \mathcal{K} \mid \|\bm{\varepsilon}\|_2 = r
	)
	$ be the conditional probability that $\bm{\varepsilon}$ is in $\mathcal{K}$ given the length of $\bm{\varepsilon}$. Let $\bm{\phi}$ be a random vector uniformly distributed on $m$ dimensional unit sphere. By the spherical symmetry of $\bm{\varepsilon},$ we can simplify $G(r)$ as 
	\begin{align*}
	G(r) & = P
	\left(
	\|\bm{\varepsilon}\|_2 \cdot 
	\frac{\bm{\varepsilon}}{\|\bm{\varepsilon}\|_2} 
	\in \mathcal{K} 
	\ \bigl\vert \ 
	\|\bm{\varepsilon}\|_2 = r
	\right) = 
	P
	(
	r\bm{\phi} \in \mathcal{K}
	)
	= 
	P
	(
	\bm{\phi} \in r^{-1}\mathcal{K}
	),
	\end{align*}
	where $r^{-1}\mathcal{K} = \{r^{-1}\bm{x}: \bm{x} \in \mathcal{K}\}$. 
	For any $r_1\geq r_2\geq 0$, if $\tilde{\bm{x}} \in r_1^{-1}\mathcal{K}$, then $r_1\tilde{\bm{x}} \in \mathcal{K}$. By the symmetric convexity of $\mathcal{K}$, we then have $r_2\tilde{\bm{x}} = r_2/r_1 \cdot (r_1\tilde{\bm{x}}) \in \mathcal{K}$, i.e., $\tilde{\bm{x}} \in r_2^{-1}\mathcal{K}$. Thus, $r_1^{-1}\mathcal{K} \subset r_2^{-1}\mathcal{K}$, which further implies $G(r_1) \leq G(r_2)$. Therefore, $G(r)$ is a nonincreasing function of $r\in [0, \infty)$. 
	We can represent the probabilities that $\bm{\zeta}_{m,a}$ and $\bm{\varepsilon}$ belong to $\mathcal{K}$, respectively, as follows:
	\begin{align*}
	P(\bm{\zeta}_{m,a} \in \mathcal{K}) & = 
	P(\bm{\varepsilon} \in \mathcal{K} \mid \bm{\varepsilon}'\bm{\varepsilon}\leq a) = 
	\E\left\{P(\bm{\varepsilon} \in \mathcal{K} \mid \bm{\varepsilon}'\bm{\varepsilon}\leq a, \|\bm{\varepsilon}\|_2)
	\mid \bm{\varepsilon}'\bm{\varepsilon}\leq a
	\right\} \\
	& = \E\left\{
	G(\|\bm{\varepsilon}\|_2) \mid \bm{\varepsilon}'\bm{\varepsilon}\leq a
	\right\}
	= \E\left\{
	G(\chi_m) \mid \chi^2_m \leq a
	\right\},
	\end{align*}
	and 
	$
	P(\bm{\varepsilon} \in \mathcal{K})  = \E\{
	G(\chi_m)
	\}.
	$
	By the monotone nonincreasing property of $G(r)$, we have
	\begin{align*}
	P(\bm{\varepsilon} \in \mathcal{K}) & = \E\left\{
	G(\chi_m)
	\right\} \\
	& = P(\chi^2_m \leq a)\E\left\{
	G(\chi_m) \mid \chi^2_m \leq a
	\right\} + 
	P(\chi^2_m > a)\E\left\{
	G(\chi_m) \mid \chi^2_m > a
	\right\}\\
	& \leq   
	P(\chi^2_m \leq a)\E\left\{
	G(\chi_m) \mid \chi^2_m \leq a
	\right\} + 
	P(\chi^2_m > a)
	G(\sqrt{a})  \\
	& \leq 
	P(\chi^2_m \leq a)\E\left\{
	G(\chi_m) \mid \chi^2_m \leq a
	\right\} + 
	P(\chi^2_m > a)
	\E\left\{
	G(\chi_m) \mid \chi^2_m \leq a
	\right\}\\
	& = \E\left\{
	G(\chi_m) \mid \chi^2_m \leq a
	\right\} = P(\bm{\zeta}_{m,a} \in \mathcal{K}). 
	\end{align*}
	Therefore, Lemma \ref{lemma:zeta_peak_epsilon} holds. 
\end{proof}

\begin{proof}[{   Proof of Proposition \ref{prop:peak_cov}}]
	From Lemma \ref{lemma:peak_linear}, for any vector $\bm{c}\in \mathbb{R}^m$, $\bm{c}'\bm{\phi}\succ \bm{c}'\bm{\psi}$. Thus, 
	\begin{align*}
	\Var(\bm{c}' \bm{\phi}) & = \E\left\{\left(\bm{c}' \bm{\phi}\right)^2 \right\} = \int_{0}^{\infty} \left[1 - P\left\{
	\left(\bm{c}' \bm{\phi}\right)^2 \leq t
	\right\} \right] \text{d}t \\
	& = 
	\int_{0}^{\infty} \left\{ 1 - P\left(
	\bm{c}' \bm{\phi}\in [-t, t]
	\right) \right\} \text{d}t\\
	& \leq \int_{0}^{\infty} \left\{ 1 - P\left(
	\bm{c}' \bm{\psi}\in [-t, t]
	\right) \right\} = \Var(\bm{c}' \bm{\psi}).
	\end{align*}
	Because the above inequality holds for any $\bm{c}$, $\Cov(\bm{\phi})\leq \Cov(\bm{\psi})$. 
\end{proof}

\subsection{Propositions for simultaneous diagonalization}

\begin{proof}[{   Proof of Proposition \ref{prop:add_simu_diag}}]
	Recall that $\bm{\Psi}$ is the matrix such that $\coef_q = \bm{\Psi} \orthocoef_q$. 
	By the construction of the $\bm{c}_q$'s, 
	\begin{align*}
	\tilde{\bm{C}}
	& =
	2^{-2(K-1)}
	\sum_{q=1}^{Q}
	n_q^{-1}(\orthocoef_q\orthocoef'_q) 
	\equiv
	\text{diag}\left\{ \tilde{\bm{C}}[1,1], \ldots, \tilde{\bm{C}}[H,H] \right\}
	\end{align*}
	is a block diagonal matrix, and $\tilde{\bm{B}} = \bm{\Psi} \tilde{\bm{C}} \bm{\Psi}'$. Partition $\tilde{\bm{C}}$ into $\tilde{\bm{C}} = (\tilde{\bm{C}}[,1], \ldots, \tilde{\bm{C}}[,H])$.
	From Proposition \ref{prop:var_cov_tier_factor}, under the additivity, we have 
	\begin{align*}
	\bm{W}_{\bm{\tau x}}[h] & = 2^{-2(K-1)}\sum_{q=1}^{Q}
	n_q^{-1}(\coef_q\orthocoef'_q[h]) \otimes \bm{S}_{1, \bm{x}} = 
	2^{-2(K-1)}\sum_{q=1}^{Q}
	n_q^{-1}(\bm{\Psi}\orthocoef_q\orthocoef'_q[h]) \otimes \bm{S}_{1, \bm{x}}\\
	& = 
	\left(\bm{\Psi}
	\tilde{\bm{C}}[,h]\right)
	\otimes \bm{S}_{1, \bm{x}}, 
	\\
	\bm{W}_{\bm{xx}}[h] & = 
	2^{-2(K-1)}\sum_{q=1}^{Q}
	n_q^{-1}(\orthocoef_q[h]\orthocoef'_q[h]) \otimes \bm{S}_{\bm{xx}} = 
	\tilde{\bm{C}}[h,h] \otimes \bm{S}_{\bm{xx}}.
	\end{align*}
	Thus, under the additivity, for each $1\leq h\leq H$,  $\bm{W}_{\bm{\tau\tau}}^\myparallel[h]$ reduces to 
	\begin{align*}
	\bm{W}_{\bm{\tau\tau}}^\myparallel[h] & = \bm{W}_{\bm{\tau x}}[h] \bm{W}_{\bm{xx}}^{-1}[h] \bm{W}_{\bm{x\tau}}[h] \\
	& = 
	\left\{
	\bm{\Psi}
	\tilde{\bm{C}}[,h]\left(\tilde{\bm{C}}[h,h]\right)^{-1}(\tilde{\bm{C}}[,h])'\bm{\Psi}'
	\right\}
	\otimes (\bm{S}_{1,\bm{x}}\bm{S}_{\bm{xx}}^{-1}\bm{S}_{\bm{x},1})\\
	& =S_{11}^\myparallel  \cdot \bm{\Psi} \cdot
	\textup{diag}(\bm{0}, \ldots, \bm{0}, \tilde{\bm{C}}[h,h], \bm{0}, \ldots, \bm{0})
	\cdot
	\bm{\Psi}'\\  
	& =S_{11}^\myparallel  \cdot \bm{\Psi} \cdot \tilde{\bm{C}}^{1/2} \cdot 
	\textup{diag}(\bm{0}, \ldots, \bm{0}, \bm{I}_{F_h}, \bm{0}, \ldots, \bm{0})
	\cdot 
	\tilde{\bm{C}}^{1/2} \cdot
	\bm{\Psi}'
	\end{align*}
	Under the additivity, 
	$
	\bm{V}_{\bm{\tau\tau}} = 2^{-2(K-1)} \sum_{q=1}^{Q} n_q^{-1} \coef_q\coef_q' \cdot S_{11}
	= S_{11}  \cdot \tilde{\bm{B}}.
	$ 
	Thus, 
	\begin{align*}
	& \quad \ \bm{V}_{\bm{\tau}\bm{\tau}}^{-1/2} \bm{W}_{\bm{\tau\tau}}^\myparallel[h] \bm{V}_{\bm{\tau}\bm{\tau}}^{-1/2} \\
	& = 
	S_{11}^{-1/2} \tilde{\bm{B}}^{-1/2} 
	\left\{
	S_{11}^\myparallel \cdot \bm{\Psi} \cdot \tilde{\bm{C}}^{1/2} \cdot 
	\textup{diag}(\bm{0}, \ldots, \bm{0}, \bm{I}_{F_h}, \bm{0}, \ldots, \bm{0})
	\cdot 
	\tilde{\bm{C}}^{1/2} \cdot
	\bm{\Psi}'
	\right\}
	\cdot S_{11}^{-1/2}  \tilde{\bm{B}}^{-1/2} 
	\\
	& = S_{11}^\myparallel/S_{11} \cdot \tilde{\bm{B}}^{-1/2} \bm{\Psi} \tilde{\bm{C}}^{1/2} \cdot
	\textup{diag}(\bm{0}, \ldots, \bm{0}, \bm{I}_{F_h}, \bm{0}, \ldots, \bm{0})
	\cdot \tilde{\bm{C}}^{1/2}
	\bm{\Psi}'  \tilde{\bm{B}}^{-1/2} \\
	& =  \bm{\Gamma} \cdot
	\textup{diag}(\bm{0}, \ldots, \bm{0}, S_{11}^\myparallel/S_{11} \cdot \bm{I}_{F_h}, \bm{0}, \ldots, \bm{0})
	\cdot \bm{\Gamma}', 
	\end{align*}
	where $\bm{\Gamma} = \tilde{\bm{B}}^{-1/2} \bm{\Psi} \tilde{\bm{C}}^{1/2}$.  Because 
	$$
	\bm{\Gamma}\bm{\Gamma}'  = \tilde{\bm{B}}^{-1/2} \bm{\Psi} \tilde{\bm{C}}^{1/2}\tilde{\bm{C}}^{1/2}
	\bm{\Psi}'  \tilde{\bm{B}}^{-1/2} = \tilde{\bm{B}}^{-1/2}
	\left(
	\bm{\Psi} \tilde{\bm{C}}
	\bm{\Psi}'
	\right)
	\tilde{\bm{B}}^{-1/2} = \tilde{\bm{B}}^{-1/2}
	\tilde{\bm{B}}
	\tilde{\bm{B}}^{-1/2} 
	= \bm{I}_F,
	$$
	we have that 
	$\bm{\Gamma}$ is an orthogonal matrix. 
	Note that 
	$\textup{diag}(\bm{0}, \ldots, \bm{0}, S_{11}^\myparallel/S_{11}\cdot\bm{I}_{F_h}, \bm{0}, \ldots, \bm{0})$ is a diagonal matrix with exactly $F_h$ nonzero
	elements, which are all equal to $S_{11}^\myparallel/S_{11}$. 
	Therefore, Condition \ref{cond:simultaneous_ortho} and Proposition \ref{prop:add_simu_diag} hold. 
\end{proof}

\begin{proof}[{   Proof of Proposition \ref{prop:add_simu_diag_refmt_cf}}]
	Recall $\bm{\Psi}$ is the matrix such that $\coef_q = \bm{\Psi} \orthocoef_q$.
	From the proof of Proposition \ref{prop:add_simu_diag}, 
	$
	\tilde{\bm{C}} = 2^{-2(K-1)}\sum_{q=1}^{Q}n_q^{-1} (\bm{c}_q\bm{c}_q')
	$
	is a block diagonal matrix, and $\tilde{\bm{B}} = \bm{\Psi} \tilde{\bm{C}} \bm{\Psi}'$.  
	From Proposition \ref{prop:var_cov_tier_factor_covariate}, under the additivity, 
	$\bm{W}_{\bm{\tau e}[t]}[h]$ and $\bm{W}_{\bm{e}[t]\bm{e}[t]}[h]$ reduce to 
	\begin{align*} 
	\bm{W}_{\bm{\tau e}[t]}[h] & = 2^{-2(K-1)}
	\left\{\bm{\Psi}
	\sum_{q=1}^{Q}
	n_q^{-1}(\orthocoef_q\orthocoef'_q[h]) \right\} \otimes \bm{S}_{1, \bm{e}[t]} = 
	\left( \bm{\Psi} \tilde{\bm{C}}[,h] \right)
	\otimes \bm{S}_{1, \bm{e}[t]}
	\\
	\bm{W}_{\bm{e}[t]\bm{e}[t]}[h] & = 
	2^{-2(K-1)}\sum_{q=1}^{Q}
	n_q^{-1}(\orthocoef_q[h]\orthocoef'_q[h]) \otimes \bm{S}_{\bm{e}[t]\bm{e}[t]}
	=  
	\tilde{\bm{C}}[h,h] \otimes \bm{S}_{\bm{e}[t]\bm{e}[t]},
	\end{align*}
	which then implies
	\begin{align*}
	\bm{W}_{\bm{e}[t]\bm{e}[t]}^\myparallel[h] & \equiv 
	\bm{W}_{\bm{\tau e}[t]}[h]\bm{W}_{\bm{e}[t]\bm{e}[t]}^{-1}[h]
	\bm{W}_{\bm{\tau e}[t]}'[h]\\
	& = \left\{
	\bm{\Psi} \tilde{\bm{C}}[,h]
	\left(\tilde{\bm{C}}[h,h]\right)^{-1}
	\left(\tilde{\bm{C}}[,h]\right)'\bm{\Psi}'
	\right\}
	\otimes 
	\left(
	\bm{S}_{1, \bm{e}[t]} \bm{S}_{\bm{e}[t]\bm{e}[t]}^{-1}
	\bm{S}_{\bm{e}[t],1}
	\right)\\
	& =  \left(
	\bm{S}_{1, \bm{e}[t]} \bm{S}_{\bm{e}[t]\bm{e}[t]}^{-1}
	\bm{S}_{\bm{e}[t],1}
	\right) \cdot
	\bm{\Psi}  \cdot
	\textup{diag}(\bm{0}, \ldots, \bm{0}, \tilde{\bm{C}}[h,h], \bm{0}, \ldots, \bm{0})
	\cdot
	\bm{\Psi}'\\
	& =  S_{11}^\myparallel[t] \cdot
	\bm{\Psi}  \cdot
	\textup{diag}(\bm{0}, \ldots, \bm{0}, \tilde{\bm{C}}[h,h], \bm{0}, \ldots, \bm{0})
	\cdot
	\bm{\Psi}'\\
	& = 
	S_{11}^\myparallel[t] \cdot
	\bm{\Psi}  \cdot \tilde{\bm{C}}^{1/2} \cdot 
	\textup{diag}(\bm{0}, \ldots, \bm{0}, \bm{I}_{F_h}, \bm{0}, \ldots, \bm{0})
	\cdot
	\tilde{\bm{C}}^{1/2}
	\cdot
	\bm{\Psi}', 
	\end{align*} 
	where 
	$S_{11}^\myparallel[t] = \bm{S}_{1, \bm{e}[t]} \bm{S}_{\bm{e}[t]\bm{e}[t]}^{-1}
	\bm{S}_{\bm{e}[t],1}$ 
	is the finite population variance of the linear projection of potential outcome $Y(1)$ on orthogonalized covariates in tier $t$. 
	Because $\bm{V}_{\bm{\tau}\bm{\tau}}$ reduces to 
	$
	\bm{V}_{\bm{\tau}\bm{\tau}} = S_{11} \tilde{\bm{B}}
	$
	under the additivity, 
	we have 
	\begin{align*}
	& \quad \ \bm{V}_{\bm{\tau}\bm{\tau}}^{-1/2}\bm{W}_{\bm{e}[t]\bm{e}[t]}^\myparallel[h]
	\bm{V}_{\bm{\tau}\bm{\tau}}^{-1/2} \\
	& = S_{11}^\myparallel[t]/S_{11}  \cdot \tilde{\bm{B}}^{-1/2}
	\left\{
	\bm{\Psi}  \tilde{\bm{C}}^{1/2} \cdot 
	\textup{diag}(\bm{0}, \ldots, \bm{0}, \bm{I}_{F_h}, \bm{0}, \ldots, \bm{0})
	\cdot
	\tilde{\bm{C}}^{1/2}
	\bm{\Psi}'
	\right\} \tilde{\bm{B}}^{-1/2}
	\\& = S_{11}^\myparallel[t]/S_{11} \cdot 
	\tilde{\bm{B}}^{-1/2}\bm{\Psi}  \tilde{\bm{C}}^{1/2}
	\cdot 
	\textup{diag}(\bm{0}, \ldots, \bm{0}, \bm{I}_{F_h}, \bm{0}, \ldots, \bm{0})
	\cdot
	\tilde{\bm{C}}^{1/2}
	\bm{\Psi}'\tilde{\bm{B}}^{-1/2}
	\\
	& = 
	\bm{\Gamma}
	\bm{\Omega}_{[t]}[h]
	\bm{\Gamma}',
	\end{align*}
	where $\bm{\Omega}_{[t]}[h] = \textup{diag}(\bm{0}, \ldots, \bm{0}, S_{11}^\myparallel[t]/S_{11} \cdot \bm{I}_{F_h}, \bm{0}, \ldots, \bm{0})$ is a diagonal matrix, 
	and $\bm{\Gamma} = \tilde{\bm{B}}^{-1/2}\bm{\Psi}  \tilde{\bm{C}}^{1/2}$ is an orthogonal matrix. 
	By definition, we then have 
	\begin{align*}
	& \quad \ \bm{V}_{\bm{\tau\tau}}^{-1/2}
	\bm{U}_{\bm{\tau\tau}}^\myparallel[j]
	\bm{V}_{\bm{\tau\tau}}^{-1/2}\\
	& = 
	\bm{V}_{\bm{\tau\tau}}^{-1/2}
	\left(
	\sum_{(t,h)\in \mathcal{S}_j}\bm{W}_{\bm{e}[t]\bm{e}[t]}^\myparallel[h]
	\right)
	\bm{V}_{\bm{\tau\tau}}^{-1/2}
	= \sum_{(t,h)\in \mathcal{S}_j}
	\bm{V}_{\bm{\tau\tau}}^{-1/2}
	\bm{W}_{\bm{e}[t]\bm{e}[t]}^\myparallel[h]
	\bm{V}_{\bm{\tau\tau}}^{-1/2}\\
	& = \sum_{(t,h)\in \mathcal{S}_j}
	\bm{\Gamma} \bm{\Omega}_{[t]}[h] \bm{\Gamma}'
	= 
	\bm{\Gamma}
	\left( \sum_{(t,h)\in \mathcal{S}_j}\bm{\Omega}_{[t]}[h] \right) \bm{\Gamma}',
	\quad (1\leq j\leq J)
	\end{align*}
	where $\sum_{(t,h)\in \mathcal{S}_j}\bm{\Omega}_{[t]}[h]$ is a diagonal matrix. 
	Therefore, Condition \ref{cond:simultaneous_ortho_refmt_cf} 
	and Proposition \ref{prop:add_simu_diag_refmt_cf} 
	hold. 
\end{proof}

\subsection{Peakedness under ReFM}

\begin{proof}[{   Proof of Theorem \ref{thm:red_qr_rem}}]
	Let $\bm{D}\sim \mathcal{N}(\bm{0}, \bm{I}_{LF})$. From Lemma \ref{lemma:zeta_peak_epsilon}, $\bm{\zeta}_{LF,a} \succ \bm{D}$. 
	From Lemma \ref{lemma:peak_linear}, 
	$(\bm{V}_{\bm{\tau \tau}}^\myparallel)^{1/2}_{LF}\bm{\zeta}_{LF,a} \succ (\bm{V}_{\bm{\tau \tau}}^\myparallel)^{1/2}_{LF}\bm{D}$. 
	From Proposition \ref{prop:ccu_refm}, $\bm{\varepsilon}$ is central convex unimodal. From Lemma \ref{lemma:ccu_linear}, $(\bm{V}_{\bm{\tau\tau}}^\myperp)^{1/2}\bm{\varepsilon}$ is also central convex unimodal. 
	Then, from Lemma \ref{lemma:peak_sum}, 
	\begin{align*}
	(\bm{V}_{\bm{\tau\tau}}^\myperp)^{1/2}\bm{\varepsilon} + (\bm{V}_{\bm{\tau \tau}}^\myparallel)^{1/2}_{LF}\bm{\zeta}_{LF,a} 
	\ \succ \ 
	(\bm{V}_{\bm{\tau\tau}}^\myperp)^{1/2}\bm{\varepsilon} +  (\bm{V}_{\bm{\tau \tau}}^\myparallel)^{1/2}_{LF}\bm{D} 
	\ \sim \ 
	\mathcal{N}(\bm{0},\bm{V}_{\bm{\tau\tau}}). 
	\end{align*}
	Therefore, Theorem \ref{thm:red_qr_rem} holds. 
\end{proof}

To prove Theorem \ref{thm:red_qr_rem_nond_ccorr}, we need the following three lemmas. 

\begin{lemma}\label{lemma:order_rho_mono}
	Let $\varepsilon, \eta\sim \mathcal{N}(0,1)$ be two independent standard Gaussian random variables. For any 
	$1 \geq \rho \geq \tilde{\rho} \geq 0$, $a\geq 0$ and $c\geq 0$, 
	\begin{align*}
	P\left(
	|\sqrt{1-\rho^2} \cdot \varepsilon_0 + \rho \cdot \eta| \leq c 
	\ \bigl\vert \ 
	\eta^2 \leq a
	\right)
	\geq 
	P\left( 
	|\sqrt{1-\tilde{\rho}^2} \cdot \varepsilon_0 + \tilde{\rho} \cdot \eta| \leq c
	\ \bigl\vert \ 
	\eta^2 \leq a
	\right). 
	\end{align*}
\end{lemma}
\begin{proof}[Proof of Lemma \ref{lemma:order_rho_mono}]
	It follows directly from \citet[][Lemma A3]{asymrerand2016}, and is also a special case of \citet[][Theorem 2.1]{dasgupta1972}.
\end{proof}

The following lemma extends the above lemma to the multivariate case.

\begin{lemma}\label{lemma:nondecreasing_eigenvalue}
	Let $\bm{\varepsilon}$ and $\bm{\eta}$ be two independent $m$ dimensional standard Gaussian random vectors, $(\rho_1, \ldots, \rho_m)$ be $m$ constants in $[0,1]$, and $\bm{\Delta}$ be a diagonal matrix with diagonal elements $(\rho_1, \ldots, \rho_m)$. 
	For any $r\geq 0$, the probability 
	$$
	P\left\{
	(\bm{I}_m - \bm{\Delta}^2)^{1/2}
	\bm{\varepsilon} + 
	\bm{\Delta}\bm{\eta} \in \mathcal{B}_m(r)
	\mid \bm{\eta}'\bm{\eta} \leq a
	\right\}
	$$
	is nondecreasing in $(\rho_1, \ldots, \rho_m)$. Specifically, for any constants $(\rho_1, \ldots, \rho_m)$ and $(\tilde{\rho}_1, \ldots, \tilde{\rho}_m)$ in $[0,1]$, if 
	$\rho_j \geq \tilde{\rho}_j$ for  $1\leq j\leq m$, then for any $r\geq 0$, 
	\begin{eqnarray}\label{eq:nondecreasing_eigenvalue}
	& & P\left\{ 
	(\bm{I}_F - \bm{\Delta}^2)^{1/2}
	\bm{\varepsilon} + 
	\bm{\Delta}\bm{\eta}
	\in \mathcal{B}_m(r)
	\mid \bm{\eta}'\bm{\eta} \leq a
	\right\} \nonumber\\
	& \geq & 
	P\left\{ 
	(\bm{I}_F - \tilde{\bm{\Delta}}^2)^{1/2}
	\bm{\varepsilon} + 
	\tilde{\bm{\Delta}}\bm{\eta}
	\in \mathcal{B}_m(r)
	\mid \bm{\eta}'\bm{\eta} \leq a
	\right\}.
	\end{eqnarray}
	where $\bm{\Delta} = \textup{diag}(\rho_1, \ldots, \rho_m)$ and $\tilde{\bm{\Delta}} = \textup{diag}(\tilde{\rho}_1, \ldots, \tilde{\rho}_m)$. 
\end{lemma}

\begin{proof}[{   Proof of Lemma \ref{lemma:nondecreasing_eigenvalue}}]
	To prove Lemma \ref{lemma:nondecreasing_eigenvalue}, it suffices to prove that for any constants $(\rho_1, \ldots, \rho_m)$ and $(\tilde{\rho}_1, \ldots, \tilde{\rho}_m)$ in $[0,1]$, if there exists $1\leq k\leq m$ such that $\rho_k \geq \tilde{\rho}_k$ and $\rho_j = \tilde{\rho}_j$ for $j\neq k$, then \eqref{eq:nondecreasing_eigenvalue} holds for any $r\geq 0$. 
	By symmetry, we consider only the case with $k=1$. 
	Let $\bm{\eta}\sim \mathcal{N}(\bm{0}, \bm{I}_m)$ independent of $\bm{\varepsilon}$, and 
	$\bm{\varepsilon}_{-1}=(\varepsilon_2, \ldots, \varepsilon_m)$ and 
	$\bm{\eta}_{-1}=(\eta_2, \ldots, \eta_m)$
	be the subvectors of $\bm{\varepsilon}$ and $\bm{\eta}$ excluding the first coordinates. 
	Define 
	\begin{align*}
	\mathcal{B}(r,\bm{\varepsilon}_{-1},\bm{\eta}_{-1},\rho_2, \ldots, \rho_m) & = 
	\left\{
	x: 
	\begin{pmatrix}
	x\\
	\sqrt{1-\rho_2^2} \cdot \varepsilon_2 + \rho_2 \eta_2\\
	\vdots\\
	\sqrt{1-\rho_m^2} \cdot \varepsilon_m + \rho_m \eta_m
	\end{pmatrix}\in \mathcal{B}_m(r)
	\right\} \subset \mathbb{R},
	\end{align*}
	which is either an empty set or a symmetric closed interval on the real line. 
	For any $r\geq 0$, $\bm{\varepsilon}_{-1}$ and $\bm{\eta}_{-1}$, 
	\begin{align}\label{eq:conditional_nondecreasing_eigenvalue}
	& \quad \  P\left\{ 
	(\bm{I}_F - \bm{\Delta}^2)^{1/2}
	\bm{\varepsilon} + 
	\bm{\Delta}\bm{\eta}
	\in \mathcal{B}_m(r)
	\mid \bm{\eta}'\bm{\eta} \leq a, \bm{\varepsilon}_{-1}, \bm{\eta}_{-1}
	\right\} 
	\nonumber
	\\
	& = 
	P\left\{ 
	\sqrt{1-\rho_1^2} \cdot \varepsilon_1 + \rho_1 \cdot \eta_1
	\in \mathcal{B}(r,\bm{\varepsilon}_{-1},\bm{\eta}_{-1},\rho_2, \ldots, \rho_m)
	\mid \eta_1^2 \leq a - \bm{\eta}_{-1}'\bm{\eta}_{-1}, \bm{\varepsilon}_{-1}, \bm{\eta}_{-1}
	\right\}
	\nonumber
	\\
	& \geq 
	P\left\{ 
	\sqrt{1-\tilde{\rho}_1^2} \cdot \varepsilon_1 + \tilde{\rho}_1 \cdot \eta_1
	\in \mathcal{B}(r,\bm{\varepsilon}_{-1},\bm{\eta}_{-1},\tilde{\rho}_2, \ldots, \tilde{\rho}_m)
	\mid \eta_1^2 \leq a - \bm{\eta}_{-1}'\bm{\eta}_{-1}, \bm{\varepsilon}_{-1}, \bm{\eta}_{-1}
	\right\}
	\nonumber
	\\
	& = 
	P\left\{ 
	(\bm{I}_F - \tilde{\bm{\Delta}}^2)^{1/2}
	\bm{\varepsilon} + 
	\tilde{\bm{\Delta}}\bm{\eta}
	\in \mathcal{B}_m(r)
	\mid \bm{\eta}'\bm{\eta} \leq a, \bm{\varepsilon}_{-1}, \bm{\eta}_{-1}
	\right\}, 
	\end{align}
	where the second last inequality follows from Lemma \ref{lemma:order_rho_mono}. 
	Taking expectations of both sides of \eqref{eq:conditional_nondecreasing_eigenvalue}, we 
	obtain \eqref{eq:nondecreasing_eigenvalue}. 
\end{proof}

\begin{lemma}\label{lemma:general_form_refm}
	Let $\bm{B}_0$ and $\bm{B}_1$ be two $m\times m$ positive semi-definite matrix,  
	$\bm{\varepsilon}\sim \mathcal{N}(\bm{0}, \bm{I}_m)$ be a standard Gaussian random vector, and $\bm{\zeta}_{p,a}\sim \bm{D}\mid \bm{D}'\bm{D}\leq a$ be a truncated Gaussian random vector, where $p\geq m$ and $a\geq 0$. Define 
	$
	\bm{\phi} \sim \left( \bm{B}_0 \right)^{1/2} \bm{\varepsilon} + 
	\left( \bm{B}_1 \right)^{1/2}_p\bm{\zeta}_{p,a}. 
	$
	If $\bm{B}=\bm{B}_0+\bm{B}_1$ is positive definite, then for any $\alpha\in (0,1),$ the threshold $c_{1-\alpha}$ for $1-\alpha$ quantile region $\{\bm{\mu}: \bm{\mu}'\bm{B}^{-1}\bm{\mu}\leq c_{1-\alpha} \}$ of $\bm{\phi}$ depends only on $(m,p,a)$ and the eigenvalues of $\bm{B}^{-1/2}\bm{B}_1\bm{B}^{-1/2}$, and is nonincreasing in these eigenvalues. 
\end{lemma}

\begin{proof}[Proof of Lemma \ref{lemma:general_form_refm}]
	Let $\bm{B}^{-1/2}\bm{B}_1\bm{B}^{-1/2}=\bm{\Gamma}\bm{\Delta}^2\bm{\Gamma}'$ be the eigen-decomposition of $\bm{B}^{-1/2}\bm{B}_1\bm{B}^{-1/2}$, where 
	$\bm{\Gamma}\in \mathbb{R}^{m\times m}$ is an orthogonal matrix, and 
	$\bm{\Delta}^2=\textup{diag}(\rho_1^2, \ldots, \rho_m^2)$ is a diagonal matrix. 
	From Lemma \ref{lemma:truncate_AApm_BBpm}, 
	\begin{align*}
	\bm{\Gamma}'\bm{B}^{-1/2}\bm{\phi} & \sim 
	\bm{\Gamma}'\bm{B}^{-1/2}\left( \bm{B}_0 \right)^{1/2} \bm{\varepsilon} + 
	\bm{\Gamma}'\bm{B}^{-1/2}\left( \bm{B}_1 \right)^{1/2}_p\bm{\zeta}_{p,a}\\
	& \sim 
	\left( \bm{\Gamma}'\bm{B}^{-1/2}\bm{B}_0\bm{B}^{-1/2}\bm{\Gamma} \right)^{1/2}
	\bm{\varepsilon} + 
	\left( \bm{\Gamma}'\bm{B}^{-1/2} \bm{B}_1 \bm{B}^{-1/2}\bm{\Gamma} \right)^{1/2}_p\bm{\zeta}_{p,a}\\
	& \sim 
	\left\{\bm{\Gamma}'\left( \bm{I}_m - \bm{\Gamma}\bm{\Delta}^2\bm{\Gamma}' \right)\bm{\Gamma}\right\}^{1/2} 
	\bm{\varepsilon}
	+ 
	\left(\bm{\Gamma}' \bm{\Gamma}\bm{\Delta}^2\bm{\Gamma}' \bm{\Gamma}\right)^{1/2}_p \bm{\zeta}_{p,a}\\
	& \sim 
	\left( \bm{I}_m- \bm{\Delta}^2\right)^{1/2} \bm{\varepsilon}+ 
	\left(\bm{\Delta}^2\right)^{1/2}_p \bm{\zeta}_{p,a}\\
	& \sim 
	\left( \bm{I}_m- \bm{\Delta}^2\right)^{1/2} \bm{\varepsilon}+ 
	\left(\bm{\Delta}, \bm{0}_{m\times (p-m)} \right)\bm{\zeta}_{p,a}
	\end{align*}
	Let $\bm{\eta} \sim \mathcal{N}(\bm{0}, \bm{I}_m)$, 
	$\bm{\xi} \sim \mathcal{N}(\bm{0}, \bm{I}_{p-m})$, 
	and $(\bm{\eta}, \bm{\xi})$ be independent.  
	From the definition of $\bm{\zeta}_{p,a}$, for any $r^2\geq 0$, 
	\begin{align}\label{eq:general_form_refm_convex_K_simple} 
	& \quad \ P\left(
	\bm{\phi}'\bm{B}^{-1}\bm{\phi}\leq r^2
	\right) \nonumber\\
	& = 
	P\left\{
	\bm{\Gamma}'\bm{B}^{-1/2}\bm{\phi} \in \mathcal{B}_m(r)
	\right\}\nonumber\\
	& = 
	P\left\{
	\left( \bm{I}_m- \bm{\Delta}^2\right)^{1/2}\bm{\varepsilon} + 
	\left(\bm{\Delta}, \bm{0}_{m\times (p-m)} \right)\bm{\zeta}_{p,a} \in \mathcal{B}_m(r)
	\right\}
	\nonumber
	\\
	& = P\left\{
	\left. 
	\left(\bm{I}_F - \bm{\Delta}^2\right)^{1/2}\bm{\varepsilon} + \left(\bm{\Delta}, \bm{0}\right) 
	\begin{pmatrix}
	\bm{\eta}\\
	\bm{\xi}
	\end{pmatrix}
	\in  \mathcal{B}_m(r) \ 
	\right| \ \bm{\eta}'\bm{\eta} + \bm{\xi}'\bm{\xi}\leq a
	\right\}
	\nonumber
	\\
	& =  P\left\{ 
	(\bm{I}_m - \bm{\Delta}^2)^{1/2}
	\bm{\varepsilon} + 
	\bm{\Delta}\bm{\eta}
	\in \mathcal{B}_m(r)
	\mid \bm{\eta}'\bm{\eta} + \bm{\xi}'\bm{\xi}\leq a
	\right\}
	\nonumber
	\\
	& =  
	\E\left[
	P\left\{ 
	(\bm{I}_m - \bm{\Delta}^2)^{1/2}
	\bm{\varepsilon} + 
	\bm{\Delta}\bm{\eta}
	\in \mathcal{B}_m(r)
	\mid \bm{\eta}'\bm{\eta} \leq a - \bm{\xi}'\bm{\xi}, \bm{\xi}
	\right\}
	\mid 
	\bm{\eta}'\bm{\eta} + \bm{\xi}'\bm{\xi}\leq a
	\right]. 
	\end{align}
	From Lemma \ref{lemma:nondecreasing_eigenvalue}, for any given $\bm{\xi}$, the conditional probability, 
	$$P\left\{ 
	(\bm{I}_m - \bm{\Delta}^2)^{1/2}
	\bm{\varepsilon} + 
	\bm{\Delta}\bm{\eta}
	\in \mathcal{B}_m(r)
	\mid \bm{\eta}'\bm{\eta} \leq a - \bm{\xi}'\bm{\xi}, \bm{\xi}
	\right\}$$
	is nondecreasing in the diagonal elements of $\bm{\Delta}^2$. Therefore, 
	the quantity in \eqref{eq:general_form_refm_convex_K_simple} 
	depends only on $(m,p,a)$ and the eigenvalues of $\bm{B}^{-1/2}\bm{B}_1\bm{B}^{-1/2}$, and is nondecreasing in these eigenvalues. 
	Therefore, Lemma \ref{lemma:general_form_refm} holds. 
\end{proof}

\begin{proof}[{   Proof of Theorem \ref{thm:red_qr_rem_nond_ccorr}}]
	It  follows directly from  Lemma \ref{lemma:general_form_refm}.  
\end{proof}

\begin{proof}[{   Comment on 
		linear transformations of $\hat{\bm{\tau}}$ under ReFM}]
	From Theorem \ref{thm:red_qr_rem} and Lemma \ref{lemma:peak_linear}, for any $\bm{C}\in \mathbb{R}^{p\times F}$ with $p\leq F$, the asymptotic sampling distribution of $\bm{C}\hat{\bm{\tau}}$ under ReFM is more peaked than that under the CRFE. From Theorem \ref{thm:joint_refm} and Lemma \ref{lemma:truncate_AApm_BBpm}, 
	\begin{eqnarray*}
		\bm{C}(\hat{\bm{\tau}} - \bm{\tau}) \mid  \mathcal{M} 
		& \apprsim & 
		\bm{C}
		\left(\bm{V}_{\bm{\tau\tau}}^\myperp\right)^{1/2} \bm{\varepsilon} +  
		\bm{C}
		\left(\bm{V}_{\bm{\tau\tau}}^\myparallel\right)^{1/2}_{LF}
		\bm{\zeta}_{LF,a}\\
		& \apprsim & 
		\left(\bm{C}\bm{V}_{\bm{\tau\tau}}^\myperp\bm{C}'\right)^{1/2} \bm{\xi} +  
		\left(\bm{C}\bm{V}_{\bm{\tau\tau}}^\myparallel \bm{C}' \right)^{1/2}_{LF}
		\bm{\zeta}_{LF,a},
	\end{eqnarray*}
	where $\bm{\xi}\sim \mathcal{N}(\bm{0}, \bm{I}_p)$. 
	From Lemma \ref{lemma:general_form_refm}, if $\bm{C}\bm{V}_{\bm{\tau\tau}}\bm{C}'$ is invertible, then 
	for any $\alpha\in (0,1),$
	the threshold $c_{1-\alpha}$ for $1-\alpha$ quantile region $\{\bm{\mu}: \bm{\mu}' (\bm{C}\bm{V}_{\bm{\tau\tau}}\bm{C}')^{-1}\bm{\mu}\leq c_{1-\alpha}\}$ of the asymptotic sampling distribution of $\bm{C}\hat{\bm{\tau}}$ 
	depends only on $(p,LF,a)$ and the canonical correlation 
	between $\bm{C}\hat{\bm{\tau}}$ and $\hat{\bm{\tau}}_{\bm{x}}$, 
	and is nonincreasing in these canonical correlations. 
\end{proof}

\begin{proof}[{   Proof of Corollary \ref{cor:red_in_qr_single_remf}}]
	Because $\hat{\tau}_f$ is a one dimensional linear transformation of $\hat{\bm{\tau}}$, Corollary \ref{cor:red_in_qr_single_remf} follows directly from the above comment on general lower dimensional linear transformations of $\hat{\bm{\tau}}$ under ReFM. 
\end{proof}

\subsection{Peakedness under $\text{ReFMT}_\text{F}$}

\begin{proof}[{   Proof of Theorem \ref{thm:quant_region_refmt_f}}]
	For each $1\leq h\leq H$, let $\bm{D}_h \sim \mathcal{N}(\bm{0}, \bm{I}_{LF_h})$, and $(\bm{\varepsilon}, \bm{D}_1, \ldots, \bm{D}_H)$ be jointly independent. 
	From Proposition \ref{prop:ccu_refm}, both $\bm{\varepsilon}$ and $\bm{\zeta}_{LF_h, a_h}$ are central convex unimodal. 
	From Lemma \ref{lemma:ccu_linear}, both 
	$(\bm{V}_{\bm{\tau}\bm{\tau}}^\myperp)^{1/2} \bm{\varepsilon}$ and 
	$
	(
	\bm{W}_{\bm{\tau\tau}}^\myparallel[h]
	)^{1/2}_{LF_h}
	\bm{\zeta}_{LF_h, a_h}$ are central convex unimodal. Thus, by Lemma \ref{lemma:ccu_sum}, 
	\begin{align*}
	\left(
	\bm{V}_{\bm{\tau}\bm{\tau}}^\myperp \right)^{1/2} \bm{\varepsilon} + 
	\sum_{h=2}^H 
	(
	\bm{W}_{\bm{\tau\tau}}^\myparallel[h]
	)^{1/2}_{LF_h}
	\bm{\zeta}_{LF_h, a_h}
	\end{align*}
	is central convex unimodal. 
	From Lemma \ref{lemma:zeta_peak_epsilon}, $\bm{\zeta}_{LF_1, a_1} \succ \bm{D}_h$, which, based on Lemma \ref{lemma:peak_linear}, further implies that 
	$
	(
	\bm{W}_{\bm{\tau\tau}}^\myparallel[1]
	)^{1/2}_{LF_1}
	\bm{\zeta}_{LF_1, a_1} \succ (
	\bm{W}_{\bm{\tau\tau}}^\myparallel[1]
	)^{1/2}_{LF_1}\bm{D}_1.
	$
	Thus, from Lemma \ref{lemma:peak_sum}, 
	\begin{eqnarray*}
		& & \left(
		\bm{V}_{\bm{\tau}\bm{\tau}}^\myperp \right)^{1/2} \bm{\varepsilon} + 
		\sum_{h=1}^H 
		(
		\bm{W}_{\bm{\tau\tau}}^\myparallel[h]
		)^{1/2}_{LF_h}
		\bm{\zeta}_{LF_h, a_h} \\
		& \succ &  
		\left(
		\bm{V}_{\bm{\tau}\bm{\tau}}^\myperp \right)^{1/2} \bm{\varepsilon}  + (
		\bm{W}_{\bm{\tau\tau}}^\myparallel[1]
		)^{1/2}_{LF_1}
		\bm{D}_1 + 
		\sum_{h=2}^H 
		(
		\bm{W}_{\bm{\tau\tau}}^\myparallel[h]
		)^{1/2}_{LF_h}\bm{\zeta}_{LF_h, a_h}\\
		& \sim & 
		\left(
		\bm{V}_{\bm{\tau}\bm{\tau}}^\myperp + \bm{W}_{\bm{\tau\tau}}^{\myparallel}[1] \right)^{1/2} \bm{\varepsilon} + \sum_{h=2}^H 
		(
		\bm{W}_{\bm{\tau\tau}}^\myparallel[h]
		)^{1/2}_{LF_h}\bm{\zeta}_{LF_h, a_h}. 
	\end{eqnarray*}
	Because $(
	\bm{V}_{\bm{\tau}\bm{\tau}}^\myperp + \bm{W}_{\bm{\tau\tau}}^{\myparallel}[1] )^{1/2} \bm{\varepsilon} + \sum_{h=3}^H 
	(
	\bm{W}_{\bm{\tau\tau}}^\myparallel[h]
	)^{1/2}_{LF_h}\bm{\zeta}_{LF_h, a_h}$ is central convex unimodal, and 
	$
	(
	\bm{W}_{\bm{\tau\tau}}^\myparallel[2]
	)^{1/2}_{LF_2}\bm{\zeta}_{LF_2, a_2} \succ 
	(
	\bm{W}_{\bm{\tau\tau}}^\myparallel[2]
	)^{1/2}_{LF_2}\bm{D}_{2}, 
	$ 
	we have
	\begin{eqnarray*}
		& & \left(
		\bm{V}_{\bm{\tau}\bm{\tau}}^\myperp \right)^{1/2} \bm{\varepsilon} + 
		\sum_{h=1}^H 
		(
		\bm{W}_{\bm{\tau\tau}}^\myparallel[h]
		)^{1/2}_{LF_h}
		\bm{\zeta}_{LF_h, a_h} \\
		& \succ & 
		\left(
		\bm{V}_{\bm{\tau}\bm{\tau}}^\myperp + \bm{W}_{\bm{\tau\tau}}^{\myparallel}[1] \right)^{1/2} \bm{\varepsilon} + \sum_{h=2}^H 
		(
		\bm{W}_{\bm{\tau\tau}}^\myparallel[h]
		)^{1/2}_{LF_h}\bm{\zeta}_{LF_h, a_h}\\
		& \succ & 
		\left(
		\bm{V}_{\bm{\tau}\bm{\tau}}^\myperp + \bm{W}_{\bm{\tau\tau}}^{\myparallel}[1] \right)^{1/2} \bm{\varepsilon} 
		+
		(
		\bm{W}_{\bm{\tau\tau}}^\myparallel[2]
		)^{1/2}_{LF_2}\bm{D}_{2}
		+ \sum_{h=3}^H 
		(
		\bm{W}_{\bm{\tau\tau}}^\myparallel[h]
		)^{1/2}_{LF_h}\bm{\zeta}_{LF_h, a_h}\\
		& \sim & 
		\left(
		\bm{V}_{\bm{\tau}\bm{\tau}}^\myperp + \bm{W}_{\bm{\tau\tau}}^{\myparallel}[1]
		+ \bm{W}_{\bm{\tau\tau}}^{\myparallel}[2]
		\right)^{1/2} \bm{\varepsilon} 
		+ \sum_{h=3}^H 
		(
		\bm{W}_{\bm{\tau\tau}}^\myparallel[h]
		)^{1/2}_{LF_h}\bm{\zeta}_{LF_h, a_h}. 
	\end{eqnarray*}
	Implementing the above procedure iteratively, we finally have
	$$
	\left(
	\bm{V}_{\bm{\tau}\bm{\tau}}^\myperp \right)^{1/2} \bm{\varepsilon} + 
	\sum_{h=1}^H 
	(
	\bm{W}_{\bm{\tau\tau}}^\myparallel[h]
	)^{1/2}_{LF_h}
	\bm{\zeta}_{LF_h, a_h} \succ 
	\left(
	\bm{V}_{\bm{\tau}\bm{\tau}}^\myperp + \sum_{h=1}^{H} \bm{W}_{\bm{\tau\tau}}^{\myparallel}[h]
	\right)^{1/2} \bm{\varepsilon} 
	\sim \mathcal{N}(\bm{0}, \bm{V}_{\bm{\tau\tau}}),
	$$
	where the last formula follows from
	$
	\bm{V}_{\bm{\tau\tau}}^\myparallel = \sum_{h=1}^H  
	\bm{W}_{\bm{\tau \tau}}^\myparallel[h]
	$
	in the 
	proof of Theorem \ref{thm:red_var_remft_f}.  
\end{proof}

To prove Theorem \ref{thm:quant_region_refmt_f_monotone_joint}, we need the following three lemmas.

\begin{lemma}\label{lemma:order_rho_mono_tier}
	Let $\varepsilon_0 \sim \mathcal{N}(0,1)$, $\eta_{k_t,a_t}\sim D_{t1}\mid \bm{D}_t'\bm{D}_t\leq a_t$, where $\bm{D}_t=(D_{t1},\ldots,D_{tk_t})\sim \mathcal{N}(\bm{0},\bm{I}_{k_t})$, $a_t$ is a nonnegative constant that can be infinity, 
	and $(\varepsilon_0, \eta_{k_1,a_1}, \eta_{k_2,a_2}, \ldots, \eta_{k_T,a_T})$ are jointly independent.
	Let $\{\rho_t\}_{t=1}^{T+1}$ and $\{\tilde{\rho}_t\}_{t=1}^{T+1}$ be
	two nonnegative constant sequences  satisfying $\sum_{t=1}^{T+1} \rho_t^2=\sum_{t=1}^{T+1} {\tilde{\rho}_t}^2= 1.$ If 
	$\rho_{t} \geq \tilde{\rho}_{t}$ for all $1\leq t\leq T$, 
	then for any $c\geq 0$,
	\begin{align*}
	& P\left(
	| \rho_{T+1}\varepsilon_0 + \sum_{t=1}^\Tau \rho_t \eta_{k_t,a_t} | \leq c
	\right)\geq
	P\left(
	|\tilde{\rho}_{T+1}\varepsilon_0 + \sum_{t=1}^\Tau \tilde{\rho}_t \eta_{k_t,a_t} | \leq c 
	\right).
	\end{align*}
\end{lemma}
\begin{proof}[Proof of Lemma \ref{lemma:order_rho_mono_tier}]
	It follows from \citet[][Lemma A10]{asymrerand2016}. 
\end{proof}

The following lemma extends the above lemma to the multivariate case. 
We introduce the following condition for a set of matrices with the same number of rows.  
\begin{condition}\label{cond:cond_sparse_matrix_tier}
	The $H$ matrices $\{\bm{\Delta}_h\in \mathbb{R}^{m\times p_h}\}_{h=1}^H$ satisfy that 
	\begin{itemize}
		\item[(a)] there is at most one nonzero element at each column and each row of $\bm{\Delta}_h$ for $1\leq h\leq H$;
		\item[(b)] the elements of $\bm{\Delta}_h$ are all nonnegative for  $1\leq h\leq H$;
		\item[(c)] $\sum_{h=1}^{H}\bm{\Delta}_h\bm{\Delta}_h' \in \mathbb{R}^{m\times m}$ has all elements less than or equal to 1 (note that (a) implies that $\sum_{h=1}^{H}\bm{\Delta}_h\bm{\Delta}_h' \in \mathbb{R}^{m\times m}$ is a diagonal matrix). 
	\end{itemize}
\end{condition}

\begin{lemma}\label{lemma:matrix_tier_monotone}
	Let $(\bm{\varepsilon},\bm{\zeta}_{p_1,a_1}, \ldots, \bm{\eta}_{p_H,a_H})$ be $H+1$ independent random vectors, where $\bm{\varepsilon}\sim \mathcal{N}(\bm{0}, \bm{I}_m)$, 
	$\bm{\zeta}_{p_h,a_h}\sim \bm{D}_h \mid \bm{D}_h'\bm{D}_h\leq a_h$ is a truncated Gaussian random vector with $a_h\geq 0$ and $\bm{D}_h\sim \mathcal{N}(\bm{0}, \bm{I}_{p_h})$. 
	Assume $\{\bm{\Delta}_h\in \mathbb{R}^{m\times p_h}\}_{h=1}^H$ satisfy Condition \ref{cond:cond_sparse_matrix_tier}. 
	For any $r\geq 0$, $1\leq h\leq H$, and $1\leq k\leq m$ the probability 
	$$
	P\left\{
	\left(\bm{I}_m - \sum_{h=1}^{H}\bm{\Delta}_h\bm{\Delta}_h'\right)^{1/2}
	\bm{\varepsilon} + 
	\sum_{h=1}^{H}
	\bm{\Delta}_h\bm{\zeta}_{p_h,a_h} \in \mathcal{B}_m(r)
	\right\}
	$$
	is nondecreasing in the nonzero elements of the $\bm{\Delta}_h$'s, that is, for any two sets of matrices $\bm{\Delta}_h$'s and  $\tilde{\bm{\Delta}}_h$'s satisfying Condition \ref{cond:cond_sparse_matrix_tier} with the positions of possible nonzero elements being the same, 
	if all elements of $\bm{\Delta}_h$'s are larger than or equal to $\tilde{\bm{\Delta}}_h$'s, then 
	\begin{eqnarray}\label{eq:nondecreasing_eigenvalue_tier}
	& & P\left\{ 
	\left(\bm{I}_m - \sum_{h=1}^{H}\bm{\Delta}_h\bm{\Delta}_h'\right)^{1/2}
	\bm{\varepsilon} + 
	\sum_{h=1}^{H}
	\bm{\Delta}_h\bm{\zeta}_{p_h,a_h} \in \mathcal{B}_m(r)
	\right\} 
	\nonumber
	\\
	& \geq & 
	P\left\{ 
	\left(\bm{I}_m - \sum_{h=1}^{H}\tilde{\bm{\Delta}}_h\tilde{\bm{\Delta}}_h'\right)^{1/2}
	\bm{\varepsilon} + 
	\sum_{h=1}^{H}
	\tilde{\bm{\Delta}}_h\bm{\zeta}_{p_h,a_h} \in \mathcal{B}_m(r)
	\right\}.
	\end{eqnarray}
\end{lemma}

\begin{proof}[Proof of Lemma \ref{lemma:matrix_tier_monotone}]
	It suffices to prove that for any two sets of matrices $\bm{\Delta}_h$'s and  $\tilde{\bm{\Delta}}_h$'s satisfying Condition \ref{cond:cond_sparse_matrix_tier} with the positions of possible nonzero elements being the same, 
	if there exists $1\leq k\leq m$ such that for all $1\leq h\leq H$, 
	(a) $\bm{\Delta}_h$ and $\tilde{\bm{\Delta}}_h$ differ only in the $k$th row, and (b) the elements in the $k$th row of $\bm{\Delta}_h$ are larger than or equal to that of $\tilde{\bm{\Delta}}_h$, then \eqref{eq:nondecreasing_eigenvalue_tier} holds for any $r\geq 0$. 
	First, by symmetry, we consider only the case with $k=1$. 
	Second,  without loss of generality, we assume that the possible nonzero elements in the first rows of $\bm{\Delta}_h$'s and $\tilde{\bm{\Delta}}_h$'s are all in the first columns. 
	This is because permuting the columns of $\bm{\Delta}_h$ will not change the distribution of $\bm{\Delta}_h\bm{\zeta}_{p_h,a_h}$, a fact implied by the spherical symmetry of $\bm{\zeta}_{p_h,a_h}$.

	Let $\Delta_{h1}$ be the $(1,1)$th element of $\bm{\Delta}_h$, and $\bm{\Delta}_{h,-1}$ be the submatrix of $\bm{\Delta}_h$ excluding the first column and the first row. 
	Define similarly $\tilde{\Delta}_{h1}$ and $\tilde{\bm{\Delta}}_{h,-1}$. 
	Let  
	$\bm{\varepsilon}_{-1}=(\varepsilon_2, \ldots, \varepsilon_m)$ and 
	$\bm{D}_{h,-1}=(D_{h2}, \ldots, D_{h,p_h})$
	be the subvectors of $\bm{\varepsilon}$ and $\bm{D}_h$ excluding the first elements.  
	We define a subset in $\mathbb{R}$ as follows: 
	\begin{eqnarray*}
		& & \mathcal{B}(r,\bm{\varepsilon}_{-1},\bm{D}_{h,-1},\bm{\Delta}_{h,-1}) \\
		& = & 
		\left\{
		x: 
		\begin{pmatrix}
			x\\
			\left(\bm{I}_{m-1} - \sum_{h=1}^{H}\bm{\Delta}_{h,-1}\bm{\Delta}_{h,-1}'\right)^{1/2}
			\bm{\varepsilon}_{-1} + 
			\sum_{h=1}^{H}
			\bm{\Delta}_{h,-1}\bm{D}_{h,-1}
		\end{pmatrix}\in \mathcal{B}_m(r)
		\right\},
	\end{eqnarray*}
	which depends on $\bm{\varepsilon}_{-1}, \bm{D}_{h,-1}$ and $\bm{\Delta}_{h,-1}$ for all $1\leq h\leq H$, 
	and is either an empty set or a symmetric closed interval on the real line. 
	For any $r\geq 0$ and $(\bm{\varepsilon}_{-1},\bm{D}_{1,-1}, \ldots, \bm{D}_{H,-1})$,  
	{\small
		\begin{align}\label{eq:conditional_nondecreasing_eigenvalue_tier_proof}
		& \quad \  P\left\{
		\left(\bm{I}_m - \sum_{h=1}^{H}\bm{\Delta}_h\bm{\Delta}_h'\right)^{1/2}
		\bm{\varepsilon} + 
		\sum_{h=1}^{H}
		\bm{\Delta}_h\bm{D}_{h} \in \mathcal{B}_m(r)
		\mid 
		\bm{D}_{h}'\bm{D}_{h}\leq a_h, \bm{\varepsilon}_{-1}, \bm{D}_{h,-1}, \forall h
		\right\}
		\nonumber\\
		& =  
		P\left\{
		\left(1-\sum_{h=1}^{H}\Delta_{h1}^2\right)^{1/2} \varepsilon_1 + 
		\sum_{h=1}^{H} \Delta_{h1} D_{h1} \in \mathcal{B}(r,\bm{\varepsilon}_{-1},\bm{D}_{h,-1},\bm{\Delta}_{h,-1})
		\mid 
		D_{h1}^2 \leq a_h - \bm{D}_{h,-1}'\bm{D}_{h,-1},
		\bm{\varepsilon}_{-1}, \bm{D}_{h,-1}
		\right\}
		\nonumber\\
		& \geq 
		P\left\{
		\left(1-\sum_{h=1}^{H}\tilde{\Delta}_{h1}^2\right)^{1/2} \varepsilon_1 + 
		\sum_{h=1}^{H} \tilde{\Delta}_{h1} D_{h1} \in \mathcal{B}(r,\bm{\varepsilon}_{-1},\bm{D}_{h,-1},\tilde{\bm{\Delta}}_{h,-1})
		\mid 
		D_{h1}^2 \leq a_h - \bm{D}_{h,-1}'\bm{D}_{h,-1},
		\bm{\varepsilon}_{-1}, \bm{D}_{h,-1}
		\right\}
		\nonumber
		\\
		& = 
		P\left\{
		\left(\bm{I}_m - \sum_{h=1}^{H}\tilde{\bm{\Delta}}_h \tilde{\bm{\Delta}}_h'\right)^{1/2}
		\bm{\varepsilon} + 
		\sum_{h=1}^{H}
		\tilde{\bm{\Delta}}_h\bm{D}_{h} \in \mathcal{B}_m(r)
		\mid 
		\bm{D}_{h}'\bm{D}_{h}\leq a_h, \bm{\varepsilon}_{-1}, \bm{D}_{h,-1}, \forall h
		\right\}, 
		\end{align}}%
	where the second last inequality follows from Lemma \ref{lemma:order_rho_mono_tier}. 
	Taking expectations of both sides of \eqref{eq:conditional_nondecreasing_eigenvalue_tier_proof}, we 
	obtain \eqref{eq:nondecreasing_eigenvalue_tier}. 
\end{proof}

The following lemma extends Lemma \ref{lemma:general_form_refm} to general case with $H\geq 1$. Moreover, even when $H=1$, Lemma \ref{lemma:general_mono_ccorr} is still more general then Lemma \ref{lemma:general_form_refm} by only requiring $p_1\geq \text{rank}(\bm{B}_1)$, instead of $p_1\geq m$ in Lemma \ref{lemma:general_form_refm}. 

\begin{lemma}\label{lemma:general_mono_ccorr}
	Let $(\bm{B}_0, \bm{B}_1, \ldots, \bm{B}_H)$ be $H+1$ positive semi-definite matrices in $\mathbb{R}^{m\times m}$ with ranks $(\gamma_0, \gamma_1, \ldots, \gamma_H)$, $\bm{\varepsilon}\sim \mathcal{N}(\bm{0}, \bm{I}_m)$ be a standard Gaussian random vector, and $\bm{\zeta}_{p_h,a_h}\sim \bm{D}_h \mid \bm{D}_h'\bm{D}_h\leq a_h$ be a truncated Gaussian random vector, where $\bm{D}_h\sim \mathcal{N}(\bm{0}, \bm{I}_{p_h})$ and $p_h\geq \gamma_h$  ($h=1,2,\ldots,H$). Define 
	$
	\bm{\phi} \sim (\bm{B}_0)^{1/2} \bm{\varepsilon} + \sum_{h=1}^{H} (\bm{B}_h)^{1/2}_{p_h}\bm{\zeta}_{p_h,a_h}.
	$
	If $\bm{B} = \sum_{h=0}^{H}\bm{B}_h$ is invertible, and there exists an orthogonal matrix $\bm{\Gamma}$ such that for all $1\leq h\leq H$,
	$
	\bm{\Gamma}' 
	\bm{B}^{-1/2}
	\bm{B}_h
	\bm{B}^{-1/2} \bm{\Gamma}=\bm{\Omega}_h^2 
	$
	is a diagonal matrix, then the threshold $c_{1-\alpha}$ for $1-\alpha$ quantile region 
	$
	\{\bm{\mu}: \bm{\mu}'\bm{B}^{-1}\bm{\mu} \leq c_{1-\alpha}
	\}
	$
	of $\bm{\phi}$ depends only on  $m$, the $p_h$'s, the $a_h$'s, and the eigenvalues of $\bm{B}^{-1/2}\bm{B}_h\bm{B}^{-1/2}$, and is nonincreasing in these eigenvalues. 
\end{lemma}
\begin{proof}[Proof of Lemma \ref{lemma:general_mono_ccorr}]
	By definition, 
	\begin{align*} 
	\bm{\Gamma}' 
	\bm{B}^{-1/2}\bm{B}_0
	\bm{B}^{-1/2} \bm{\Gamma}=
	\bm{\Gamma}' \left(
	\bm{I}_m - \sum_{h=1}^{H}\bm{B}^{-1/2}\bm{B}_h
	\bm{B}^{-1/2}
	\right)
	\bm{\Gamma}
	=
	\bm{I}_m - \sum_{h=1}^{H} \bm{\Omega}_h^2. 
	\end{align*}
	Thus, from Lemma \ref{lemma:truncate_AApm_BBpm}, 
	the distribution of $\bm{\phi}$ satisfies 
	\begin{eqnarray*}
		\bm{\Gamma}'\bm{B}^{-1/2} \bm{\phi}& \sim & \bm{\Gamma}'\bm{B}^{-1/2}
		\left\{
		(\bm{B}_0)^{1/2} \bm{\varepsilon} + \sum_{h=1}^{H} (\bm{B}_h)^{1/2}_{p_h}\bm{\zeta}_{p_h,a_h}
		\right\}\\
		& \sim & 
		\left(
		\bm{\Gamma}'\bm{B}^{-1/2}\bm{B}_0  \bm{B}^{-1/2}\bm{\Gamma} \right)^{1/2}
		\bm{\varepsilon} + \sum_{h=1}^H
		\left(\bm{\Gamma}'\bm{B}^{-1/2}\bm{B}_h 
		\bm{B}^{-1/2}\bm{\Gamma}
		\right)^{1/2}_{p_h}
		\bm{\zeta}_{p_h, a_h}\\
		& \sim & 
		\left(
		\bm{I}_m - 
		\sum_{h=1}^{H}\bm{\Omega}_h^2
		\right)^{1/2}
		\bm{\varepsilon} + \sum_{h=1}^H
		\left(
		\bm{\Omega}_h^2
		\right)^{1/2}_{p_h}
		\bm{\zeta}_{p_h, a_h}. 
	\end{eqnarray*}
	For each $1\leq h\leq H$, if $p_h \geq m$, we further define 
	$
	\tilde{\bm{\Omega}}_h = (\bm{\Omega}_h, \bm{0}_{m\times (p_h-m)});
	$
	otherwise, $p_h < m$, $\bm{B}_h$ has rank at most $p_h$ and thus $\bm{B}^{-1/2}\bm{B}_h\bm{B}^{-1/2}$ has at most $p_h$ nonzero eigenvalues, we further define $
	\tilde{\bm{\Omega}}_h = \bm{\Omega}_h[, \mathcal{I}_h]
	$
	as the submatrix of $\bm{\Omega}_h$ consisting of the $|\mathcal{I}_h|=p_h$ columns with possible nonzero eigenvalues. 
	Thus, by the construction of $\tilde{\bm{\Omega}}_h$'s, we have 
	$\tilde{\bm{\Omega}}_h\in \mathbb{R}^{m\times p_h}$ and 
	$\tilde{\bm{\Omega}}_h\tilde{\bm{\Omega}}_h' = \bm{\Omega}_h^2$, i.e., 
	$
	\tilde{\bm{\Omega}}_h = (\bm{\Omega}_h^2)^{1/2}_{p_h}. 
	$
	Therefore, we can further simply the distribution of  $\bm{\Gamma}'\bm{B}^{-1/2} \bm{\phi}$ as 
	\begin{eqnarray*}
		\bm{\Gamma}'\bm{B}^{-1/2} \bm{\phi} 
		& \sim & 
		\left(
		\bm{I}_m - 
		\sum_{h=1}^{H}\tilde{\bm{\Omega}}_h\tilde{\bm{\Omega}}_h'
		\right)^{1/2}
		\bm{\varepsilon} + \sum_{h=1}^H
		\tilde{\bm{\Omega}}_h 
		\bm{\zeta}_{p_h, a_h}. 
	\end{eqnarray*} 
	For any $r\geq 0$,  we have
	\begin{align}\label{eq:gen_from_refm_t_f_convex_K}
	P(\bm{\phi}'\bm{B}^{-1}\bm{\phi}\leq r^2)
	& =   
	P\left\{
	\bm{\Gamma}'\bm{B}^{-1/2} \bm{\phi} 
	\in  \mathcal{B}_m(r)
	\right\} \nonumber\\
	& =  P\left\{
	\left(
	\bm{I}_m - 
	\sum_{h=1}^{H}\bm{\Omega}_h^2
	\right)^{1/2}
	\bm{\varepsilon} + \sum_{h=1}^H
	\left(
	\bm{\Omega}_h^2
	\right)^{1/2}_{p_h}
	\bm{\zeta}_{p_h, a_h}
	\in  \mathcal{B}_m(r)
	\right\} \nonumber
	\\
	& = 
	P\left\{
	\left(
	\bm{I}_m - 
	\sum_{h=1}^{H}\tilde{\bm{\Omega}}_h\tilde{\bm{\Omega}}_h'
	\right)^{1/2}
	\bm{\varepsilon} + \sum_{h=1}^H
	\tilde{\bm{\Omega}}_h 
	\bm{\zeta}_{p_h, a_h}
	\in  \mathcal{B}_m(r)
	\right\}. 
	\end{align}
	Because the matrices $\tilde{\bm{\Omega}}_h$'s satisfy Condition \ref{cond:cond_sparse_matrix_tier}, 
	from Lemma \ref{lemma:matrix_tier_monotone}, 
	the quantity in \eqref{eq:gen_from_refm_t_f_convex_K} 
	depends only on  $m$, $p_h$'s, $a_h$'s, and the eigenvalues of $\bm{B}^{-1/2}\bm{B}_h\bm{B}^{-1/2}$'s, and is nondecreasing in these eigenvalues. 
	Therefore, Lemma \ref{lemma:general_mono_ccorr} holds. 
\end{proof}

\begin{proof}[{   Proof of Theorem \ref{thm:quant_region_refmt_f_monotone_joint}}]
	It follows directly from Lemma \ref{lemma:general_mono_ccorr}. 
\end{proof}

\begin{proof}[{   Comment on 
		linear transformations of $\hat{\bm{\tau}}$ under $\text{ReFMT}_\text{F}$}]
	From Theorem \ref{thm:quant_region_refmt_f} and Lemma \ref{lemma:peak_linear}, for any $\bm{C}\in \mathbb{R}^{p\times F}$ with $p\leq F$, the asymptotic sampling distribution of $\bm{C}\hat{\bm{\tau}}$ under $\text{ReFMT}_\text{F}$ is more peaked than that under the CRFE. From Theorem \ref{thm:asymp_refmt_f} and Lemma \ref{lemma:truncate_AApm_BBpm}, 
	\begin{eqnarray*}
		\bm{C}(\hat{\bm{\tau}} - \bm{\tau})  \mid \mathcal{T}_{\text{F}} 
		& \apprsim & \bm{C}\left( \bm{V}_{\bm{\tau}\bm{\tau}}^\myperp \right)^{1/2} \bm{\varepsilon} + \sum_{h=1}^H \bm{C}
		\left(\bm{W}_{\bm{\tau\tau}}^\myparallel[h] \right)^{1/2}_{LF_h}
		\bm{\zeta}_{LF_h, a_h}\\
		& \apprsim & 
		\left(\bm{C}\bm{V}_{\bm{\tau\tau}}^\myperp\bm{C}'\right)^{1/2} \bm{\xi} +  
		\sum_{h=1}^H
		\left(\bm{C}\bm{W}_{\bm{\tau\tau}}^\myparallel[h] \bm{C}' \right)^{1/2}_{LF_h}
		\bm{\zeta}_{LF_h,a_h},
	\end{eqnarray*}
	where $\bm{\xi}\sim \mathcal{N}(\bm{0}, \bm{I}_p)$. 
	Based on Lemma \ref{lemma:general_mono_ccorr}, we can know that if the condition in Lemma \ref{lemma:general_mono_ccorr} holds for $\bm{B}_h =\bm{C}\bm{W}_{\bm{\tau\tau}}^\myparallel[h] \bm{C}'$,  
	then for any $\alpha\in (0,1)$, the threshold $c_{1-\alpha}$ for $1-\alpha$ quantile region $\{\bm{\mu}: \bm{\mu}' (\bm{C}\bm{V}_{\bm{\tau\tau}}\bm{C}')^{-1}\bm{\mu}\leq c_{1-\alpha}\}$ of the asymptotic sampling distribution of $\bm{C}\hat{\bm{\tau}}$ 
	depends only on $p, LF_h$'s, $a_h$'s, and the canonical correlation 
	between $\bm{C}\hat{\bm{\tau}}$ and $\hat{\bm{\theta}}_{\bm{x}}[h]$'s, 
	and is nonincreasing in these canonical correlations.
\end{proof}

\begin{proof}[{   Proof of Corollary \ref{cor:quant_region_refmt_f_monotone_marg}}]
	Because $\hat{\tau}_f$ is a one dimensional linear transformation of $\hat{\bm{\tau}}$, Corollary \ref{cor:quant_region_refmt_f_monotone_marg} follows directly from the above comment on general lower dimensional linear transformations of $\hat{\bm{\tau}}$ under $\text{ReFMT}_\text{F}$. 
\end{proof}

The proofs of Theorems \ref{thm:quant_region_refmt_cf}, \ref{thm:quant_region_refmt_cf_monotone_joint}, and Corollary \ref{cor:red_qr_refmt_cf} under $\text{ReFMT}_\text{CF}$, as well as the comment on lower dimensional linear transformations of $\hat{\bm{\tau}}$, are almost the same as those under $\text{ReFMT}_\text{F}$ and thus omitted. 

\section{Asymptotic conservativeness in inference}\label{app:inference}

\subsection{Asymptotic conservativeness of sampling covariance estimators}\label{sec:proof_conservative_variance}
We need the following two lemmas to prove the asymptotic conservativeness of the sampling covariance estimators. 
\begin{lemma}\label{lemma:cov_AB}
	Under either ReFM, $\text{ReFMT}_\text{F}$, or $\text{ReFMT}_\text{CF}$, 
	if Condition \ref{cond:fp} holds, then for any $1\leq r,k\leq Q$, $1\leq l,m \leq L$, and any $(A_i,B_i)$ equal to $(Y_i(r), Y_i(k)), (Y_i(r), x_{il})$ or $(x_{il}, x_{im})$,
	\begin{align*}
	s_{AB}(q)-S_{AB} = o_p(1), \quad (q=1,2,\ldots,Q)
	\end{align*}
	where 
	$s_{AB}(q)$ is the sample covariance between the $A_i$'s and the $B_i$'s under treatment combination $q$, and $S_{AB}$ is the corresponding finite population covariance.

\end{lemma}
\begin{proof}[Proof of Lemma \ref{lemma:cov_AB}]
	The proof is similar to the proof of Lemma A15 in \citet{asymrerand2016}. We omit it.
\end{proof}

\begin{lemma}\label{lemma:unbiased_sample_var}
	Under either ReFM, $\text{ReFMT}_\text{F}$, or $\text{ReFMT}_\text{CF}$, 
	if Condition \ref{cond:fp} holds,  then 
	for each $1\leq q\leq Q$, 
	\begin{align*}
	s_{qq}^{\myperp}-S_{qq}^{\myperp} = o_p(1), \ \ 
	\bm{s}_{q,\bm{x}}-\bm{S}_{q,\bm{x}} = o_p(1),\ \ 
	\bm{s}_{\bm{xx}}(q)-\bm{S}_{\bm{xx}} = o_p(1).
	\end{align*}
\end{lemma}

\begin{proof}[Proof of Lemma \ref{lemma:unbiased_sample_var}]
	From Lemma \ref{lemma:cov_AB},  
	$s_{qq}$, $\bm{s}_{q,\bm{x}}$ and $\bm{s}_{\bm{xx}}(q)$ are consistent for $S_{qq}, \bm{S}_{q,\bm{x}}$ and $\bm{S}_{\bm{xx}}$, respectively. 
	Because  
	$
	s_{qq}^{\myperp} = s_{qq} - \bm{s}_{q,\bm{x}}\bm{s}_{\bm{xx}}^{-1}(q)\bm{s}_{\bm{x}, q}
	$
	and 
	$
	S_{qq}^{\myperp} = S_{qq} - \bm{S}_{q,\bm{x}}\bm{S}_{\bm{xx}}^{-1}\bm{S}_{\bm{x}, q}, 
	$
	we know that $s_{qq}^{\myperp}$ is also consistent for $S_{qq}^{\myperp}$, and therefore Lemma \ref{lemma:unbiased_sample_var} holds. 
\end{proof}

Define
$
\tilde{\bm{V}}_{\bm{\tau}\bm{\tau}}^\myperp \equiv
2^{-2(K-1)} \sum_{q=1}^{Q} n_q^{-1} \coef_q\coef_q' \cdot S_{qq}^{\myperp} \geq \bm{V}_{\bm{\tau}\bm{\tau}}^\myperp. 
$
Under ReFM, 
from Lemma \ref{lemma:unbiased_sample_var}, $\hat{\bm{V}}_{\bm{\tau\tau}}^\myperp$ is consistent for $\tilde{\bm{V}}_{\bm{\tau\tau}}^\myperp$, and $\hat{\bm{V}}_{\bm{\tau x}}\bm{V}_{\bm{xx}}^{-1/2}$ is consistent for $\bm{V}_{\bm{\tau x}}\bm{V}_{\bm{xx}}^{-1/2}$. Therefore, the sampling covariance estimator is asymptotically conservative. 
Under $\text{ReFMT}_\text{F}$ or $\text{ReFMT}_\text{CF}$, $\hat{\bm{V}}_{\bm{\tau\tau}}^\myperp$ is also consistent for $\tilde{\bm{V}}_{\bm{\tau\tau}}^\myperp$, and the estimated coefficients of the $\bm{\zeta}_{LF_h,a_h}$'s
or 
$\bm{\zeta}_{\lambda_j,a_j}$'s are consistent for the true ones. 
Therefore, the sampling covariance estimators under $\text{ReFMT}_\text{F}$ and $\text{ReFMT}_\text{CF}$ are also asymptotically conservative.

\subsection{Asymptotic conservativeness of the confidence sets}

We need the following lemma to prove Theorem \ref{thm:conser_conf_set_refm}, 
\begin{lemma}\label{lemma:peak_two_normal}
	Let $\bm{V}_1$ and $\bm{V}_2$ be two positive semi-definite matrices in $\mathbb{R}^{m\times m}$ satisfying that $\bm{V}_1\leq \bm{V}_2$,  and $\bm{\varepsilon}_1$ and $\bm{\varepsilon}_2$ be two Gaussian random vectors with mean zero and covariance matrices $\bm{V}_1$ and $\bm{V}_2$. Then $\bm{\varepsilon}_1 \succ \bm{\varepsilon}_2$. 
\end{lemma}

\begin{proof}[Proof of Lemma \ref{lemma:peak_two_normal}]
	Let $\bm{\varepsilon}_3 \sim \mathcal{N}(\bm{0}, \bm{V}_2-\bm{V}_1)$ be independent of $\bm{\varepsilon}_1$. From Proposition \ref{prop:ccu_refm} and Lemma \ref{lemma:ccu_linear}, $\bm{\varepsilon}_1$ is central convex unimodal. Because $\bm{0}\succ \bm{\varepsilon}_3$ and $\bm{\varepsilon}_3+\bm{\varepsilon}_1\sim \bm{\varepsilon}_2$, from Lemma \ref{lemma:peak_sum}, $\bm{\varepsilon}_1 \sim \bm{0}+\bm{\varepsilon}_1\succ \bm{\varepsilon}_3+\bm{\varepsilon}_1\sim \bm{\varepsilon}_2$. 
	Lemma \ref{lemma:peak_two_normal} holds.
\end{proof}

\begin{proof}[{   Proof of Theorem \ref{thm:conser_conf_set_refm}}]
	The proof of the asymptotic conservativeness of covariance estimator for $\bm{C}\hat{\bm{\tau}}$ follows directly from the discussion in \ref{sec:proof_conservative_variance}, and thus we consider only the asymptotic conservativeness of confidence sets for $\bm{C}{\bm{\tau}}$ here. 
	Let $ {\mathcal{L}}$, $\tilde{\mathcal{L}}$ and $ \hat{\mathcal{L}}$ be three $F$ dimensional random vectors following the asymptotic sampling distribution of $\hat{\bm{\tau}}-\bm{\tau}$ under ReFM,  the asymptotic sampling distribution with $\bm{V}_{\bm{\tau\tau}}^\myperp$ replaced by $\tilde{\bm{V}}_{\bm{\tau\tau}}^\myperp$, and 
	the estimated asymptotic sampling distribution: 
	\begin{eqnarray*}
		{\mathcal{L}}
		& \sim & \left(\bm{V}_{\bm{\tau\tau}}^\myperp\right)^{1/2} \bm{\varepsilon} +  \left(\bm{V}_{\bm{\tau\tau}}^\myparallel\right)^{1/2}_{LF} \bm{\zeta}_{LF,a}
		\ \sim \ 
		\left(\bm{V}_{\bm{\tau\tau}}^\myperp\right)^{1/2} \bm{\varepsilon} +  \bm{V}_{\bm{\tau x}}\bm{V}_{\bm{xx}}^{-1/2} \bm{\zeta}_{LF,a},
		\\
		\tilde{\mathcal{L}}
		& \sim & \left(\bm{V}_{\bm{\tau\tau}}^\myperp\right)^{1/2} \bm{\varepsilon} +  \left(\bm{V}_{\bm{\tau\tau}}^\myparallel\right)^{1/2}_{LF} \bm{\zeta}_{LF,a}
		\ \sim \ 
		\left(\tilde{\bm{V}}_{\bm{\tau\tau}}^\myperp\right)^{1/2} \bm{\varepsilon} +  \bm{V}_{\bm{\tau x}}\bm{V}_{\bm{xx}}^{-1/2} \bm{\zeta}_{LF,a},
		\\
		\hat{\mathcal{L}}
		& \sim & 
		\left(\hat{\bm{V}}_{\bm{\tau\tau}}^\myperp\right)^{1/2} \bm{\varepsilon} +  \hat{\bm{V}}_{\bm{\tau x}}\bm{V}_{\bm{xx}}^{-1/2} \bm{\zeta}_{LF,a}. 
	\end{eqnarray*}
	Under Condition \ref{cond:fp}, from the discussion in Section \ref{sec:proof_conservative_variance}, by Slutsky's theorem,  $ \hat{\mathcal{L}}\apprsim \tilde{\mathcal{L}}$. Note that $(\bm{V}_{\bm{\tau\tau}}^\myperp)^{1/2} \bm{\varepsilon}\succ (\tilde{\bm{V}}_{\bm{\tau\tau}}^\myperp)^{1/2} \bm{\varepsilon}$ from Lemma \ref{lemma:peak_two_normal}, and $\bm{V}_{\bm{\tau x}}\bm{V}_{\bm{xx}}^{-1/2} \bm{\zeta}_{LF,a}$ is central convex unimodal from Proposition \ref{prop:ccu_refm} and Lemma \ref{lemma:ccu_linear}. From Lemma \ref{lemma:peak_sum}, $ {\mathcal{L}}\succ \tilde{\mathcal{L}}$. Above all, 
	$
	{\mathcal{L}} \succ 
	\tilde{\mathcal{L}} \apprsim  \hat{\mathcal{L}}.
	$
	
	From Slutsky's theorem,  
	$
	(\bm{C}\tilde{\bm{V}}_{\bm{\tau\tau}}^\myperp \bm{C}')^{-1/2}\bm{C}\tilde{\mathcal{L}} \apprsim (\bm{C}\hat{\bm{V}}_{\bm{\tau\tau}}^\myperp \bm{C}')^{-1/2}\bm{C} \hat{\mathcal{L}}. 
	$
	From the continuous mapping theorem, 
	$
	(\bm{C}\tilde{\mathcal{L}})'
	(\bm{C}\tilde{\bm{V}}_{\bm{\tau\tau}}^\myperp \bm{C}')^{-1}\bm{C}\tilde{\mathcal{L}} \apprsim
	(\bm{C} \hat{\mathcal{L}})' (\bm{C}\hat{\bm{V}}_{\bm{\tau\tau}}^\myperp \bm{C}')^{-1}\bm{C} \hat{\mathcal{L}}. 
	$
	Thus, the $1-\alpha$ quantile of $(\bm{C} \hat{\mathcal{L}})' (\bm{C}\hat{\bm{V}}_{\bm{\tau\tau}}^\myperp \bm{C}')^{-1}\bm{C} \hat{\mathcal{L}}$, $\hat{c}_{1-\alpha}$, is consistent for the $1-\alpha$ quantile of 
	$(\bm{C}\tilde{\mathcal{L}})'
	(\bm{C}\tilde{\bm{V}}_{\bm{\tau\tau}}^\myperp \bm{C}')^{-1}\bm{C}\tilde{\mathcal{L}}$, $\tilde{c}_{1-\alpha}$. 
	Because 
	$ {\mathcal{L}} \succ 
	\tilde{\mathcal{L}}$, and the set of form  
	$\{\bm{\mu}:(\bm{C}\bm{\mu})'
	(\bm{C}\tilde{\bm{V}}_{\bm{\tau\tau}}^\myperp \bm{C}')^{-1}\bm{C}\bm{\mu}\leq c\}$ 
	is symmetric convex, 
	$ {c}_{1-\alpha}$, 
	the $1-\alpha$ quantile of 
	$(\bm{C} {\mathcal{L}})'
	(\bm{C}\tilde{\bm{V}}_{\bm{\tau\tau}}^\myperp \bm{C}')^{-1}\bm{C} {\mathcal{L}}$, is smaller than or equal to 
	$\tilde{c}_{1-\alpha}$. Above all, $\hat{c}_{1-\alpha}$ is consistent for $\tilde{c}_{1-\alpha}\geq  {c}_{1-\alpha}$. Therefore, the $1-\alpha$ confidence set for $\bm{\tau}$ is asymptotically conservative.
	When 
	$\bm{S}_{\bm{\tau\tau}}^\myperp=o(1)$, 
	we have ${\bm{V}}_{\bm{\tau\tau}}^\myperp - \tilde{\bm{V}}_{\bm{\tau\tau}}^\myperp = o(1)$, which implies $ {\mathcal{L}} \apprsim 
	\tilde{\mathcal{L}}$. Thus, $\tilde{c}_{1-\alpha}=  {c}_{1-\alpha}+o(1)$, and the $1-\alpha$ confidence set for $\bm{\tau}$  becomes asymptotically exact. 
\end{proof}

Note that both 
$
\sum_{h=1}^H 
\bm{W}_{\bm{\tau x}}[h]
(
\bm{W}_{\bm{xx}}[h]
)^{-1/2}
\bm{\zeta}_{LF_h, a_h}
$
and 
$
\sum_{j=1}^{J}
\bm{U}_{\bm{\tau e}}[j]
(
\bm{U}_{\bm{ee}}[j]
)^{-1/2}
\bm{\zeta}_{\lambda_j,a_j}
$
are central convex unimodal. The proofs for 
the asymptotic conservativeness of symmetric convex confidence sets 
of form $
\bm{C}\hat{\bm{\tau}} + \mathcal{O}(\bm{C}\hat{\bm{V}}_{\bm{\tau\tau}}^\myperp\bm{C}', c)
$
for $\bm{C\tau}$
under $\text{ReFMT}_\text{F}$ and $\text{ReFMT}_\text{CF}$ are almost the same as ReFM. 
Thus we omit the proofs of Theorems \ref{thm:conserv_cs_remft_f} and \ref{thm:conserv_cs_remft_cf}. 

Moreover, under for ReFM, $\text{ReFMT}_\text{F}$, or $\text{ReFMT}_\text{CF}$, 
we consider $1-\alpha$ confidence set $\bm{C}\hat{\bm{\tau}} + \tilde{\mathcal{O}}$ for $\bm{C\tau}$, where $\tilde{\mathcal{O}}$ can depend on  
$(\hat{\bm{V}}_{\bm{\tau\tau}}^\myperp, \hat{\bm{V}}_{\bm{\tau x}}, \bm{V}_{\bm{xx}})$ and 
satisfies that $P( \hat{\mathcal{L}} \in \tilde{\mathcal{O}}) = 1-\alpha$. 
If $\tilde{\mathcal{O}}$ is a symmetric convex set, then the confidence set is generally asymptotically conservative, and the proof is similar to the proof of Theorem \ref{thm:conser_conf_set_refm}.

\end{document}